\documentclass[a4paper,11pt]{amsart}

\usepackage{hyperref}
\hypersetup{backref,colorlinks=true,allcolors=black}
\usepackage[foot]{amsaddr}
\usepackage{mathrsfs}
\usepackage{enumerate}
\usepackage{dsfont}
\usepackage{mathtools}
\usepackage{amssymb}
\usepackage{xcolor}
\usepackage{bm}
\theoremstyle{plain}
\newtheorem{theorem}{Theorem}[section]
\newtheorem{proposition}[theorem]{Proposition}
\newtheorem{lemma}[theorem]{Lemma}
\newtheorem{corollary}[theorem]{Corollary}
\theoremstyle{definition}
\newtheorem{definition}[theorem]{Definition}
\theoremstyle{remark}
\newtheorem{remark}[theorem]{Remark}
\newtheorem{assumption}[theorem]{Assumption}

\newtheorem*{acknow*}{Acknowledgments}

\newcommand{\ignore}[1]{}
\newcommand{\R}{\mathbb{R}}
\newcommand{\N}{\mathbb{N}}

\newcommand{\norm}[1]{\left\lVert #1 \right\rVert}
\newcommand{\abs}[1]{\left\vert #1 \right\rvert}

\newcommand{\E}{\mathbb{E}}
\newcommand{\Var}{\operatorname{Var}}
\newcommand{\Law}{\mathcal{L}}
\newcommand{\Ptwo}{\mathcal{P}_2\big(\R^d\big)}

\allowdisplaybreaks[3]

\author{Solesne Bourguin}
\email[Solesne Bourguin]{bourguin@math.bu.edu}
\thanks{S.B. was partially supported by the Simons Foundation Award
  635136.}
\author{Konstantinos Spiliopoulos}
\address{Boston University, Department of Mathematics and Statistics\\ 665 Commonwealth Ave, Boston, MA 02215, USA}
\email[Konstantinos Spiliopoulos]{kspiliop@math.bu.edu}
\thanks{K.S. was partially supported by NSF DMS-2311500}

\title[Uniform-in-time quantitative
  fluctuations of large scale IPS]{Uniform-in-time quantitative fluctuations of large scale
interacting particle systems}
\begin{document}
\begin{abstract} We study fluctuations of mean-field interacting particle systems around
their McKean--Vlasov limit. Our main result provides a uniform-in-time
quantitative central limit theorem for the fluctuation process, with
convergence rate of order $N^{-1/2}$ to the corresponding Gaussian
limit in the Wasserstein metric. The proof relies on two main
ingredients. First, we establish a uniform-in-time weak expansion for
specific functionals of the empirical measure around their limiting
behavior. This yields, in particular, uniform-in-time control of the
convergence of the prelimit variance to its limiting counterpart. We
also derive a backward PDE representation of the limiting variance,
which is of independent interest. Second, we use Malliavin calculus
tools and, in particular, a second-order Poincar\'e inequality that
bounds the Wasserstein distance between the fluctuation process and its
Gaussian limit in terms of the first- and second-order Malliavin
derivatives of the particle flow. The quantitative convergence rates
then follow from a delicate analysis of these derivatives, yielding
the sharp
estimates required for uniform-in-time control. 
\end{abstract}	
\bibliographystyle{amsalpha}

	\maketitle

        \section{Introduction}\label{S:Introduction}

The objective of this paper is to obtain uniform-in-time quantitative
control of fluctuations for mean-field interacting particle systems
around their McKean--Vlasov limit. More specifically, the aim is to
understand not only the limiting behavior of the empirical measure as
the number of particles tends to infinity, but also the size and
structure of the fluctuations around this limit, and to do so in a way
that remains stable over long time horizons. More precisely, fix $T > 0$, let
$\left( B^1, \dots, B^N \right)$ be independent $\R^m$-valued Brownian
motions, and let $\left( X_0^{i,N} \right)_{i=1}^N$ be i.i.d.\
$\R^d$-valued random variables with common law $\nu$, independent of
the Brownian motions. The mean-field interacting particle system under
consideration is given by
\begin{equation}
\label{eq:IPS}
dX_t^{i,N}
=
b\left( X_t^{i,N}, \mu_t^N \right) dt
+
\sigma\left( X_t^{i,N}, \mu_t^N \right) dB_t^i,
\qquad
X_0^{i,N} \sim \nu,
\quad
i = 1, \dots, N,
\end{equation}
where
\begin{equation*}
\mu_t^N
=
\frac{1}{N}\sum_{k=1}^N \delta_{X_t^{k,N}}
\end{equation*}
is the empirical measure associated with the particle configuration at
time $t$. The asymptotic behavior of systems of the form \eqref{eq:IPS} has been
studied extensively, both for its intrinsic mathematical interest and
for its many applications. On finite time intervals, that is, when
$T < \infty$, law-of-large-numbers results describe the limiting
behavior of the empirical measure process
$\left\{ \mu_t^N, t \in [0, T] \right\}_{N \in \mathbb{N}}$. A far from
exhaustive list of classical references on propagation of chaos for
systems of the form \eqref{eq:IPS} is
\cite{Oel84,G88,Szn91}. These works show that the sequence
$\left\{ \mu_t^N, t \in [0, T] \right\}_{N \in \mathbb{N}}$ converges,
in the space of probability measures on
$C\left( [0, T]; \R^d \right)$, to a limiting measure $\mu_t$,
characterized as the law of the McKean--Vlasov process $X_t$ solving
\begin{equation}
\label{eq:IPSlimit}
dX_t
=
b\left( X_t, \mu_t \right) dt
+
\sigma\left( X_t, \mu_t \right) dB_t, \qquad \mu_{t}=\Law\left( X_t \right), \qquad
X_0 \sim \nu.
\end{equation}
Systems of the form \eqref{eq:IPS} arise in a wide range of contexts,
including engineering, finance, machine learning, economics, and
biology -- see, for example,
\cite{BezemekSpiliopoulos1,CDreview,DGPS23,IMS22,SirSpi20a,Spi15}.

One important refinement of the law-of-large-numbers theory is
uniform-in-time propagation of chaos. In that setting, one asks whether
the convergence of the law of the particle system \eqref{eq:IPS} to the
law of the McKean--Vlasov limit \eqref{eq:IPSlimit} holds not only on a
fixed interval $\left[ 0, T \right]$, but on the whole half-line
$\left[ 0, \infty \right)$. This question is especially natural in
dissipative regimes, where the limiting McKean--Vlasov dynamics is
expected to remain stable at large times. A number of works have
successfully addressed this problem -- see
\cite{Malrieu2001,DEGZ18,LackerLeFlem2023,SchSou24}.

Complementary to this literature is the fluctuation analysis of such
systems. While the law of large numbers describes the first-order
behavior of the empirical measure, it does not capture the stochastic
error that remains at the next order. One is therefore naturally led to
study the asymptotic behavior, as $N \to \infty$, of the signed
fluctuation measure
$\sqrt{N}\left( \mu_t^N - \mu_t \right)$. On compact time intervals,
fluctuation central limit theorems go back, in particular, to the work
of Dawson \cite{Dawson83}, and were further developed for
McKean--Vlasov systems in the classical work \cite{FM1997}. Subsequent
related studies have treated a variety of application domains -- see, for
example, \cite{DGPS23,HIMS24,SirSpi20b,Spi15}.

The long-time fluctuation problem contains an additional subtlety. Two
limiting procedures are involved: one may first let $N \to \infty$ at
fixed time and then study the behavior of the limiting fluctuation
process as $t \to \infty$, or one may first analyze the long-time
behavior of the finite particle system and only afterwards let
$N \to \infty$. In general, these two procedures need not be
compatible. This is already visible in Dawson's work \cite{Dawson83}
where it is proved that away from the critical regime, the usual central limit scaling leads to
Gaussian fluctuations, whereas at the critical point the fluctuations
are non-Gaussian and occur on a different scale. Related obstructions
also appear in the recent work \cite{DGPS23}, where the absence of
phase transitions, nondegeneracy of functional inequalities,
uniform-in-time propagation of chaos, and Gaussianity of equilibrium
fluctuations are shown to be closely related. Thus, one should not necessarily
expect uniform-in-time Gaussian fluctuation bounds in regimes where the
McKean--Vlasov dynamics admits several competing long-time states, or
where the linearized dynamics around equilibrium fails to have a
spectral gap.

The present paper is carried out precisely in such a stable regime where the limiting fluctuations are expected to be Gaussian. For
a smooth test function $\varphi \in C_b^7\left( \R^d \right)$, we study
the fluctuations of the empirical average
$\left\langle \mu_t^N, \varphi \right\rangle$ around its
McKean--Vlasov limit $\left\langle \mu_t, \varphi \right\rangle$, and
we seek quantitative bounds that hold uniformly in time. To this end,
we introduce the fluctuation process
\begin{equation*}
G_t^N\left( \varphi \right)
=
\sqrt{N}
\left(
\left\langle \mu_t^N, \varphi \right\rangle
-
\left\langle \mu_t, \varphi \right\rangle
\right).
\end{equation*}
Our main result gives concrete conditions under which the convergence of
$G_t^N\left( \varphi \right)$ to its Gaussian limit can be quantified
uniformly in time, with an explicit rate in the Wasserstein metric.

The assumptions of the paper for the general model (\ref{eq:IPS}) exclude the critical or
multi-phase behavior described above. More precisely, the dissipativity
and smallness conditions in Assumption~\ref{ass:coeff-level-new} imply
exponential decay of the derivatives of the decoupled McKean--Vlasov
flow, and this decay is isolated later in the more intrinsic flow-level
Assumption~\ref{ass:flow-level-new-abstract}. This is the analytic
mechanism that makes the large-particle and long-time limits commute in
our setting: the linearized McKean--Vlasov dynamics remains stable, the
weak expansion coefficients are uniformly controlled, and the
Malliavin-derivative estimates decay in time. The uniform Gaussian
approximation proved below should therefore be understood as a result in
a stable single-phase regime.

At a high level, the contribution of this paper is two-fold. First, we
quantitatively characterize the convergence of fluctuations for the
particle system \eqref{eq:IPS}, with precise convergence rates that hold
uniformly in time in the Wasserstein metric. Second, in order to
establish this result, we exploit a deep connection with Malliavin
calculus
\cite{nualart_malliavin_2006}, in a form
that is sufficiently robust to be useful more generally in settings
where the limiting fluctuations are of Gaussian nature.

At a more concrete mathematical level, the proof of the uniform-in-time
bounds rests on a number of carefully derived approximations and
estimates. The first main ingredient is a uniform-in-time weak expansion
for the specific functionals of the empirical measure that arise in the
fluctuation analysis. For this, we build on the fixed-time weak
expansion results of \cite{CST2022}, and show how to control the
relevant terms uniformly in time. This yields, in particular,
uniform-in-time convergence of the prelimit variance to its limiting
counterpart, and also leads to a backward PDE representation of the
limiting variance; see Sections \ref{sec:uniform-time} and
\ref{limiting-variance-section}. The second main ingredient is a
quantitative Gaussian approximation based on Malliavin calculus. More
precisely, we use a second-order Poincar\'e inequality
\cite{nourdin_second_2009,Vidotto2020}, which
bounds the Wasserstein distance to the Gaussian limit in terms of the
first- and second-order Malliavin derivatives of the particle system.
The key point is to establish estimates on these derivatives that decay
exponentially in time. This is done in the technical Section \ref{sec:malliavin}, and it is precisely this decay that
ultimately yields the desired uniform-in-time control.

An important motivating class of examples is the classical one-body and
two-body potential model with constant diffusion coefficient,
\begin{equation}\label{eq:convex-potential-model}
b\left( x,\mu \right)
=
-\nabla U\left( x \right)
-
\int_{\R^d} \nabla W\left( x-y \right)\mu\left( dy \right),
\qquad
\sigma\left( x,\mu \right) = \Sigma.
\end{equation}
This model plays a central role in
the literature on uniform-in-time propagation of chaos; see, for
example, \cite{Dawson83,DEGZ18,DGPS23,LackerLeFlem2023,Malrieu2001}. In the case of
\eqref{eq:convex-potential-model}, one typically assumes that the confining potential
$U$ is convex and that $W$ is convex and even. In particular, this means that the drift
typically has linear growth and therefore falls outside the bounded-coefficient framework in which the
general part of the present paper is developed. 

This limitation is not specific to the present work, but reflects a
broader feature of the subject. Results formulated for general
coefficients $b$ and $\sigma$ are typically proved under boundedness
assumptions, or under closely related uniform regularity conditions,
whereas analyses devoted to a specific model typically exploit the
additional structure of that model in an essential way in order to go
beyond the general bounded framework. Thus, one cannot in general
expect a theorem proved in full generality to apply automatically to a specific model simply by specializing the coefficients.

Nevertheless, as the arguments of the present paper make clear, the main ingredients
needed to treat the model \eqref{eq:convex-potential-model} are 
a second-order weak expansion adapted to the relevant empirical-measure
functionals, and long-time decay estimates for the derivatives of the
flow. The unbounded convex-potential regime is left for future research, because
its particular structure requires weighted regularity and weak-expansion estimates that
are best treated in a model-specific analysis rather than folded into a
general theorem. Investigation of this question is under way.

The rest of the paper is organized as follows. In
Section~\ref{S:MainResults}, we present the precise formulation of the
problem, the assumptions, and the statement of the main result,
Theorem~\ref{thm:mainCLT}. Section~\ref{S:ProofMainTheorem} gives the
structural proof of Theorem~\ref{thm:mainCLT}, assuming the main
technical ingredients proved later in the paper.  The technical ingredients are then established in Sections
\ref{sec:uniform-time}, \ref{limiting-variance-section}, and
\ref{sec:malliavin}. In Section~\ref{sec:uniform-time},
we prove the uniform-in-time weak expansion of the empirical measure
around its limit in Proposition
\ref{prop:uniform-time-weak-expansion-abs-assump}. In
Section~\ref{limiting-variance-section}, we analyze the variance of the
limiting fluctuations and derive its backward PDE representation, which
we believe to be of independent interest. In
Section~\ref{sec:malliavin}, we prove the uniform-in-time
bounds on the first- and second-order Malliavin derivatives stated in
Propositions~\ref{prop:firstbounds-uniform} and
\ref{prop:secondbounds-uniform}.

\section{Formulation, Assumptions and Main Results}\label{S:MainResults}

We now formulate the assumptions and state the main result. We first fix the
Wasserstein notation used throughout the paper. For probability measures
$\mu,\eta \in \mathcal P_2\left(\R^d\right)$, we denote by
$W_2\left(\mu,\eta\right)$ the usual $2$-Wasserstein distance on
$\mathcal P_2\left(\R^d\right)$. If $Y$ and $Z$ are real-valued integrable
random variables, we write
\begin{equation*}
W_1\left(Y,Z\right)
=
W_1\left(\mathcal L\left(Y\right),\mathcal L\left(Z\right)\right)
=
\sup_{\operatorname{Lip}\left(h\right)\le 1}
\left|
\E\left[h\left(Y\right)\right]
-
\E\left[h\left(Z\right)\right]
\right|.
\end{equation*}
Thus, in the statement of the main theorem, $W_1$ always denotes the
$1$-Wasserstein distance between the laws of the corresponding real-valued
random variables.

Fix $\varphi \in C_b^7\left( \R^d \right)$. We study the fluctuations of the
empirical average $\left\langle \mu_t^N,\varphi \right\rangle$ around its
McKean--Vlasov limit $\left\langle \mu_t,\varphi \right\rangle$. The corresponding
fluctuation process is defined to be
\begin{equation*}
G_t^N\left( \varphi \right)
=
\sqrt{N}
\left(
\left\langle \mu_t^N,\varphi \right\rangle
-
\left\langle \mu_t,\varphi \right\rangle
\right).
\end{equation*}
Under the assumptions below, these fluctuations are asymptotically Gaussian. More
precisely, we identify a limiting variance $\sigma_t^2\left(\varphi\right)$ and prove
that $G_t^N\left(\varphi\right)$ converges to
$\mathcal N\left(0,\sigma_t^2\left(\varphi\right)\right)$ in $W_1$, with rate
$N^{-1/2}$ uniformly over $t\ge 0$.

\subsection{Measure derivatives and regularity classes}
We recall the notation for the measure derivatives and regularity classes used
throughout the paper. All derivatives with respect to the measure variable are
understood in the sense of Lions. If
$f \colon \R^d \times \Ptwo \to E$, where $E$ is a finite-dimensional Euclidean
space, we write
\begin{equation*}
\partial_\mu^n f\left(x,\mu,v_1,\dots,v_n\right)
\end{equation*}
for its $n$-th Lions derivative with respect to the measure variable, evaluated at
$v_1,\dots,v_n \in \R^d$. We also use derivatives with respect to the spatial
variables $v_1,\dots,v_n$ appearing in the Lions derivatives. Thus, for
$n,\ell \in \N_0$ and $\beta=\left(\beta_1,\dots,\beta_n\right)\in\N_0^n$, we set
\begin{equation*}
D^{\left(n,\ell,\beta\right)}
f\left(x,\mu,v_1,\dots,v_n\right)
=
\partial_{v_n}^{\beta_n}\cdots
\partial_{v_1}^{\beta_1}
\partial_x^\ell
\partial_\mu^n
f\left(x,\mu,v_1,\dots,v_n\right),
\end{equation*}
whenever this derivative is well-defined. Its order is $\left|\left(n,\ell,\beta\right)\right|
=
n+\ell+\beta_1+\cdots+\beta_n$. If $\Phi \colon \Ptwo \to \R$ depends only on the measure variable, we use the
analogous notation
\begin{equation*}
D^{\left(n,\beta\right)}
\Phi\left(\mu,v_1,\dots,v_n\right)
=
\partial_{v_n}^{\beta_n}\cdots
\partial_{v_1}^{\beta_1}
\partial_\mu^n
\Phi\left(\mu,v_1,\dots,v_n\right),
\end{equation*}
with $\left|\left(n,\beta\right)\right|
=
n+\beta_1+\cdots+\beta_n$. For $k\in\N$, we say that
$f\in C_{b,\mathrm{Lip}}^{k,k}\left(\R^d\times\Ptwo;E\right)$ if, for every
multi-index $\left(n,\ell,\beta\right)$ such that
$\left|\left(n,\ell,\beta\right)\right|\le k$, the derivative
$D^{\left(n,\ell,\beta\right)}f$ exists, is bounded, and is globally Lipschitz in
all its variables. More precisely, there exists a constant $C>0$ such that, for all
$x,x',v_1,v_1',\dots,v_n,v_n'\in\R^d$ and all $\mu,\mu'\in\Ptwo$,
\begin{align*}
&\norm{
D^{\left(n,\ell,\beta\right)}
f\left(x,\mu,v_1,\dots,v_n\right)
-
D^{\left(n,\ell,\beta\right)}
f\left(x',\mu',v_1',\dots,v_n'\right)
  }\\
  &\qquad\qquad\qquad\qquad\qquad\qquad\qquad\qquad\qquad
  \le
C
\left(
\abs{x-x'}
+
W_2\left(\mu,\mu'\right)
+
\sum_{r=1}^n \abs{v_r-v_r'}
\right),
\end{align*}
and
\begin{equation*}
\sup_{x,\mu,v_1,\dots,v_n}
\left\|
D^{\left(n,\ell,\beta\right)}
f\left(x,\mu,v_1,\dots,v_n\right)
\right\|
\le C.
\end{equation*}
When the target space $E$ is clear from the context, we simply write
$C_{b,\mathrm{Lip}}^{k,k}\left(\R^d\times\Ptwo\right)$.

We shall also use the regularity classes $M^k$ appearing in \cite{CST2022}. For
$k\in\N$, we say that $f\in M^k\left(\R^d\times\Ptwo;E\right)$ if the derivatives
$D^{\left(n,\ell,\beta\right)}f$ exist for all
$\left|\left(n,\ell,\beta\right)\right|\le k$ and satisfy the boundedness and
global Lipschitz estimates above. We denote by
$\norm{f}_{M^k\left(\R^d\times\Ptwo\right)}$ the smallest admissible constant in
these bounds. If $\Phi\colon\Ptwo\to\R$, we say that
$\Phi\in M^k\left(\Ptwo\right)$ if, viewing $\Phi$ as a function on
$\R^d\times\Ptwo$ independent of the first variable, the corresponding derivatives
$D^{\left(n,\beta\right)}\Phi$ of order at most $k$ are bounded and globally
Lipschitz. The corresponding norm is denoted by
$\norm{\Phi}_{M^k\left(\Ptwo\right)}$.
\begin{remark}
  The two notations are kept for clarity: $C_{b,\mathrm{Lip}}^{k,k}$ is used for
coefficient-level assumptions, following the convention of \cite{CrisanMcMurray2018},
whereas $M^k$ is used for the measure functionals appearing in the weak expansion
argument of \cite{CST2022}.
\end{remark}

\subsection{Assumptions}

We are now ready to state the assumptions under which the uniform-in-time
quantitative central limit theorem holds.
\begin{assumption}[Global Lipschitz well-posedness]
\label{ass:smooth}
The initial law satisfies $\nu \in \Ptwo$. Moreover, the coefficients
\begin{equation*}
b \colon \R^d \times \mathcal P_2\left( \R^d \right) \to \R^d\quad\mbox{and}\quad
\sigma \colon \R^d \times \mathcal P_2\left( \R^d \right) \to \R^{d \times m}
\end{equation*}
are globally Lipschitz, that is, there exists a constant $ L > 0 $ such that, for all
$ x,y \in \R^d $ and all $ \mu,\eta \in \mathcal P_2\left( \R^d \right) $,
\begin{equation*}
\abs{b\left( x,\mu \right) - b\left( y,\eta \right)}
+
\norm{\sigma\left( x,\mu \right) - \sigma\left( y,\eta \right)}
\le
L \left( \abs{x-y} + W_2\left( \mu,\eta \right) \right).
\end{equation*}
\end{assumption}

\begin{remark}
Under Assumption~\ref{ass:smooth}, for every $ N \ge 1 $, the interacting particle
system \eqref{eq:IPS} is well posed on $ \left[ 0,T \right] $. Moreover, the
McKean--Vlasov SDE
\begin{equation}
  \label{eq:MV}
dX_s
=
b\left( X_s,\mu_s \right) ds
+
\sigma\left( X_s,\mu_s \right) dB_s,
\qquad
\mu_s = \Law\left( X_s \right),
\qquad
X_0 \sim \nu,
\end{equation}
admits a unique strong solution on $ \left[ 0,T \right] $, and the flow
$ \left( \mu_s \right)_{s \in \left[ 0,T \right]} $ is deterministic. In what follows, we will write $a\left(x,\mu\right)=\sigma\left(x,\mu\right)\sigma\left(x,\mu\right)^\top\in\R^{d\times d}$.
\end{remark}

\begin{assumption}[Coefficient regularity]
\label{ass:second-malliavin-uniform}
Assume the following.
\begin{enumerate}
\item[(i)] The first-order derivatives $\partial_x b$, $\partial_x
  \sigma$, $\partial_\mu b$ and $\partial_\mu \sigma$ exist, are jointly continuous in all variables, and are uniformly bounded on
$ \R^d \times \Ptwo $.

\item[(ii)] The second-order derivatives $\partial_{xx}^2 b$,
$\partial_\mu\!\left[ \partial_x b \right]$,
$\partial_x\!\left[ \partial_\mu b \right]$,
$\partial_v\!\left[ \partial_\mu b \right]$,
$\partial_{\mu\mu}^2 b$, and the analogous derivatives of $\sigma$ exist, are jointly continuous in all
variables, and are uniformly bounded. We denote
\begin{align*}
M_2
&=
\sup
\Bigg(
\norm{\partial_{xx}^2 b}
+
\norm{\partial_\mu\!\left[ \partial_x b \right]}
+
\norm{\partial_x\!\left[ \partial_\mu b \right]}
+
\norm{\partial_v\!\left[ \partial_\mu b \right]}
+
\norm{\partial_{\mu\mu}^2 b}
  \Bigg)\\
  &\quad
    +
\sup
\Bigg(
\norm{\partial_{xx}^2 \sigma}
+
\norm{\partial_\mu\!\left[ \partial_x \sigma \right]}
+
\norm{\partial_x\!\left[ \partial_\mu \sigma \right]}
+
\norm{\partial_v\!\left[ \partial_\mu \sigma \right]}
+
\norm{\partial_{\mu\mu}^2 \sigma}
\Bigg)
<
\infty.
\end{align*}
\end{enumerate}
\end{assumption}

For $\mu \in \Ptwo$, let $\left(\mu_t^\mu\right)_{t\ge 0}$ denote the
McKean--Vlasov law flow started from $\mu$. For $x\in \R^d$, we define the
decoupled flow associated with this law flow by
\begin{equation}
  \label{eq:def-decoupl-mcvlasov-flow}
X_t^{x,\mu}
=
x
+
\int_0^t b\left( X_s^{x,\mu},\mu_s^\mu \right) ds
+
\int_0^t \sigma\left( X_s^{x,\mu},\mu_s^\mu \right) dB_s .
\end{equation}

\begin{assumption}[Flow-level regularity]
\label{ass:flow-smothness}
Let $\left( X_t^{x,\mu} \right)_{t \ge 0}$ be the decoupled flow
defined in \eqref{eq:def-decoupl-mcvlasov-flow}. We assume that
for all multi-indices $\gamma$, $\beta$ and every integer $n \ge 0$
such that $|\gamma| + |\beta| + n \le 7$, the mixed derivative $\partial_x^\gamma \partial_v^\beta \partial_\mu^n
X_t^{x,\mu}\left( v_1,\dots,v_n \right)$ exists for every $t \ge 0$, admits a jointly continuous version in
$\left( t,x,\mu,v_1,\dots,v_n \right)$, and belongs to
$ L^p\left( \Omega \right) $ for every $p \ge 1$.
\end{assumption}

The preceding assumptions provide the regularity needed to differentiate the
McKean--Vlasov flow and to apply the weak expansion and Malliavin calculus
arguments. For the estimates to be uniform in time, however, regularity alone is not
enough. We also need a long-time stability mechanism ensuring that perturbations of
the initial condition and of the initial law are damped as time evolves. The following
dissipativity assumption is a coefficient-level condition designed to provide this
stability: it yields exponential decay of the relevant derivatives of the decoupled
McKean--Vlasov flow, which is the key input for the uniform-in-time estimates proved
later.
\begin{assumption}[Dissipative regime]
\label{ass:coeff-level-new}
Assume that
$b,\sigma \in C_{b,\mathrm{Lip}}^{7,7}\left( \R^d \times \Ptwo
\right)$. Define the set $\mathcal P = \left\{2,4,6,8,10,12,14
\right\}$ and assume the following.
\begin{enumerate}
\item[(i)] For every $ p \in \mathcal P $, there exists $ \kappa_p > 0 $ such that, for all
$ x,y \in \R^d $ and all $ \mu \in \Ptwo $,
\begin{equation}
\label{eq:uniform-time-monotonicity-p}
2 \left\langle x-y, b\left( x,\mu \right) - b\left( y,\mu \right) \right\rangle
+
\left( p-1 \right)\norm{\sigma\left( x,\mu \right) - \sigma\left( y,\mu \right)}^2
\le
-\kappa_p \abs{x-y}^2.
\end{equation}

\item[(ii)] There exists $ \gamma \in \left( 0,1 \right] $ such that
\begin{equation}
\label{eq:small-measure-dependence}
\sup_{x,\mu,y}
\left(
\abs{\partial_\mu b\left( x,\mu \right)\left( y \right)}
+
\norm{\partial_\mu \sigma\left( x,\mu \right)\left( y \right)}
\right)
\le \gamma.
\end{equation}

\item[(iii)] Writing $M_\sigma = \sup_{x,\mu} \norm{\partial_x
    \sigma\left( x,\mu \right)}$, assume that
\begin{equation}
\label{eq:def-omega-uniform-time}
\omega
=
\min_{p \in \mathcal P}
\left(
\frac{\kappa_p}{2}
-
2^{\frac{p-1}{p}} \gamma
-
\left( p-1 \right) 2^{\frac{p-1}{p}} M_\sigma \gamma
-
\frac{p-1}{2} 2^{\frac{2\left( p-1 \right)}{p}} \gamma^2
\right)
> 0.
\end{equation}
\end{enumerate}
\end{assumption}

The preceding dissipativity assumption gives the exponential stability of the
McKean--Vlasov flow at the level of the limiting dynamics. For the Malliavin
calculus argument, we also need this stability to propagate to the finite particle
system, uniformly in $N$. The next assumption is a quantitative smallness condition
which allows us to close the estimates for the first and second-order Malliavin
derivatives of the particle flow. It will be used in
Propositions~\ref{prop:firstbounds-uniform}
and~\ref{prop:secondbounds-uniform}.
\begin{assumption}[Particle-flow Malliavin stability]
\label{ass:first-malliavin-uniform}
The diffusion coefficient $\sigma$ is bounded. Let $\gamma$ and $M_\sigma$ be the
constants appearing in Assumption~\ref{ass:coeff-level-new}. For
$p\in\left\{4,8\right\}$, let $\kappa_p>0$ be the corresponding constant from
Assumption~\ref{ass:coeff-level-new}(i), and assume
\begin{equation*}
2^{1/p}\Lambda_{p,\kappa_p,M_\sigma}\gamma
\left(
\frac{8}{p\kappa_p}
\right)^{1/p} < 1,
\end{equation*}
where
\begin{equation*}
\Lambda_{p,\kappa_p,M_\sigma}
=
\left(
\left( \frac{12 \left( p-1 \right)}{\kappa_p} \right)^{p-1}
\left(
1 + \left( \left( p-1 \right) M_\sigma \right)^p
\right)
+
\Xi_{p,\kappa_p}
\right)^{1/p},
\end{equation*}
and
\begin{equation*}
\Xi_{p,\kappa_p}
=
\begin{cases}
\displaystyle
\left( p-1 \right)
\left(
\frac{6 \left( p-1 \right)\left( p-2 \right)}{\kappa_p}
\right)^{\left( p-2 \right)/2}
& \mbox{if } p > 2\\[1.2ex]
1
& \mbox{if } p = 2
\end{cases}.
\end{equation*}
\end{assumption}

\subsection{Main result}

The main result of the paper is Theorem~\ref{thm:mainCLT}, which gives a
uniform-in-time quantitative control of the convergence of the fluctuations. Before
stating it, we define the limiting variance
$\sigma_t^2\left(\varphi\right)$. This variance is expressed in terms of the
coefficients appearing in the asymptotic expansions developed later in the paper.
In Subsection
\ref{S:BackwardPDErepresentationVar} we also present an alternative
backward PDE representation of $\sigma_t^2\left(\varphi\right)$, which
is of independent interest.
\begin{definition}[Limiting fluctuation variance]
\label{def:uniform-time-limiting-variance-abs-assump}
Let $\varphi \in C_b^7\left( \R^d \right)$. We define the limiting fluctuation
variance by
\begin{equation*}
\sigma_t^2\left( \varphi \right)
=
\alpha_2\left( t \right)
-
2 \left\langle \mu_t,\varphi \right\rangle \alpha_1\left( t \right),
\qquad t \ge 0,
\end{equation*}
where $\alpha_1$ and $\alpha_2$ are the functions introduced in
Corollary~\ref{cor:uniform-time-alpha-abs-assump}. The explicit formulas for
$\alpha_1$ and $\alpha_2$ are postponed to Section~\ref{sec:uniform-time}, since
their definition requires several intermediate quantities from the weak expansion
argument.
\end{definition}

Proposition~\ref{P:BackwardRepPDEVariance} in
Subsection~\ref{S:BackwardPDErepresentationVar} shows that the quantity
$\sigma_t^2\left(\varphi\right)$ defined above is indeed nonnegative for every
$\varphi \in C_b^7\left( \R^d \right)$ and every $t\ge 0$. In the main theorem,
we impose the additional nondegeneracy condition that this variance is bounded
away from zero uniformly in time. Under this assumption, the next result gives a
uniform-in-time quantitative control of the convergence of the fluctuations in the $W_1$ metric.

\begin{theorem}[Uniform-in-time quantitative CLT for the fluctuations]
\label{thm:mainCLT}
Assume that Assumptions~\ref{ass:smooth},~\ref{ass:second-malliavin-uniform},~\ref{ass:flow-smothness}, ~\ref{ass:coeff-level-new}, and~\ref{ass:first-malliavin-uniform} hold.
Let $\varphi\in C_b^7\left(\R^d\right)$ and let
$\sigma_t^2\left(\varphi\right)$ be the limiting fluctuation variance from
Definition~\ref{def:uniform-time-limiting-variance-abs-assump}. Assume that
there exists a constant $\underline \sigma > 0$ such that
\begin{equation}
\label{eq:uniform-positive-variance}
\inf_{t \ge 0}\sigma_t^2\left(\varphi\right)\ge \underline \sigma^2.
\end{equation}
Then, there exists a constant $C>0$ such that for every $N\ge 1$,
\begin{equation}
\label{eq:mainCLT}
\sup_{t \ge 0}
W_1\left(
G_t^N\left(\varphi\right),
\mathcal N\left(0,\sigma_t^2\left(\varphi\right)\right)
\right)
\le
\frac{C}{\sqrt{N}}.
\end{equation}
\end{theorem}

The proof of Theorem~\ref{thm:mainCLT} is carried out in the subsequent
sections. In Section~\ref{S:ProofMainTheorem}, we present the core of
the proof and
show how the desired uniform-in-time convergence control follows from three
ingredients: a uniform-in-time weak expansion, a uniform comparison between the
finite-particle variance and the limiting variance, and uniform-in-time
Malliavin derivative estimates. The remaining sections are devoted to proving these ingredients. In
Section~\ref{sec:uniform-time}, we establish the uniform-in-time weak expansion
of the empirical measure around its McKean--Vlasov limit for the specific test
functionals needed in the variance analysis. In
Section~\ref{limiting-variance-section}, we identify the limiting variance of
the fluctuations, prove the convergence of the finite-particle variances, and
derive a backward PDE representation for the limiting variance, which is of
independent interest. Finally, in
Section~\ref{sec:malliavin}, we prove the uniform-in-time bounds
on the first- and second-order Malliavin derivatives of the particle flow. These
estimates are then combined in Section~\ref{S:ProofMainTheorem} to complete the
proof of Theorem~\ref{thm:mainCLT}.

\section{Proof of the uniform-in-time fluctuation convergence Theorem \ref{thm:mainCLT}}\label{S:ProofMainTheorem}

In this section, we give the proof of Theorem \ref{thm:mainCLT}. The
proof relies on a number of preliminary results, whose proofs are deferred to later sections.  We collect the statements of these preliminary results in this section, and then we present the proof of Theorem \ref{thm:mainCLT}.

\subsection{A three-term decomposition}

Although $ G_t^N\left(\varphi\right)
=
\sqrt{N}
\left(
\left\langle\mu_t^N,\varphi\right\rangle
-
\left\langle\mu_t,\varphi\right\rangle
\right)$ is the natural fluctuation quantity, it is not
centered in general. Since the second-order Poincar\'e inequalities that we will use
later are formulated for centered random variables, we introduce the centered counterpart
\begin{equation*}
F_t^N\left( \varphi \right)
=
\sqrt{N}
\left(
\left\langle \mu_t^N,\varphi \right\rangle
-
\E\left( \left\langle \mu_t^N,\varphi \right\rangle \right)
\right)
=
\frac1{\sqrt{N}}
\sum_{i=1}^N
\left(
\varphi\left( X_t^{i,N} \right)
-
\E\left( \varphi\left( X_t^{i,N} \right) \right)
\right).
\end{equation*}
The two random variables differ only by the deterministic shift
\begin{equation*}
G_t^N\left( \varphi \right)
-
F_t^N\left( \varphi \right)
=
\sqrt{N}
\left(
\E\left( \left\langle \mu_t^N,\varphi \right\rangle \right)-\left\langle \mu_t,\varphi \right\rangle
\right),
\end{equation*}
and therefore they have the same variance, so that
\begin{equation*}
\operatorname{Var}\left( F_t^N\left( \varphi \right) \right)
=
\operatorname{Var}\left( G_t^N\left( \varphi \right) \right)
=
\sigma_{N,t}^2\left( \varphi \right),
\end{equation*}
where the latter equality is understood as the definition of $\sigma_{N,t}^2\left( \varphi \right)$ for $t\geq 0$. By the triangle inequality,
\begin{align}
\label{eq:maintriangle}
W_1\left(
G_t^N\left(\varphi\right),
\mathcal N\left(0,\sigma_t^2\left(\varphi\right)\right)
\right)
&\le
W_1\left(
G_t^N\left(\varphi\right),
F_t^N\left(\varphi\right)
\right)
+
W_1\left(
F_t^N\left(\varphi\right),
\mathcal N\left(0,\sigma_{N,t}^2\left(\varphi\right)\right)
\right)
\nonumber
\\
&\quad
+
W_1\left(
\mathcal N\left(0,\sigma_{N,t}^2\left(\varphi\right)\right),
\mathcal N\left(0,\sigma_t^2\left(\varphi\right)\right)
\right).
\end{align}

\subsection{Main estimates used in the proof}

The proof proceeds by estimating the three terms on the right-hand
side of \eqref{eq:maintriangle} uniformly in $ t \ge 0 $. To estimate the first term in \eqref{eq:maintriangle}, i.e., $W_1\left(
G_t^N\left(\varphi\right),
F_t^N\left(\varphi\right)
\right)$, we prove a uniform-in-time weak expansion for the difference $G_t^N\left(\varphi\right)-F_t^N\left(\varphi\right)
=
\sqrt{N}
\left(
\E\left(\left\langle\mu_t^N,\varphi\right\rangle\right)
-
\left\langle\mu_t,\varphi\right\rangle
\right)$. More precisely, in their paper \cite{CST2022}, the authors prove that
if the coefficients $b$ and $\sigma$ satisfy certain smoothness and
boundedness assumptions,  then a second-order weak expansion for the
interacting particle approximation holds. Concretely, for each sufficiently smooth functional $\Phi \colon \Ptwo\to\R$ (in the Lions sense), one has
\begin{equation*}
\E\left(\Phi\left(\mu_t^N\right)\right)
=
\Phi\left(\mu_t\right)+\frac{1}{N} \mathcal C_t\left(\Phi\right)+\mathcal O\!\left(\frac{1}{N^2}\right),
\end{equation*}
where the remainder is bounded in absolute value by $C_\Phi(t)/N^2$,
uniformly in $N$. In this paper, and in particular in Section
\ref{sec:uniform-time} we upgrade the above fixed-time weak expansion
to a statement that is uniform in time for the functionals needed for the proof of Theorem \ref{thm:mainCLT}. In particular, in Section  \ref{sec:uniform-time} we prove Proposition \ref{prop:uniform-time-weak-expansion-abs-assump}, which is of independent interest. Before presenting this result, we need a bit of notation. Let the semigroup
associated with the McKean--Vlasov flow be $P_t \Phi\left( \mu \right)
= \Phi\left( \mu_t^\mu \right)$, and for any $\Psi\in M^5\big(\Ptwo\big)$, set
\begin{equation}
  \label{defgamma-abs-assump}
\Gamma\Psi\left(\mu\right)
=
\frac12
\int_{\R^d}
\operatorname{Tr}
\left(
a\left(v,\mu\right)\partial_{\mu\mu}^2\Psi\left(\mu\right)\left(v,v\right)
\right)
\mu\left(dv\right),
\end{equation}
where we recall that
$a\left(v,\mu\right)=\sigma\left(v,\mu\right)\sigma\left(v,\mu\right)^\top\in\R^{d\times
  d}$. Since the proof requires uniform control of both the first and second moments of
$\left\langle\mu_t^N,\varphi\right\rangle$, we introduce the corresponding
measure functionals. For $\varphi \in C_b^7\left(\R^d\right)$, set
\begin{equation*}
\Phi_1\left(\mu\right)
=
\left\langle \mu,\varphi \right\rangle\quad\mbox{and}\quad
\Phi_2\left(\mu\right)
=
\left\langle \mu,\varphi \right\rangle^2,
\qquad
\mu\in\Ptwo.
\end{equation*}
We can now state our uniform-in-time weak expansion result.
\begin{proposition}[Uniform-in-time weak expansion]
\label{prop:uniform-time-weak-expansion-abs-assump}
Assume that Assumptions \ref{ass:smooth}, \ref{ass:flow-smothness} and
\ref{ass:coeff-level-new} hold. Let $\varphi \in C_b^7\left(\R^d\right)$ and let
$\Phi_1,\Phi_2\colon\Ptwo\to\R$ be defined above. Then,
there exist measurable functions $\beta_\ell \colon \left[ 0,\infty
\right) \to \R$, $\ell \in \left\{ 1,2 \right\}$, and constants $ C_\ell, C_\ell' > 0 $ such that, for every $ t \ge 0 $ and every $ N \ge 1 $,
\begin{equation}
\label{eq:uniform-time-weak-expansion-abs-assump}
\E\left( \Phi_\ell\left( \mu_t^N \right) \right)
=
\Phi_\ell\left( \mu_t \right)
+
\frac{\beta_\ell\left( t \right)}{N}
+
\frac{1}{N}
\int_0^t \Gamma P_{t-s}\Phi_\ell \left( \mu_s \right) ds
+
R_{\ell,N}\left( t \right),
\end{equation}
with
\begin{equation}
\label{eq:uniform-time-weak-coeff-abs-assump}
\abs{\beta_\ell\left( t \right)} \le C_\ell e^{-\omega t}\quad\mbox{and}\quad
\abs{\Gamma P_r\Phi_\ell \left( \mu \right)} \le C_\ell e^{-\omega r},
\end{equation}
for all $ t,r \ge 0 $ and $ \mu \in \Ptwo $, and with
\begin{equation}
\label{eq:uniform-time-weak-rem-abs-assump}
\sup_{t \ge 0} \abs{R_{\ell,N}\left( t \right)} \le \frac{C_\ell'}{N^2}.
\end{equation}
\end{proposition}

The second term in \eqref{eq:maintriangle} is controlled by a second-order
Poincar\'e inequality. More precisely, we use the bound of \cite{Vidotto2020},
which estimates $W_1\left(
F_t^N\left(\varphi\right),
\mathcal N\left(0,\sigma_{N,t}^2\left(\varphi\right)\right)
\right)$ in terms of the first and second Malliavin derivatives of
$F_t^N\left(\varphi\right)$. The estimates needed to make this bound uniform in
time are proved in Sections~\ref{S:BoundFirstMalliavinDerivatives} and
\ref{S:BoundSecondMalliavinDerivatives}. For the reader's convenience, we next
recall the setting and the statement of the second-order Poincar\'e inequality derived in \cite{Vidotto2020}. Let $(A,\mathcal A,\lambda)$ be a Polish space with $\sigma$-finite non-atomic measure $\lambda$ and let
$\mathfrak{H}=L^2(A,\mathcal A,\lambda)$.
Let $X=\{X(h):h\in\mathfrak{H}\}$ be an isonormal Gaussian process and let $D$ and $D^2$ denote the Malliavin derivatives.
\begin{theorem}
\label{thm:Vidotto}
Let $F\in\mathbb{D}^{2,4}$ be centered with $\E[F^2]=\sigma^2>0$ and
let $Z_\sigma\sim\mathcal{N}(0,\sigma^2)$. Then,
\begin{equation*}
W_1(F,Z_\sigma)\le \sqrt{\frac{8}{\pi\sigma^2}}
\sqrt{
\iint_{A\times A}
\sqrt{\E\left(\left(D^2F\otimes_1 D^2F\right)(x,y)^2\right)}
\sqrt{\E\left(DF(x)^2DF(y)^2\right)}
 \lambda(dx)\lambda(dy)
}.
\end{equation*}
\end{theorem}

In particular, in Section \ref{S:UniformControlVidotto} we establish
the following proposition, which is a key statement needed to control the upper bound in Theorem \ref{thm:Vidotto}. We stress that its proof is based on very precise uniform-in-time control of the first- and second-order Malliavin derivatives of the particle system (proven in Sections \ref{S:BoundFirstMalliavinDerivatives} and \ref{S:BoundSecondMalliavinDerivatives} respectively).
\begin{proposition}
\label{prop:VidottoFunctional}
Assume that Assumptions \ref{ass:smooth},
\ref{ass:second-malliavin-uniform}, \ref{ass:flow-smothness} and \ref{ass:coeff-level-new} hold,  and let
$\varphi\in C_b^2\left(\R^d\right)$.
For $ t \ge 0 $, define
\begin{align*}
\Delta_{N,t}
&=
\iint_{A_t\times A_t}
\sqrt{
\E\left(
\left(
D^2F_t^N\left(\varphi\right)\otimes_1
D^2F_t^N\left(\varphi\right)
\right)\left(x,y\right)^2
\right)
}
\\
&\qquad\qquad\qquad\qquad\qquad\qquad\qquad
\sqrt{
\E\left(
DF_t^N\left(\varphi\right)\left(x\right)^2
DF_t^N\left(\varphi\right)\left(y\right)^2
\right)
}
\lambda_t\left(dx\right)\lambda_t\left(dy\right),
\end{align*}
where
$A_t=\left[0,t\right]\times\left\{1,\dots,N\right\}\times\left\{1,\dots,m\right\}$,
and
\begin{equation*}
\lambda_t\left(ds,dj,d\alpha\right)=ds\sum_{j=1}^N\sum_{\alpha=1}^m
\delta_{\left(j,\alpha\right)}.
\end{equation*}
Then, there exists a constant $ C > 0 $, independent of $ t $ and $ N $, such that
\begin{equation}
\label{eq:GammaBound}
\sup_{t \ge 0} \Delta_{N,t}\le \frac{C}{N}.
\end{equation}
\end{proposition}

For the third term  in \eqref{eq:maintriangle}, i.e.,  $W_1\left(
\mathcal N\left(0,\sigma_{N,t}^2\left(\varphi\right)\right),
\mathcal N\left(0,\sigma_t^2\left(\varphi\right)\right)
\right)$, we essentially need uniform-in-time control of the difference $\abs{
\sigma_{N,t}^2\left(\varphi\right)-\sigma_t^2\left(\varphi\right)
}$. This is the content of the following proposition, whose proof is based on the uniform-in-time weak
expansion of Proposition
\ref{prop:uniform-time-weak-expansion-abs-assump} and is given in Section
\ref{limiting-variance-section}.

\begin{proposition}[Uniform comparison of variances]
\label{prop:uniform-time-variance-comparison-abs-assump}
Assume that Assumptions \ref{ass:smooth}, \ref{ass:flow-smothness} and
\ref{ass:coeff-level-new} hold. Let $\varphi\in C_b^7\left(\R^d\right)$. Then, there
exists $C>0$ such that, for every $N\ge 1$,
\begin{equation*}
\sup_{t\ge 0}
\left|
\sigma_{N,t}^2\left(\varphi\right)
-
\sigma_t^2\left(\varphi\right)
\right|
\le
\frac{C}{N}.
\end{equation*}
\end{proposition}

\subsection{Proof of Theorem~\ref{thm:mainCLT}}

We now have all the ingredients to present the proof of our main
theorem. As discussed above, we make use of the triangle inequality
\eqref{eq:maintriangle} and we estimate the three terms on the right-hand side of \eqref{eq:maintriangle} uniformly in $ t \ge 0 $.
\\~\\
For the first term in \eqref{eq:maintriangle}, the Wasserstein distance is bounded by the $ L^1 $-norm
of the difference. By Proposition \ref{prop:uniform-time-weak-expansion-abs-assump},
\begin{equation*}
\E\left(\left\langle\mu_t^N,\varphi\right\rangle\right)
=
\left\langle\mu_t,\varphi\right\rangle
+
\frac{\alpha_1\left( t \right)}{N}
+
R_{1,N}\left( t \right),
\end{equation*}
with $\sup_{t \ge 0}\abs{\alpha_1\left( t \right)} < \infty$ and
$\sup_{t \ge 0}\abs{R_{1,N}\left( t \right)} \le C/N^2$.
Hence,
\begin{equation}
\label{eq:uniform-mainclt-first-term}
\sup_{t \ge 0}
W_1\left(
G_t^N\left(\varphi\right),
F_t^N\left(\varphi\right)
\right)
\le
\frac{C}{\sqrt{N}}.
\end{equation}

For the second term in \eqref{eq:maintriangle}, the random variable $ F_t^N\left(\varphi\right) $ is centered by
definition, and Lemma~\ref{lem:Fbounds} shows that
$ F_t^N\left(\varphi\right)\in\mathbb D^{2,4} $ for every $ t \ge 0 $.
Moreover, Proposition~\ref{prop:uniform-time-variance-comparison-abs-assump}
implies
\begin{equation*}
\sup_{t \ge 0}
\abs{
\sigma_{N,t}^2\left( \varphi \right) - \sigma_t^2\left( \varphi \right)
}
\le
\frac{C}{N}.
\end{equation*}

Hence, by \eqref{eq:uniform-positive-variance}, there exists $ N_0 \ge
1 $ such that for all $t \geq 0$ and all $N \geq N_0$,
$\sigma_{N,t}^2\left( \varphi \right) \ge \underline \sigma^2/2$. For such $ N $, Theorem~\ref{thm:Vidotto} and
Proposition~\ref{prop:VidottoFunctional} yield
\begin{align}
\label{eq:uniform-mainclt-second-term}
\sup_{t \ge 0}
W_1\left(
F_t^N\left(\varphi\right),
\mathcal N\left(0,\sigma_{N,t}^2\left(\varphi\right)\right)
\right)
&\le
\sup_{t \ge 0}
\sqrt{
\frac{8}{\pi \sigma_{N,t}^2\left(\varphi\right)}
}
\sqrt{\Delta_{N,t}}
\le
C \sup_{t \ge 0}\sqrt{\Delta_{N,t}}
\le
\frac{C}{\sqrt{N}}.
\end{align}

For the finitely many values $ 1 \le N < N_0 $, the same estimate holds after enlarging
the constant.

For the third term in \eqref{eq:maintriangle}, note that centered Gaussian random variables satisfy
\begin{equation*}
W_1\left(
\mathcal N\left(0,\tau_1^2\right),
\mathcal N\left(0,\tau_2^2\right)
\right)
=
\sqrt{\frac{2}{\pi}}\abs{\tau_1-\tau_2}.
\end{equation*}
Applying this identity with
$\tau_1=\sigma_{N,t}\left(\varphi\right)$ and
$\tau_2=\sigma_t\left(\varphi\right)$, and using
\begin{equation*}
\abs{\tau_1-\tau_2}
=
\frac{\abs{\tau_1^2-\tau_2^2}}{\tau_1+\tau_2},
\end{equation*}
we first note that Proposition~\ref{prop:uniform-time-variance-comparison-abs-assump}
and the nondegeneracy assumption \eqref{eq:uniform-positive-variance} imply that,
for all sufficiently large $N$,
\begin{equation*}
\inf_{t\ge 0}\sigma_{N,t}^2\left(\varphi\right)
\ge
\frac{\underline{\sigma}^2}{2}.
\end{equation*}
Consequently, $\sigma_{N,t}\left(\varphi\right)+\sigma_t\left(\varphi\right)$ is bounded
away from zero uniformly in $t\ge 0$ and for all sufficiently large $N$. Therefore,
using Proposition~\ref{prop:uniform-time-variance-comparison-abs-assump} once more,
we obtain 
\begin{align}
\label{eq:uniform-mainclt-third-term}
\sup_{t \ge 0}
W_1\left(
\mathcal N\left(0,\sigma_{N,t}^2\left(\varphi\right)\right),
\mathcal N\left(0,\sigma_t^2\left(\varphi\right)\right)
\right)
&\le
C
\sup_{t \ge 0}
\abs{
\sigma_{N,t}^2\left(\varphi\right)
-
\sigma_t^2\left(\varphi\right)
}
\le
\frac{C}{N}
\le
\frac{C}{\sqrt{N}}.
\end{align}
Combining
\eqref{eq:maintriangle},
\eqref{eq:uniform-mainclt-first-term},
\eqref{eq:uniform-mainclt-second-term}, and
\eqref{eq:uniform-mainclt-third-term}
proves \eqref{eq:mainCLT} for all sufficiently large $ N $.
The finitely many remaining values of $ N $ are absorbed into the
constant.\qed

\section{Uniform-in-time weak expansion of functionals of the particle empirical measure}
\label{sec:uniform-time}

In \cite{CST2022}, the authors prove a fixed-time second-order weak expansion for
interacting particle approximations under suitable smoothness and boundedness assumptions
on the coefficients $b$ and $\sigma$. More precisely, if
$\Phi \colon \Ptwo \to \R$ is sufficiently smooth in the Lions sense, then for each fixed
$t \ge 0$,
\begin{equation*}
\E\left( \Phi\left( \mu_t^N \right) \right)
=
\Phi\left( \mu_t \right)
+
\frac{1}{N}\mathcal C_t\left( \Phi \right)
+
\mathcal O\!\left( \frac{1}{N^2} \right),
\end{equation*}
where the remainder is bounded in absolute value by $C_\Phi(t)/N^2$, uniformly in $N$.
The dependence of the constant on $t$ is harmless at fixed time, but it becomes the main
issue when one seeks estimates that hold uniformly on $\left[ 0,\infty \right)$.

The aim of this section is to show how this fixed-time expansion can be upgraded to a
uniform-in-time statement for the specific functionals needed in the fluctuation analysis.
The natural tool for doing so is the McKean--Vlasov semigroup
\begin{equation*}
P_t\Phi\left( \mu \right) = \Phi\left( \mu_t^\mu \right),
\end{equation*}
which allows us to recast the problem in terms of the propagated functionals
$P_t\Phi_1$ and $P_t\Phi_2$, where $\Phi_1\left( \mu \right) =
\left\langle \mu,\varphi \right\rangle$ and $\Phi_2\left( \mu \right)
= \left\langle \mu,\varphi \right\rangle^2$. Thus, the problem reduces to obtaining regularity and long-time bounds on these
propagated functionals that are strong enough to make the constants in the expansion
uniform in time.

The section is organized accordingly. In
Subsection~\ref{subsec-flow-level}, we study functionals of the decoupled flow
\eqref{eq:def-decoupl-mcvlasov-flow}. In particular, we prove the required fixed-time
regularity in Proposition~\ref{prop:fixed-time-flow-u-abs-assump}, together with the
decay estimates for the first derivatives in
Lemma~\ref{lem:first-derivatives-decay} and for the higher mixed derivatives in
Lemma~\ref{lem:uniform-time-second-variation}. These estimates provide the core
long-time control needed in the sequel. In
Subsection~\ref{S:UniformTimeWeakExpansionFunct}, we then apply these bounds to the
specific functionals $\Phi_1$ and $\Phi_2$ and derive the desired uniform-in-time weak
expansion, proving Proposition~\ref{prop:uniform-time-weak-expansion-abs-assump}.
Finally, in Subsection~\ref{subsec-coeff-level}, we explain the level of generality at
which the argument really operates: the essential input is the decay of the derivatives of
the decoupled flow along the McKean--Vlasov law flow
$\left( \mu_t^\mu \right)_{t \ge 0}$. In this sense,
Assumption~\ref{ass:coeff-level-new} should be viewed only as a concrete sufficient
condition for the more intrinsic Assumption~\ref{ass:flow-level-new-abstract}.

\subsection{Fixed-time regularity of functionals of the decoupled flow.}
\label{subsec-flow-level}
Fix $ x,v \in \R^d $, $ \mu \in \Ptwo $, and $ h \in \R^d $ and recall the decoupled flow associated with the McKean--Vlasov law flow
$ \left( \mu_t^\mu \right)_{t \ge 0} $, as defined via \eqref{eq:def-decoupl-mcvlasov-flow},
\begin{equation*}
X_t^{x,\mu}
=
x
+
\int_0^t b\left( X_s^{x,\mu},\mu_s^\mu \right) ds
+
\int_0^t \sigma\left( X_s^{x,\mu},\mu_s^\mu \right) dB_s.
\end{equation*}
The next proposition establishes the fixed-time regularity of functionals of the decoupled flow.
\begin{proposition}[Fixed-time regularity of functionals of the decoupled flow]
\label{prop:fixed-time-flow-u-abs-assump}
Let Assumption \ref{ass:flow-smothness} hold. Let $ \varphi \in C_b^7\left( \R^d \right) $.
For $ t \ge 0 $, $ x \in \R^d $, and $ \mu \in \Ptwo $, let $ X_t^{x,\mu} $ be the decoupled flow
associated with the McKean--Vlasov equation, and define
\begin{equation*}
u_t\left( x,\mu \right) = \E\left( \varphi\left( X_t^{x,\mu} \right) \right).
\end{equation*}
Then, for every fixed $ t \ge 0 $, all multi-indices $ \gamma $ and $ \beta $, and
  every integer $ n \ge 0 $ with $1 \le \abs{\gamma} + \abs{\beta} + n
  \le 7$, the mixed derivative
\begin{equation*}
\partial_x^\gamma \partial_v^\beta \partial_\mu^n
u_t\left( x,\mu \right)\left( v_1,\dots,v_n \right)
\end{equation*}
exists and is continuous in $\left( x,\mu,v_1,\dots,v_n \right)$. More precisely, every such derivative is a finite sum of terms of the form
\begin{equation}
\label{eq:fixed-time-ut-structure-abs-assump}
\E\left(
\nabla^r \varphi\left( X_t^{x,\mu} \right)
\left[
\partial_x^{\gamma_1} \partial_v^{\beta_1}\partial_\mu^{n_1}
X_t^{x,\mu}\left( \mathbf{v}_1 \right),\dots,
\partial_x^{\gamma_r} \partial_v^{\beta_r}\partial_\mu^{n_r}
X_t^{x,\mu}\left( \mathbf{v}_r \right)
\right]
\right),
\end{equation}
where $ 1 \le r \le 7 $, for each $1 \leq j \leq r$, $\mathbf{v}_j$
denotes an $n_j$-tuple of Lions directions, and each
$\partial_x^{\gamma_j} \partial_v^{\beta_j}\partial_\mu^{n_j}
X_t^{x,\mu}\left( \mathbf{v}_j \right)$
is a mixed derivative of the flow of positive total order $1 \le \abs{\gamma_j} + \abs{\beta_j} + n_j
  \le 7$.
\end{proposition}

\begin{proof}
Fix $ \mu \in \Ptwo $ and only consider differentiation in the state variable $x$.
Since $ \varphi \in C_b^7\left( \R^d \right) $ and the derivatives of
$ X_t^{x,\mu} $ with respect to $x$ up to order $7$ exist by
Assumption \ref{ass:flow-smothness}, repeated applications of the usual
chain rule give, for every multi-index $ \gamma $ with $ 1 \le \abs{\gamma} \le 7 $,
\begin{equation*}
\partial_x^\gamma u_t\left( x,\mu \right)
=
\partial_x^\gamma \E\left( \varphi\left( X_t^{x,\mu} \right) \right).
\end{equation*}
Since all derivatives of $ \varphi $ up to order $ 7 $ are bounded, and all derivatives of the flow
appearing after differentiation have finite moments of every order by
Assumption \ref{ass:flow-smothness}, we may differentiate
under the expectation.
Thus, $ \partial_x^\gamma u_t\left( x,\mu \right) $ exists and is a finite sum of terms of the form
\begin{equation*}
\E\left(
\nabla^r \varphi\left( X_t^{x,\mu} \right)
\left[
\partial_x^{\Lambda_1}
X_t^{x,\mu},\dots,\partial_x^{\Lambda_r}
X_t^{x,\mu}
\right]
\right),
\end{equation*}
where each $\partial_x^{\Lambda_j}
X_t^{x,\mu}$ is a state derivative of the flow of positive order
$\abs{\Lambda_j} \leq \abs{\gamma} $.
This is exactly of the form
\eqref{eq:fixed-time-ut-structure-abs-assump}.
\\~\\
We now treat derivatives with respect to the measure variable.
Let $ \xi \in L^2\left( \Omega;\R^d \right) $ have law $ \mu $, and define the lift
\begin{equation*}
\widetilde u_t\left( x,\xi \right)
=
\E\left( \varphi\left( X_t^{x,\left[ \xi \right]} \right) \right).
\end{equation*}
We first compute its Fr\'echet derivative in the $\xi$-variable. By \cite[Proposition~4.2 and Remark~4.4]{BuckdahnLiPengRainer2017}, for every
$ \eta \in L^2\left( \Omega;\R^d \right) $,
\begin{equation}
\label{eq:fixed-time-flow-lift-derivative-abs-assump}
D_\xi X_t^{x,\left[ \xi \right]}\left[ \eta \right]
=
\widetilde{\E}\left(
\partial_\mu X_t^{x,\mu}\left( \widetilde \xi \right)\widetilde \eta
\right),
\end{equation}
where $\mu = \left[ \xi \right]$. Since $ \varphi \in C_b^7\left( \R^d \right) $, the Banach-space chain rule yields
\begin{equation*}
D_\xi \widetilde u_t\left( x,\xi \right)\left[ \eta \right]
=
\E\left(
\nabla \varphi\left( X_t^{x,\mu} \right)
D_\xi X_t^{x,\left[ \xi \right]}\left[ \eta \right]
\right).
\end{equation*}
Substituting \eqref{eq:fixed-time-flow-lift-derivative-abs-assump} and using Fubini's theorem, we obtain
\begin{align*}
D_\xi \widetilde u_t\left( x,\xi \right)\left[ \eta \right]
=
\E\left(
\nabla \varphi\left( X_t^{x,\mu} \right)
\cdot
\widetilde{\E}\left(
\partial_\mu X_t^{x,\mu}\left( \widetilde \xi \right)\widetilde \eta
\right)
\right)
=
\widetilde{\E}\left(
\E\left(
\nabla \varphi\left( X_t^{x,\mu} \right)
\cdot
\partial_\mu X_t^{x,\mu}\left( \widetilde \xi \right)
\right)
\widetilde \eta
\right).
\end{align*}
Therefore,
\begin{equation*}
\partial_\mu u_t\left( x,\mu \right)\left( v \right)
=
\E\left(
\nabla \varphi\left( X_t^{x,\mu} \right)
\partial_\mu X_t^{x,\mu}\left( v \right)
\right).
\end{equation*}
We next differentiate once more in the Lions variable. Starting from the preceding formula, and using again that the relevant mixed
derivatives of the flow exist by Assumption \ref{ass:flow-smothness}, we may differentiate under
the expectation for the same reason as before.
When differentiating
\begin{equation*}
\nabla \varphi\left( X_t^{x,\mu} \right)
\partial_\mu X_t^{x,\mu}\left( v \right)
\end{equation*}
with respect to $\mu$ in direction $\bar v$, we obtain two terms: one from differentiating
the factor $\nabla \varphi\left( X_t^{x,\mu} \right)$, namely $\nabla^2 \varphi\left( X_t^{x,\mu} \right)
\left[
\partial_\mu X_t^{x,\mu}\left( v \right),
\partial_\mu X_t^{x,\mu}\left( \bar v \right)
\right]$, and one from differentiating the factor
$\partial_\mu X_t^{x,\mu}\left( v \right)$.
Therefore,
\begin{equation*}
\partial_{\mu\mu}^2 u_t\left( x,\mu \right)\left( v,\bar v \right)
=
\E\left(
\nabla^2 \varphi\left( X_t^{x,\mu} \right)
\left[
\partial_\mu X_t^{x,\mu}\left( v \right),
\partial_\mu X_t^{x,\mu}\left( \bar v \right)
\right]
\right)
+
\E\left(
\nabla \varphi\left( X_t^{x,\mu} \right)
\cdot
\partial_{\mu\mu}^2 X_t^{x,\mu}\left( v,\bar v \right)
\right).
\end{equation*}
The same argument can now be iterated. Take any mixed derivative of $u_t$ already constructed.
By what we have shown, it is a finite sum of expectation terms in which a bounded derivative of
$\varphi$ is evaluated at $X_t^{x,\mu}$ and applied to mixed derivatives of the flow.
To take one additional state derivative, we differentiate inside the expectation and apply the
ordinary product rule and chain rule.
To take one additional Lions derivative, we lift the corresponding expectation term to
$L^2 \left( \Omega;\R^d \right)$, use the representation
\eqref{eq:fixed-time-flow-lift-derivative-abs-assump} and apply again the Banach-space chain rule and Leibniz rule.
Because every mixed derivative of the flow of total order at most $ 7 $ exists by Assumption \ref{ass:flow-smothness},
iterating this argument finitely many times proves the existence claim
of part (ii).
\\~\\
We now prove the continuity statement. Fix one mixed derivative of $ u_t $ of positive total order at most $ 7 $.
By the previous step, it is a finite sum of terms of the form \eqref{eq:fixed-time-ut-structure-abs-assump}.
It is therefore enough to prove continuity for one such term.
Fix a term of the form
\begin{equation*}
T\left( x,\mu,v_1,\dots,v_n \right)
=
\E\left(
\nabla^r \varphi\left( X_t^{x,\mu} \right)
\left[
\partial_x^{\gamma_1} \partial_v^{\beta_1}\partial_\mu^{n_1}
X_t^{x,\mu}\left( \mathbf{v}_1 \right),\dots,
\partial_x^{\gamma_r} \partial_v^{\beta_r}\partial_\mu^{n_r}
X_t^{x,\mu}\left( \mathbf{v}_r \right)
\right]
\right),
\end{equation*}
where each
$\partial_x^{\gamma_j} \partial_v^{\beta_j}\partial_\mu^{n_j}
X_t^{x,\mu}\left( \mathbf{v}_j \right)$
is a mixed derivative of the flow of positive total order at most $7$, the integers
$n_1,\dots,n_r$ satisfy
$n_1+\cdots+n_r=n$, and the families
$\mathbf{v}_1,\dots,\mathbf{v}_r$ form a partition of
$\left( v_1,\dots,v_n \right)$.
Let $\left( x_k,\mu_k,v_1^k,\dots,v_n^k \right)
\longrightarrow
\left( x,\mu,v_1,\dots,v_n \right)$ in $\R^d \times \Ptwo \times \left(
\R^d \right)^n$ as $k \to \infty$. By Assumption \ref{ass:flow-smothness}, we may work with jointly continuous versions
of $X_t^{x,\mu}$ and of the mixed derivatives $\partial_x^{\gamma_j} \partial_v^{\beta_j}\partial_\mu^{n_j}
X_t^{x,\mu}\left( \mathbf{v}_j \right)$, $j = 1,\dots,r$, in all
deterministic parameters. Therefore, $X_t^{x_k,\mu_k} \longrightarrow
X_t^{x,\mu}$ and, for each $ j = 1,\dots,r $,
\begin{equation*}
\partial_x^{\gamma_j} \partial_v^{\beta_j}\partial_\mu^{n_j}
X_t^{x_k,\mu_k}\left( \mathbf{v}_j^k \right)
\longrightarrow
\partial_x^{\gamma_j} \partial_v^{\beta_j}\partial_\mu^{n_j}
X_t^{x,\mu}\left( \mathbf{v}_j \right)
\end{equation*}
almost surely, where $\mathbf{v}_j^k$ denotes the subfamily of
$\left( v_1^k,\dots,v_n^k \right)$ corresponding to $\mathbf{v}_j$.
Since the derivatives of $ \varphi $ are bounded, it follows that the integrand in
$ T\left( x_k,\mu_k,v_1^k,\dots,v_n^k \right) $
converges almost surely to the integrand in
$ T\left( x,\mu,v_1,\dots,v_n \right) $.
\\~\\
It remains to justify passage to the limit under the expectation.
By Assumption \ref{ass:flow-smothness}, all flow derivatives of total order at most $ 7 $ have finite moments of every order.
Hence, there exists a neighborhood $U$ of $\left( x,\mu,v_1,\dots,v_n \right)$ such that the family of random variables
\begin{equation*}
\nabla^r \varphi\left( X_t^{\bar x,\bar \mu} \right)
\left[
\partial_x^{\gamma_1} \partial_v^{\beta_1}\partial_\mu^{n_1}
X_t^{\bar x,\bar \mu}\left( \bar{\mathbf{v}}_1 \right),\dots,
\partial_x^{\gamma_r} \partial_v^{\beta_r}\partial_\mu^{n_r}
X_t^{\bar x,\bar \mu}\left( \bar{\mathbf{v}}_r \right)
\right],
\quad
\left( \bar x,\bar \mu,\bar v_1,\dots,\bar v_n \right) \in U,
\end{equation*}
has uniformly bounded moments of some order strictly larger than $ 1 $.
Therefore, this family is uniformly integrable.
Since these integrands converge almost surely and form a uniformly integrable family,
Vitali's theorem implies that they converge in $ L^1 \left( \Omega \right)$. Therefore,
\begin{equation*}
T\left( x_k,\mu_k,v_1^k,\dots,v_n^k \right)
\longrightarrow
T\left( x,\mu,v_1,\dots,v_n \right).
\end{equation*}
Thus, each term of the form \eqref{eq:fixed-time-ut-structure-abs-assump} is continuous, and therefore every
mixed derivative of $ u_t $ of positive total order at most $ 7 $ is
continuous, which concludes the proof.
\end{proof}

\begin{proposition}[Fixed-time regularity of functionals of the decoupled flow]
\label{prop:fixed-time-flow-u}
Assume that $ b $ and $ \sigma $ are bounded and belong to the class
$ C_{b,\mathrm{Lip}}^{7,7}\left( \R^d \times \Ptwo \right) $. Then,
Assumption \ref{ass:flow-smothness} holds.
\end{proposition}

\begin{proof}
Fix $ t \ge 0 $ and choose $ T_t > t $. Since $ b $ and $ \sigma $ belong to
$ C_{b,\mathrm{Lip}}^{7,7}\left( \R^d \times \Ptwo \right) $, we may apply
\cite[Theorem~3.2]{CrisanMcMurray2018} on the compact time interval
$ \left[ 0,T_t \right] $.
It follows that, for all multi-indices $ \gamma $ and $ \beta $, and every integer
$ n \ge 0 $ such that $\abs{\gamma} + \abs{\beta} + n \le 7$, the derivative
\begin{equation*}
\partial_x^\gamma \partial_v^\beta \partial_\mu^n
X_s^{x,\mu}\left( v_1,\dots,v_n \right)
\end{equation*}
exists for every $ s \in \left[ 0,T_t \right] $. Evaluating at the
fixed time $s=t$ yields the existence claim. Moreover, \cite[Theorem~3.2]{CrisanMcMurray2018} places these
derivatives in the Kusuoka--Stroock classes defined in
\cite{CrisanMcMurray2018}, which in particular implies that they are
jointly continuous and have finite $ L^p \left( \Omega;\R^d\right) $-moments for every $ p \ge 1 $ on compact time intervals.
\end{proof}

\begin{lemma}[Spatial coercivity]
\label{lem:spatial-coercivity}
Assume Assumption~\ref{ass:coeff-level-new}. Then, for every $ p \in \mathcal P $, every
$ x \in \R^d $, every $ \mu \in \Ptwo $, and every $ h \in \R^d $, one has
\begin{equation}
\label{eq:spatial-coercivity-p}
2 \left\langle h, \partial_x b\left( x,\mu \right)\left[ h \right] \right\rangle
+
\left( p-1 \right)\norm{\partial_x \sigma\left( x,\mu \right)\left[ h \right]}^2
\le
-\kappa_p \abs{h}^2.
\end{equation}
\end{lemma}

\begin{proof}
Fix $ p \in \mathcal P $, $ x \in \R^d $, $ \mu \in \Ptwo $, and $ h \in \R^d $. In
\eqref{eq:uniform-time-monotonicity-p}, take $ y = x + r h $, divide by $ r^2 $, and let
$ r \to 0 $. Since $ b $ and $ \sigma $ are differentiable in the state variable, we obtain
\eqref{eq:spatial-coercivity-p}.
\end{proof}

\begin{lemma}[Decay of the first derivatives of the decoupled flow]
\label{lem:first-derivatives-decay}
Assume Assumption~\ref{ass:coeff-level-new}. Fix $ x \in \R^d $, $ \mu \in \Ptwo $, let
$ X_t^{x,\mu} $ be the decoupled flow, and $\omega$ be the constant
defined in \eqref{eq:def-omega-uniform-time}. Then, the following hold.

\begin{enumerate}
\item[(i)] For every $ h \in \R^d $, the first spatial derivative
$ \partial_x X_t^{x,\mu}\left[ h \right] $ exists and solves
\begin{equation}
\label{eq:first-variation-x}
d\left( \partial_x X_t^{x,\mu}\left[ h \right] \right)
=
\partial_x b\left( X_t^{x,\mu}, \mu_t^\mu \right)
\partial_x X_t^{x,\mu}\left[ h \right] dt
+
\partial_x \sigma\left( X_t^{x,\mu}, \mu_t^\mu \right)
\partial_x X_t^{x,\mu}\left[ h \right] dB_t,
\end{equation}
with initial condition $\partial_x X_0^{x,\mu}\left[ h \right] =
h$. Moreover, for every $ p \in \mathcal P $ and all $t \ge 0$,
\begin{equation}
\label{eq:first-variation-x-lp}
\E\left(
\abs{\partial_x X_t^{x,\mu}\left[ h \right]}^p
\right)
\le
e^{-\frac{p\kappa_p}{2} t} \abs{h}^p
\le
e^{-p\omega t} \abs{h}^p.
\end{equation}

\item[(ii)] For every $ v \in \R^d $, the Lions derivative
$ \partial_\mu X_t^{x,\mu}\left( v \right) $ exists and solves
\begin{align}
\label{eq:first-variation-mu}
d\left( \partial_\mu X_t^{x,\mu}\left( v \right) \right)
&=
\Bigg(
\partial_x b\left( X_t^{x,\mu}, \mu_t^\mu \right)
\partial_\mu X_t^{x,\mu}\left( v \right)
+
\widetilde{\E}\left(
\partial_\mu b\left( X_t^{x,\mu}, \mu_t^\mu \right)
\left( \widetilde{X}_t^{v,\mu} \right)
\partial_x \widetilde{X}_t^{v,\mu}
\right)
\nonumber\\
&\quad
+
\widetilde{\E}\left(
\partial_\mu b\left( X_t^{x,\mu}, \mu_t^\mu \right)
\left( \widetilde{X}_t^{\widetilde{\xi},\mu} \right)
\partial_\mu \widetilde{X}_t^{\widetilde{\xi},\mu}\left( v \right)
\right)
\Bigg) dt
\nonumber\\
&\quad
+
\Bigg(
\partial_x \sigma\left( X_t^{x,\mu}, \mu_t^\mu \right)
\partial_\mu X_t^{x,\mu}\left( v \right)
+
\widetilde{\E}\left(
\partial_\mu \sigma\left( X_t^{x,\mu}, \mu_t^\mu \right)
\left( \widetilde{X}_t^{v,\mu} \right)
\partial_x \widetilde{X}_t^{v,\mu}
\right)
\nonumber\\
&\quad
+
\widetilde{\E}\left(
\partial_\mu \sigma\left( X_t^{x,\mu}, \mu_t^\mu \right)
\left( \widetilde{X}_t^{\widetilde{\xi},\mu} \right)
\partial_\mu \widetilde{X}_t^{\widetilde{\xi},\mu}\left( v \right)
\right)
\Bigg) dB_t,
\end{align}
with initial condition $\partial_\mu X_0^{x,\mu}\left( v \right) =
0$. Moreover, for every $ p \in \mathcal P $, there exists a constant
$ C_p > 0 $ such that, for all $t \ge 0$,
\begin{equation*}
\sup_{x,\mu,v}
\E\left(
\abs{\partial_\mu X_t^{x,\mu}\left( v \right)}^p
\right)
\le
C_p e^{-p\omega t}.
\end{equation*}
\end{enumerate}
\end{lemma}

\begin{proof}
We start with the spatial derivative. Fix $ h \in \R^d $ and set $Y_t
= \partial_x X_t^{x,\mu}\left[ h \right]$. Differentiating the decoupled equation with respect to the initial condition gives
\eqref{eq:first-variation-x}. Fix $ p \in \mathcal P $. By It\^o's formula,
\begin{align*}
d\abs{Y_t}^p
&=
p \abs{Y_t}^{p-2}
\left\langle Y_t,
\partial_x b\left( X_t^{x,\mu}, \mu_t^\mu \right) Y_t
\right\rangle dt
+
p \abs{Y_t}^{p-2}
\left\langle Y_t,
\partial_x \sigma\left( X_t^{x,\mu}, \mu_t^\mu \right) Y_t
\right\rangle dB_t
\\
&\quad
+
\frac{p}{2} \abs{Y_t}^{p-2}
\norm{
\partial_x \sigma\left( X_t^{x,\mu}, \mu_t^\mu \right) Y_t
}^2 dt
+
\frac{p\left( p-2 \right)}{2}
\abs{Y_t}^{p-4}
\abs{
\left(
\partial_x \sigma\left( X_t^{x,\mu}, \mu_t^\mu \right) Y_t
\right)^\top
Y_t
}^2 dt.
\end{align*}
Using the fact that
\begin{equation*}
\abs{
\left(
\partial_x \sigma\left( X_t^{x,\mu}, \mu_t^\mu \right) Y_t
\right)^\top
Y_t
}^2
\le
\abs{Y_t}^2
\norm{
\partial_x \sigma\left( X_t^{x,\mu}, \mu_t^\mu \right) Y_t
}^2,
\end{equation*}
we obtain
\begin{align*}
\frac{d}{dt}
\E\left( \abs{Y_t}^p \right)
&\le
p \E\Bigg(
\abs{Y_t}^{p-2}
\left\langle Y_t,
\partial_x b\left( X_t^{x,\mu}, \mu_t^\mu \right) Y_t
\right\rangle
\Bigg)
+
\frac{p\left( p-1 \right)}{2}
\E\Bigg(
\abs{Y_t}^{p-2}
\norm{
\partial_x \sigma\left( X_t^{x,\mu}, \mu_t^\mu \right) Y_t
}^2
\Bigg).
\end{align*}
Applying Lemma~\ref{lem:spatial-coercivity} with this value of $ p $ yields
\begin{equation*}
\frac{d}{dt}
\E\left( \abs{Y_t}^p \right)
\le
-\frac{p\kappa_p}{2}
\E\left( \abs{Y_t}^p \right).
\end{equation*}
Since $ Y_0 = h $, Gronwall's lemma coupled with the fact that $ \omega \le \kappa_p / 2$ by \eqref{eq:def-omega-uniform-time} gives
\begin{equation*}
\E\left(
\abs{\partial_x X_t^{x,\mu}\left[ h \right]}^p
\right)
\le
e^{-\frac{p\kappa_p}{2} t} \abs{h}^p \le
e^{-p\omega t} \abs{h}^p.
\end{equation*}
We now turn to the measure derivative. Fix $ v \in \R^d $ and set $Y_t
= \partial_\mu X_t^{x,\mu}\left( v \right)$. Equation \eqref{eq:first-variation-mu} can be rewritten as
\begin{equation*}
dY_t
=
\left(
A_t Y_t + G_t\left( v \right)
\right) dt
+
\left(
C_t Y_t + H_t\left( v \right)
\right) dB_t,
\end{equation*}
where $A_t = \partial_x b\left( X_t^{x,\mu}, \mu_t^\mu \right)$, $C_t
= \partial_x \sigma\left( X_t^{x,\mu}, \mu_t^\mu \right)$,
\begin{align*}
G_t\left( v \right)
&=
\widetilde{\E}\left(
\partial_\mu b\left( X_t^{x,\mu}, \mu_t^\mu \right)
\left( \widetilde{X}_t^{v,\mu} \right)
\partial_x \widetilde{X}_t^{v,\mu}
\right)
+
\widetilde{\E}\left(
\partial_\mu b\left( X_t^{x,\mu}, \mu_t^\mu \right)
\left( \widetilde{X}_t^{\widetilde{\xi},\mu} \right)
\partial_\mu \widetilde{X}_t^{\widetilde{\xi},\mu}\left( v \right)
\right),
\end{align*}
and
\begin{align*}
H_t\left( v \right)
&=
\widetilde{\E}\left(
\partial_\mu \sigma\left( X_t^{x,\mu}, \mu_t^\mu \right)
\left( \widetilde{X}_t^{v,\mu} \right)
\partial_x \widetilde{X}_t^{v,\mu}
\right)
+
\widetilde{\E}\left(
\partial_\mu \sigma\left( X_t^{x,\mu}, \mu_t^\mu \right)
\left( \widetilde{X}_t^{\widetilde{\xi},\mu} \right)
\partial_\mu \widetilde{X}_t^{\widetilde{\xi},\mu}\left( v \right)
\right).
\end{align*}
Fix $ p \in \mathcal P $ and define
\begin{equation*}
m_p\left( t,v \right)
=
\sup_{x,\mu}
\E\left(
\abs{\partial_\mu X_t^{x,\mu}\left( v \right)}^p
\right).
\end{equation*}
By \eqref{eq:small-measure-dependence}, Jensen's inequality, and
\eqref{eq:first-variation-x-lp}, there exists a constant $ C_{p,d} > 0 $ such that
\begin{align}
\label{eq:G-H-bound-general}
\E\left( \abs{G_t\left( v \right)}^p \right)
+
\E\left( \norm{H_t\left( v \right)}^p \right)
\le
2^{p-1} \gamma^p
\left(
C_{p,d} e^{-\frac{p\kappa_p}{2} t}
+
m_p\left( t,v \right)
\right).
\end{align}

Applying It\^o's formula to $ \abs{Y_t}^p $ and using again
\begin{equation*}
\abs{
\left( \left( C_t Y_t + H_t\left( v \right) \right)^\top Y_t \right)
}^2
\le
\abs{Y_t}^2
\norm{C_t Y_t + H_t\left( v \right)}^2,
\end{equation*}
we obtain
\begin{align*}
\frac{d}{dt}
\E\left( \abs{Y_t}^p \right)
&\le
p \E\left(
\abs{Y_t}^{p-2}
\left\langle Y_t, A_t Y_t \right\rangle
\right)
+
\frac{p\left( p-1 \right)}{2}
\E\left(
\abs{Y_t}^{p-2}
\norm{C_t Y_t}^2
\right)
+
p \E\left(
\abs{Y_t}^{p-1} \abs{G_t\left( v \right)}
  \right)
  \\
&\quad
+
p\left( p-1 \right)
\E\left(
\abs{Y_t}^{p-1} \norm{C_t}_{\mathrm{op}} \norm{H_t\left( v \right)}
\right)
+
\frac{p\left( p-1 \right)}{2}
\E\left(
\abs{Y_t}^{p-2} \norm{H_t\left( v \right)}^2
\right).
\end{align*}
Since $ \norm{C_t}_{\mathrm{op}} \le M_\sigma $ and by Lemma~\ref{lem:spatial-coercivity},
\begin{align}
\label{eq:mp-before-young}
\frac{d}{dt}
\E\left( \abs{Y_t}^p \right)
&\le
-\frac{p\kappa_p}{2}
\E\left( \abs{Y_t}^p \right)
+
p \E\left(
\abs{Y_t}^{p-1} \abs{G_t\left( v \right)}
\right)
\nonumber\\
&\quad
+
p\left( p-1 \right) M_\sigma
\E\left(
\abs{Y_t}^{p-1} \norm{H_t\left( v \right)}
\right)
+
\frac{p\left( p-1 \right)}{2}
\E\left(
\abs{Y_t}^{p-2} \norm{H_t\left( v \right)}^2
\right).
\end{align}

We now estimate the three terms with the optimal Young parameters. First, using Young's inequality with parameter
$2^{\frac{p-1}{p}} \gamma $, we get
\begin{align*}
p \abs{Y_t}^{p-1} \abs{G_t\left( v \right)}
&\le
\left( p-1 \right) 2^{\frac{p-1}{p}} \gamma \abs{Y_t}^p
+
\frac{1}{2^{\frac{\left( p-1 \right)^2}{p}} \gamma^{p-1}}
\abs{G_t\left( v \right)}^p.
\end{align*}
Combining this with \eqref{eq:G-H-bound-general} yields
\begin{align}
\label{eq:source-estimate-G}
p \E\left(
\abs{Y_t}^{p-1} \abs{G_t\left( v \right)}
\right)
\le
p 2^{\frac{p-1}{p}} \gamma
\E\left( \abs{Y_t}^p \right)
+
C_p \gamma e^{-\frac{p\kappa_p}{2} t}
+
C_p \gamma m_p\left( t,v \right).
\end{align}

Second, using Young's inequality with the same parameter, we obtain
\begin{align*}
p\left( p-1 \right) M_\sigma \abs{Y_t}^{p-1} \norm{H_t\left( v \right)}
&\le
\left( p-1 \right)^2 2^{\frac{p-1}{p}} M_\sigma \gamma \abs{Y_t}^p
+
\frac{\left( p-1 \right) M_\sigma}{2^{\frac{\left( p-1 \right)^2}{p}} \gamma^{p-1}}
\norm{H_t\left( v \right)}^p.
\end{align*}
Using \eqref{eq:G-H-bound-general}, we can write
\begin{align}
\label{eq:source-estimate-H1}
p\left( p-1 \right) M_\sigma
\E\left(
\abs{Y_t}^{p-1} \norm{H_t\left( v \right)}
\right)
&\le
p\left( p-1 \right) 2^{\frac{p-1}{p}} M_\sigma \gamma
  \E\left( \abs{Y_t}^p \right)\nonumber \\
  &\quad
+
C_p M_\sigma \gamma e^{-\frac{p\kappa_p}{2} t}
+
C_p M_\sigma \gamma m_p\left( t,v \right).
\end{align}
Finally, using Young's inequality in the form
\begin{equation*}
a^{p-2} b^2
\le
\frac{p-2}{p} \delta a^p
+
\frac{2}{p} \delta^{-\frac{p-2}{2}} b^p,
\end{equation*}
with parameter
$ \delta = 2^{\frac{2\left( p-1 \right)}{p}} \gamma^2$, we get
\begin{align*}
\frac{p\left( p-1 \right)}{2}
\abs{Y_t}^{p-2} \norm{H_t\left( v \right)}^2
&\le
\frac{\left( p-1 \right)\left( p-2 \right)}{2}
2^{\frac{2\left( p-1 \right)}{p}} \gamma^2 \abs{Y_t}^p
+
\frac{p-1}{2^{\frac{\left( p-2 \right)\left( p-1 \right)}{p}} \gamma^{p-2}}
\norm{H_t\left( v \right)}^p.
\end{align*}
Using \eqref{eq:G-H-bound-general}, we obtain
\begin{align}
\label{eq:source-estimate-H2}
\frac{p\left( p-1 \right)}{2}
\E\left(
\abs{Y_t}^{p-2} \norm{H_t\left( v \right)}^2
\right)
\le
\frac{p\left( p-1 \right)}{2}
2^{\frac{2\left( p-1 \right)}{p}} \gamma^2
\E\left( \abs{Y_t}^p \right)
+
C_p \gamma^2 e^{-\frac{p\kappa_p}{2} t}
+
C_p \gamma^2 m_p\left( t,v \right).
\end{align}

Substituting \eqref{eq:source-estimate-G}, \eqref{eq:source-estimate-H1}, and
\eqref{eq:source-estimate-H2} into \eqref{eq:mp-before-young}, and then taking the
supremum over $ x \in \R^d $, we obtain
\begin{align*}
\frac{d}{dt} m_p\left( t,v \right)
&\le
-
\Bigg(
\frac{p\kappa_p}{2}
-
p 2^{\frac{p-1}{p}} \gamma
-
p\left( p-1 \right) 2^{\frac{p-1}{p}} M_\sigma \gamma
-
\frac{p\left( p-1 \right)}{2} 2^{\frac{2\left( p-1 \right)}{p}} \gamma^2
\Bigg)
m_p\left( t,v \right)
+
C_p e^{-\frac{p\kappa_p}{2} t}.
\end{align*}
By \eqref{eq:def-omega-uniform-time},
\begin{equation*}
\frac{p\kappa_p}{2}
-
p 2^{\frac{p-1}{p}} \gamma
-
p\left( p-1 \right) 2^{\frac{p-1}{p}} M_\sigma \gamma
-
\frac{p\left( p-1 \right)}{2} 2^{\frac{2\left( p-1 \right)}{p}} \gamma^2
\ge
p\omega.
\end{equation*}
Hence,
\begin{equation*}
\frac{d}{dt} m_p\left( t,v \right)
\le
-p\omega  m_p\left( t,v \right)
+
C_p e^{-\frac{p\kappa_p}{2} t}.
\end{equation*}
Since $ \omega \le \kappa_p / 2 $, the forcing term is bounded by $ C_p e^{-p\omega t} $. Moreover,
$ m_p\left( 0,v \right) = 0 $. Gronwall's lemma therefore yields, for
all $t \ge 0$,
\begin{equation*}
m_p\left( t,v \right)
\le
C_p e^{-p\omega t},
\end{equation*}
which concludes the proof.
\end{proof}

\begin{lemma}[Decay of the higher mixed derivatives of the decoupled flow]
\label{lem:uniform-time-second-variation}
Assume Assumption~\ref{ass:coeff-level-new}, and let $ \omega $ be the constant
defined in \eqref{eq:def-omega-uniform-time}. Then, for all multi-indices
$ \gamma $, $ \beta $ and every integer $ n \ge 0 $ satisfying $2 \le
\abs{\gamma} + \abs{\beta} + n \le 7$, and every $p \in \left[
  2,\frac{14}{\abs{\gamma} + \abs{\beta} + n} \right]$, there exists a constant $ C_{p,\gamma,\beta,n} > 0 $ such that
\begin{equation}
\label{eq:uniform-time-higher-flow-decay-p}
\sup_{x,\mu,v_1,\dots,v_n}
\E\left(
\abs{
\partial_x^\gamma \partial_v^\beta \partial_\mu^n
X_t^{x,\mu}\left( v_1,\dots,v_n \right)
}^p
\right)
\le
C_{p,\gamma,\beta,n} e^{-p\omega t},
\qquad t \ge 0.
\end{equation}
\end{lemma}

\begin{proof}
We argue by induction on the total order $q = \abs{\gamma} +
\abs{\beta} + n \in \left\{2,\dots,7\right\}$. For a fixed $ q $, we
prove \eqref{eq:uniform-time-higher-flow-decay-p} for every $p \in
\left[ 2,\frac{14}{q} \right]$. We start with the case $q=2$. Fix a
mixed derivative of total order $ 2 $ and write $Y_t
=
\partial_x^\gamma \partial_v^\beta \partial_\mu^n
X_t^{x,\mu}\left( v_1,\dots,v_n \right)$, $\abs{\gamma} + \abs{\beta}
+ n = 2$. By Proposition~\ref{prop:fixed-time-flow-u}, this derivative
exists. By the recursive construction in \cite[Theorem~3.2 and Appendix~6.2]{CrisanMcMurray2018},
the process $ Y_t $ satisfies a linear equation of the same form as the first-order
measure derivative equation: the coefficient multiplying $ Y_t $ is still given by
$ \partial_x b\left( X_t^{x,\mu},\mu_t^\mu \right) $ in the drift and
$ \partial_x \sigma\left( X_t^{x,\mu},\mu_t^\mu \right) $ in the diffusion, while the
same-order Lions coupling is carried by the terms involving $ \partial_\mu b $ and
$ \partial_\mu \sigma $, and the inhomogeneous terms are finite sums of bounded
second-order derivatives of $ b $ and $ \sigma $, multiplied by products of two first-order
derivatives, possibly inside tilded expectations.
\\~\\
We first estimate these source terms. Let $ T_t $ be one such term. If $ T_t $ is not
inside a tilded expectation, then $T_t=\Lambda_t Z_t^{(1)} Z_t^{(2)}$, where $ \Lambda_t $ is bounded and $Z_t^{(1)}$, $Z_t^{(2)}$ are first-order derivatives.
Hence, for every $ p \in \left[ 2,7 \right] $,
\begin{align*}
\E\left( \abs{T_t}^p \right)
&\le
C \E\left(
\abs{Z_t^{(1)}}^p \abs{Z_t^{(2)}}^p
\right)
\le
C
\E\left(
\abs{Z_t^{(1)}}^{2p}
\right)^{1/2}
\E\left(
\abs{Z_t^{(2)}}^{2p}
\right)^{1/2}
\le
C e^{-2p\omega t},
\end{align*}
because $ 2p \le 14 $ and Lemma~\ref{lem:first-derivatives-decay} gives the required
bound for first-order derivatives. If $T_t$ is inside a tilded expectation, Jensen's
inequality gives
\begin{align*}
\E\left(
\abs{
\widetilde{\E}\left( T_t \right)
}^p
\right)
&\le
\E\left(
\widetilde{\E}\left( \abs{T_t}^p \right)
\right)
=
\E\left( \abs{T_t}^p \right)
\le
C e^{-2p\omega t}.
\end{align*}
Since there are only finitely many such terms, the total source terms
$ R_t $ and $ S_t $ satisfy
\begin{equation}
\label{eq:uniform-time-higher-source-order2}
\sup \E\left( \abs{R_t}^p \right)
+
\sup \E\left( \norm{S_t}^p \right)
\le
C_p e^{-2p\omega t}
\end{equation}
for any $p \in \left[ 2,7 \right]$. We now apply the same It\^o--Young argument as in the proof of
Lemma~\ref{lem:first-derivatives-decay}: since the homogeneous coefficients are
the same as in the first-order equation, the same coercive estimate yields
\begin{equation*}
\frac{d}{dt}
\E\left( \abs{Y_t}^p \right)
\le
-p\omega \E\left( \abs{Y_t}^p \right)
+
C_p
\left(
\E\left( \abs{R_t}^p \right)
+
\E\left( \norm{S_t}^p \right)
\right).
\end{equation*}
Using \eqref{eq:uniform-time-higher-source-order2}, we obtain
\begin{equation*}
\frac{d}{dt}
\E\left( \abs{Y_t}^p \right)
\le
-p\omega \E\left( \abs{Y_t}^p \right)
+
C_p e^{-2p\omega t}.
\end{equation*}
Since $ Y_0 = 0 $, Gronwall's lemma yields, for all $p \in \left[ 2,7 \right]$,
\begin{equation*}
\E\left( \abs{Y_t}^p \right)
\le
C_p e^{-p\omega t},
\end{equation*}
which proves \eqref{eq:uniform-time-higher-flow-decay-p} for all second-order
derivatives.
\\~\\
We now proceed to the induction on the total order. Fix $q \in \left\{3,\dots,7\right\}$ and assume that
\eqref{eq:uniform-time-higher-flow-decay-p} holds for every mixed derivative of total
order between $ 2 $ and $ q-1 $, and every $p \in \left[
  2,\frac{14}{r} \right]$ when the total order is $r$. Let $Y_t
=
\partial_x^\gamma \partial_v^\beta \partial_\mu^n
X_t^{x,\mu}\left( v_1,\dots,v_n \right)$, $\abs{\gamma} + \abs{\beta}
+ n = q$. By the recursive construction in \cite[Theorem~3.2 and Appendix~6.2]{CrisanMcMurray2018},
the process $ Y_t $ again satisfies a linear equation of the same form as above: the
homogeneous coefficients are still the first state derivatives of the coefficients, while
the inhomogeneous terms are finite sums of bounded derivatives of $ b $ and $ \sigma $,
possibly inside tilded expectations, multiplied by products of strictly lower-order
derivatives of the flow. More precisely, every source term is of the form
\begin{equation*}
T_t
=
\Lambda_t \prod_{j=1}^m Z_t^{(j)}
\quad\mbox{or}\quad
\widetilde{\E}\left(
\widetilde \Lambda_t \prod_{j=1}^m \widetilde Z_t^{(j)}
\right),
\end{equation*}
where $\Lambda_t$ and $\widetilde \Lambda_t$ are bounded
derivatives of $b$ or $\sigma$, each $Z_t^{(j)}$ is a mixed derivative of the flow of total order
$q_j \in \left\{1,\dots,q-1\right\}$, one has $q_1 + \cdots + q_m =
q$ and $m \ge 2$. Now, fix $p \in \left[ 2,\frac{14}{q} \right]$. For
each $j$, define $r_j = \frac{pq}{q_j}$. Then $ r_j \ge 2 $. Moreover,
if $ q_j = 1 $, then $r_j = pq \le 14$, while if $ q_j \ge 2 $, then $r_j
=
\frac{pq}{q_j}
\le
\frac{14}{q_j}$. Therefore, if $q_j = 1$, Lemma~\ref{lem:first-derivatives-decay} applies and gives
\begin{equation*}
\sup \E\left( \abs{Z_t^{(j)}}^{r_j} \right)
\le
C_{r_j} e^{-r_j\omega t},
\end{equation*}
and if $ q_j \ge 2 $, the induction hypothesis applies and gives
\begin{equation*}
\sup \E\left( \abs{Z_t^{(j)}}^{r_j} \right)
\le
C_{r_j} e^{-r_j\omega t}.
\end{equation*}
Now, Hölder's inequality with exponents
$\frac{q}{q_1},\dots,\frac{q}{q_m}$ gives
\begin{align*}
\E\left( \abs{T_t}^p \right)
&\le
C
\E\left(
\prod_{j=1}^m \abs{Z_t^{(j)}}^p
\right)
\le
C
\prod_{j=1}^m
\E\left(
\abs{Z_t^{(j)}}^{r_j}
\right)^{q_j/q}
\le
C
\prod_{j=1}^m
e^{-r_j\omega t   q_j/q}
=
C e^{-pm\omega t}
\le
C e^{-2p\omega t},
\end{align*}
because $m \ge 2$. If the term is inside a tilded expectation, Jensen's inequality gives
the same bound, namely
\begin{equation*}
\E\left(
\abs{
\widetilde{\E}\left(
\widetilde \Lambda_t \prod_{j=1}^m \widetilde Z_t^{(j)}
\right)
}^p
\right)
\le
C e^{-2p\omega t}.
\end{equation*}
Since the family of source terms is finite, we conclude that for every $p \in \left[ 2,\frac{14}{q} \right]$,
\begin{equation*}
\sup \E\left( \abs{R_t}^p \right)
+
\sup \E\left( \norm{S_t}^p \right)
\le
C_p e^{-2p\omega t}.
\end{equation*}
Applying again the same It\^o--Young estimate as in the proof of
Lemma~\ref{lem:first-derivatives-decay} gives
\begin{equation*}
\frac{d}{dt}
\E\left( \abs{Y_t}^p \right)
\le
-p\omega \E\left( \abs{Y_t}^p \right)
+
C_p e^{-2p\omega t}.
\end{equation*}
Since $ Y_0 = 0 $, Gronwall's lemma yields
\begin{equation*}
\E\left( \abs{Y_t}^p \right)
\le
C_p e^{-p\omega t}
\end{equation*}
for every $p \in \left[ 2,\frac{14}{q} \right]$, which closes the
induction and concludes the proof.
\end{proof}

\subsection{Uniform-in-time weak expansion for specific functionals of the empirical measure}\label{S:UniformTimeWeakExpansionFunct}

Let $\Phi_1\left( \mu \right) = \left\langle \mu,\varphi
\right\rangle$ and $\Phi_2\left( \mu \right) = \left\langle
  \mu,\varphi \right\rangle^2$. The goal of this section is to prove Proposition \ref{prop:uniform-time-weak-expansion-abs-assump} on the uniform-in-time weak expansion for $ \Phi_1 $ and $ \Phi_2 $. We present this proof at the end of this section, as we first need some preliminary results which we present next.

  For every bounded measurable functional $ \Phi \colon \Ptwo \to \R $, we have defined the propagated
functional $P_t \Phi\left( \mu \right) = \Phi\left( \mu_t^\mu \right)$
for all $t \ge 0$ and $\mu \in \Ptwo$. We hence have
\begin{equation*}
P_t\Phi_1\left( \mu \right) = \left\langle \mu_t^\mu,\varphi
\right\rangle\quad \mbox{and}\quad P_t\Phi_2\left( \mu \right) = \left\langle \mu_t^\mu,\varphi
\right\rangle^2.
\end{equation*}
Furthermore, for any $\Psi\in M^5\big(\Ptwo\big)$, we recall the definition
\begin{equation*}
\Gamma\Psi\left(\mu\right)
=
\frac12
\int_{\R^d}
\operatorname{Tr}
\left(
a\left(v,\mu\right)\partial_{\mu\mu}^2\Psi\left(\mu\right)\left(v,v\right)
\right)
\mu\left(dv\right).
\end{equation*}

\begin{proposition}[Long-time bounds for the propagated test functionals]
\label{prop:uniform-time-semigroup-abs-assump}
Assume Assumption~\ref{ass:coeff-level-new}, and let $ \varphi \in C_b^7\left( \R^d \right) $.
Then, there exists a constant $C > 0$ such that the following hold.
\begin{enumerate}
\item[(i)] For $\ell \in \left\{ 1,2 \right\}$, $ \sup_{t \ge 0}
\norm{P_t\Phi_{\ell}}_{M^5\big( \Ptwo \big)}
\le C$.

\item[(ii)] For $\ell \in \left\{ 1,2 \right\}$ and every $t \ge 0$,
  $\sup_{x,y \in \R^d,\, \mu \in \Ptwo}
\abs{\partial_{\mu\mu}^2 P_t\Phi_{\ell}\left( \mu \right)\left( x,y \right)}
\le C e^{-\omega t}$.

\item[(iii)] For $\ell \in \left\{ 1,2 \right\}$ and every $t \ge 0$,
  $ \norm{\Gamma P_t \Phi_\ell}_{M^5\big( \Ptwo \big)}
\le C e^{-\omega t}$.
\end{enumerate}
\end{proposition}
\begin{proof}
We start by treating $P_t\Phi_1$.
Recall that $u_t\left( x,\mu \right) = \E\left( \varphi\left(
    X_t^{x,\mu} \right) \right)$ and observe that if $X_0^\mu$ is a
random variable with law $\mu$ and $\left( X_s^\mu \right)_{s \ge 0}$
is the corresponding McKean--Vlasov solution started from $X_0^\mu$,
we can write
\begin{align*}
P_t\Phi_1\left( \mu \right)
&=
\left\langle \mu_t^\mu, \varphi \right\rangle
=
\int_{\R^d} \varphi\left( y \right)\mu_t^\mu\left( dy \right)
=
\E\left( \varphi\left( X_t^\mu \right) \right).
\end{align*}
Using the tower property with respect to $X_0^\mu$, and the fact that conditionally on
$X_0^\mu = x$ the time-$t$ state has the same law as $X_t^{x,\mu}$, we get
\begin{align*}
\E\left( \varphi\left( X_t^\mu \right) \right)
=
\E\left( \E\left( \varphi\left( X_t^\mu \right)\mid X_0^\mu \right) \right)
=
\int_{\R^d}
\E\left( \varphi\left( X_t^{x,\mu} \right) \right)\mu\left( dx \right)
=
\int_{\R^d} u_t\left( x,\mu \right)\mu\left( dx \right).
\end{align*}
Hence, we have the representation
\begin{equation}
  \label{eq:rep-pt-with-ut-abs-assump}
P_t\Phi_1\left( \mu \right)
=
\int_{\R^d} u_t\left( x,\mu \right)\mu\left( dx \right).
\end{equation}
Now, by Proposition~\ref{prop:fixed-time-flow-u-abs-assump}, every mixed derivative of
$ u_t\left( x,\mu \right) $ of positive total order at most $ 7 $ is a
finite sum of terms of the form
\begin{equation*}
\E\left(
\nabla^r \varphi\left( X_t^{x,\mu} \right)
\left[
\partial_x^{\gamma_1} \partial_v^{\beta_1}\partial_\mu^{n_1}
X_t^{x,\mu}\left( \mathbf{v}_1 \right),\dots,
\partial_x^{\gamma_r} \partial_v^{\beta_r}\partial_\mu^{n_r}
X_t^{x,\mu}\left( \mathbf{v}_r \right)
\right]
\right),
\end{equation*}
where $ 1 \le r \le 7 $, where, for each $1 \leq j \leq r$, $\mathbf{v}_j$
denotes an $n_j$-tuple of Lions directions, and each
$\partial_x^{\gamma_j} \partial_v^{\beta_j}\partial_\mu^{n_j}
X_t^{x,\mu}\left( \mathbf{v}_j \right)$
is a mixed derivative of the flow of positive total order $1 \le \abs{\gamma_j} + \abs{\beta_j} + n_j
\le 7$. By Lemma \ref{lem:uniform-time-second-variation}, every such
mixed derivative of the flow satisfies an $ L^2 $-bound of order $ e^{-\omega t} $. More precisely, if
$ \abs{\gamma_j} + \abs{\beta_j} + n_j \ge 2 $, then
Lemma \ref{lem:uniform-time-second-variation} applied with $ p = 2 $ or Lemma \ref{lem:first-derivatives-decay} applied with $
p = 2 $ if $ \abs{\gamma_j} + \abs{\beta_j} + n_j = 1 $ give
\begin{equation*}
\E\left(
\abs{
\partial_x^{\gamma_j} \partial_v^{\beta_j}\partial_\mu^{n_j}
X_t^{x,\mu}\left( \mathbf{v}_j \right)
}^2
\right)
\le C e^{-2\omega t},
\end{equation*}
for every $ j = 1,\dots,r $.
On the other hand, the derivatives of $ \varphi $ are bounded, and by Assumption~\ref{ass:flow-smothness} all flow derivatives of total order at most $ 7 $ have finite moments of every order.
Therefore, for any term of the form
\begin{equation*}
\E\left(
\nabla^r \varphi\left( X_t^{x,\mu} \right)
\left[
\partial_x^{\gamma_1} \partial_v^{\beta_1}\partial_\mu^{n_1}
X_t^{x,\mu}\left( \mathbf{v}_1 \right),\dots,
\partial_x^{\gamma_r} \partial_v^{\beta_r}\partial_\mu^{n_r}
X_t^{x,\mu}\left( \mathbf{v}_r \right)
\right]
\right),
\end{equation*}
we can bound $\nabla^r \varphi\left( X_t^{x,\mu} \right)$ by $ \norm{\nabla^r \varphi}_{\infty} $ and then apply H\"older's inequality.
Since the total order of differentiation is positive, at least one
factor $\partial_x^{\gamma_j} \partial_v^{\beta_j}\partial_\mu^{n_j}
X_t^{x,\mu}\left( \mathbf{v}_j \right)$ is of positive order, and we
use the above decay estimate on that factor. All remaining factors are absorbed into the constant by their uniform finite-moment bounds.
Consequently, every such term is bounded by $ C e^{-\omega t} $, uniformly in
$ x $, $ \mu $, and the Lions variables.
Since every positive-order mixed derivative of $ u_t $ is a finite sum of such terms, we obtain
\begin{equation}
\label{eq:proof-prop-semigroup-ut-derivative-bound-abs-assump}
\sup_{x \in \R^d,\, \mu \in \Ptwo,\, v_1,\dots,v_n \in \R^d}
\abs{
\partial_x^\gamma \partial_v^\beta \partial_\mu^n u_t\left( x,\mu \right)\left( v_1,\dots,v_n \right)
}
\le
C e^{-\omega t},
\end{equation}
whenever $ \abs{\gamma} + \abs{\beta} + n \ge 1 $ and the total order is at most $ 7 $.
The order-zero term is bounded by
\begin{equation*}
\sup_{x,\mu} \abs{u_t\left( x,\mu \right)} \le \norm{\varphi}_{\infty}.
\end{equation*}
We now transfer the bounds for $u_t$ to $P_t\Phi_1$. In view of
\eqref{eq:rep-pt-with-ut-abs-assump}, $P_t\Phi_1$ is of the form
\begin{equation*}
F\left( \mu \right) = \int_{\R^d} g\left( x,\mu \right)\mu\left( dx \right),
\end{equation*}
and when differentiating a map of this form, the probability measure
 $\mu$ appears both inside the integrand $g\left( x,\mu \right)$ and
 also as the outer integrating measure. Both contributions have to be
 differentiated. To do this, let $ \xi \in L^2\left( \Omega;\R^d \right) $ have law $ \mu $, and let $ \eta \in L^2\left( \Omega;\R^d \right)$.
The lift to $L^2\left( \Omega;\R^d \right)$ of $ F $ is $\widetilde F\left( \xi \right) =
\E\left( g\left( \xi,\left[ \xi \right] \right) \right)$. For $ h \ne 0 $, we write
\begin{equation*}
\frac{\widetilde F\left( \xi + h \eta \right) - \widetilde F\left( \xi \right)}{h}
=
I_h + J_h,
\end{equation*}
where
\begin{align*}
I_h
=
\E\left(
\frac{
g\left( \xi + h \eta,\left[ \xi + h \eta \right] \right)
-
g\left( \xi,\left[ \xi + h \eta \right] \right)
}{h}
\right)\quad\mbox{and}\quad
J_h
=
\E\left(
\frac{
g\left( \xi,\left[ \xi + h \eta \right] \right)
-
g\left( \xi,\left[ \xi \right] \right)
}{h}
\right).
\end{align*}
The term $ I_h $ differentiates the first argument of $ g $, with the law frozen, and the term
$ J_h $ differentiates the dependence of the integrand on the law, with the spatial variable frozen.
Passing to the limit, we obtain
\begin{equation}
\label{eq:proof-prop-semigroup-basic-rule-abs-assump}
\partial_\mu F\left( \mu \right)\left( v \right)
=
\partial_x g\left( v,\mu \right)
+
\int_{\R^d} \partial_\mu g\left( x,\mu \right)\left( v \right)\mu\left( dx \right).
\end{equation}
The first term on the right-hand side comes from differentiating the outer measure
$ \mu\left( dx \right) $, and the second term comes from differentiating the parameter
$ \mu $ inside the integrand. We apply
\eqref{eq:proof-prop-semigroup-basic-rule-abs-assump} with
$ g = u_t $.
This yields
\begin{equation*}
\partial_\mu P_t\Phi_1\left( \mu \right)\left( v \right)
=
\partial_x u_t\left( v,\mu \right)
+
\int_{\R^d}
\partial_\mu u_t\left( x,\mu \right)\left( v \right)\mu\left( dx \right).
\end{equation*}
We now iterate the same argument.
Indeed, whenever
\begin{equation*}
F\left( \mu \right)
=
\int_{\R^d} g\left( x,\mu \right)\mu\left( dx \right),
\end{equation*}
and $ g $ has one additional mixed derivative, formula
\eqref{eq:proof-prop-semigroup-basic-rule-abs-assump} applies again.
Starting from $ g = u_t $ and repeating this differentiation rule finitely many times,
we obtain every mixed Lions derivative of $P_t\Phi_1$ entering
$ \norm{P_t\Phi_1}_{M^5\left( \Ptwo \right)} $. More precisely, every such derivative is a finite linear combination of terms of the two forms
\begin{equation}
\label{eq:proof-prop-semigroup-U-point-abs-assump}
\partial_x^\gamma \partial_v^\beta \partial_\mu^n
u_t\left( z,\mu \right)
\quad\mbox{and}\quad
\int_{\R^d}
\partial_x^\gamma \partial_v^\beta \partial_\mu^n
u_t\left( x,\mu \right)\mu\left( dx \right),
\end{equation}
where the total order of differentiation of $u_t$ is at most $ 7 $,
and $z$ denotes either one of the Lions variables introduced by the
successive differentiations, or the integration variable arising from
a previous application of the differentiation rule. The point-evaluation terms in
\eqref{eq:proof-prop-semigroup-U-point-abs-assump} are bounded by
\eqref{eq:proof-prop-semigroup-ut-derivative-bound-abs-assump}, and the integral
terms in
\eqref{eq:proof-prop-semigroup-U-point-abs-assump} satisfy the same bound because the
integrand is uniformly bounded and $ \mu $ is a probability measure. Here, the decay
in \eqref{eq:proof-prop-semigroup-ut-derivative-bound-abs-assump} comes from
Lemma \ref{lem:first-derivatives-decay} when the corresponding mixed derivative of the flow
has total order $ 1 $, and from Lemma \ref{lem:uniform-time-second-variation} when its total
order is at least $ 2 $. Consequently, every mixed Lions derivative of $ P_t\Phi_1 $ of
positive total order at most $ 5 $ is bounded by $ C e^{-\omega t} $, uniformly in all
parameters, while the order-zero term is bounded by $ \norm{\varphi}_{\infty} $. Hence,
\begin{equation}
\label{eq:proof-prop-semigroup-U-M5-abs-assump}
\sup_{t \ge 0} \norm{P_t\Phi_1}_{M^5\big( \Ptwo \big)} \le C.
\end{equation}
Moreover, applying the differentiation rule twice, every term in
$\partial_{\mu\mu}^2 P_t\Phi_1\left( \mu \right)\left( x,y \right)$ is again either a
point-evaluation term or an integral term built from positive-order derivatives of
$ u_t $. Therefore,
\eqref{eq:proof-prop-semigroup-ut-derivative-bound-abs-assump} gives
\begin{equation}
\label{eq:proof-prop-semigroup-U-second-abs-assump}
\sup_{\mu \in \Ptwo}
\sup_{x,y \in \R^d}
\abs{\partial_{\mu\mu}^2 P_t\Phi_1\left( \mu \right)\left( x,y \right)}
\le
C e^{-\omega t}.
\end{equation}
We now deal with $ P_t\Phi_2 $. Since $P_t\Phi_2 = \left( P_t\Phi_1
\right)^2$, repeated use of the ordinary product rule shows that every mixed Lions derivative of
$ P_t\Phi_2 $ of total order at most $ 5 $ is a finite sum of products of derivatives
of $ P_t\Phi_1 $. In each such product, either all factors are of order zero, in which case
the term is uniformly bounded, or at least one factor has positive total order, in which case
that factor contributes a decay $ e^{-\omega t} $ by the previous step, while all remaining
factors are controlled uniformly by
\eqref{eq:proof-prop-semigroup-U-M5-abs-assump}. Hence,
\begin{equation}
\label{eq:proof-prop-semigroup-V-M5-abs-assump}
\sup_{t \ge 0} \norm{P_t\Phi_2}_{M^5\big( \Ptwo \big)} \le C.
\end{equation}
Applying the same argument to the second Lions derivative, every term
in $\partial_{\mu\mu}^2 P_t\Phi_2\left( \mu \right)\left( x,y \right)$
contains at least one positive-order derivative of $ P_t\Phi_1 $, and therefore
\begin{equation}
\label{eq:proof-prop-semigroup-V-second-abs-assump}
\sup_{\mu \in \Ptwo}
\sup_{x,y \in \R^d}
\abs{\partial_{\mu\mu}^2 P_t\Phi_2\left( \mu \right)\left( x,y \right)}
\le
C e^{-\omega t}.
\end{equation}
This proves parts (i) and (ii). It remains to estimate $ \Gamma P_t \Phi_\ell $.
For $ \ell \in \left\{ 1,2 \right\} $, we again use the differentiation rule
\eqref{eq:proof-prop-semigroup-basic-rule-abs-assump}, now with
\begin{equation*}
g\left( v,\mu \right)
=
\frac{1}{2}
\operatorname{Tr}\left(
a\left( v,\mu \right)
\partial_{\mu\mu}^2 P_t \Phi_\ell\left( \mu \right)\left( v,v \right)
\right).
\end{equation*}
Thus, every mixed Lions derivative of
$ \Gamma P_t \Phi_\ell $ of total order at most $ 5 $ is a finite linear combination of
point-evaluation terms and integral terms obtained by differentiating $ g $. Each such
term is a product of bounded derivatives of $ a $ and derivatives of $ P_t\Phi_\ell $.
Moreover, every such term contains at least one factor carrying the decay
$ e^{-\omega t} $: either the factor
$ \partial_{\mu\mu}^2 P_t\Phi_\ell $
itself, controlled by
\eqref{eq:proof-prop-semigroup-U-second-abs-assump} and
\eqref{eq:proof-prop-semigroup-V-second-abs-assump}, or another positive-order
derivative of $ P_t\Phi_\ell $, controlled by
\eqref{eq:proof-prop-semigroup-U-M5-abs-assump} and
\eqref{eq:proof-prop-semigroup-V-M5-abs-assump} together with the decay estimates
established above. More precisely, the decay of this positive-order factor is obtained
from Lemma \ref{lem:first-derivatives-decay} if the underlying mixed derivative of the flow
has total order $ 1 $, and from Lemma \ref{lem:uniform-time-second-variation} if its total order is
at least $ 2 $. All remaining factors are bounded uniformly in $ t $. Therefore,
for every $ t \ge 0 $,
\begin{equation*}
\norm{\Gamma P_t \Phi_\ell}_{M^5\big( \Ptwo \big)}
\le
C e^{-\omega t},
\end{equation*}
which is part (iii) and concludes the proof.
\end{proof}
Now, we are finally in position to prove Proposition \ref{prop:uniform-time-weak-expansion-abs-assump} on the uniform-in-time weak expansion for $ \Phi_1 $ and $ \Phi_2 $.

\begin{proof}[Proof of Proposition \ref{prop:uniform-time-weak-expansion-abs-assump}]
Fix $ \ell \in \left\{ 1,2 \right\} $ and $ t \ge 0 $. We start by
deriving an exact dynamic identity for $\E\left( \Phi_\ell\left(
    \mu_t^N \right) \right) - \Phi_\ell\left( \mu_t \right)$. For $ s
\in \left[ 0,t \right] $, define $\Psi_s\left( \mu \right) =
P_{t-s}\Phi_\ell\left( \mu \right)$, so that $\Psi_t\left( \mu \right)
= \Phi_\ell\left( \mu \right)$ and $\Psi_0\left( \mu \right) =
P_t\Phi_\ell\left( \mu \right)$. Moreover, if $ \left( \mu_s \right)_{s \ge 0} $ denotes the deterministic McKean--Vlasov law
flow started from $ \nu $, then, for any $0 \le s \le t$, the semigroup property gives
\begin{equation*}
\Psi_s\left( \mu_s \right)
=
P_{t-s}\Phi_\ell\left( \mu_s \right)
=
\Phi_\ell\left( \mu_t \right),
\end{equation*}
and in particular $\Psi_0\left( \nu \right) = P_t\Phi_\ell\left( \nu
\right) = \Phi_\ell\left( \mu_t \right)$. We now apply \cite[Lemma~2.11(ii)]{CST2022}
to the time-dependent functional $\Psi_s$ on the interval $ \left[ 0,t
\right]$ (in the notation of \cite{CST2022}, we take $U_t\left( s,\mu \right) = \Psi_s\left( \mu \right) = P_{t-s}\Phi_\ell\left( \mu \right)$), which yields
\begin{align*}
\Psi_t\left( \mu_t^N \right)
&=
\Psi_0\left( \mu_0^N \right)
+
\frac{1}{2N}
\int_0^t
\int_{\R^d}
\operatorname{Tr}\left(
a\left( v,\mu_s^N \right)
\partial_{\mu\mu}^2 \Psi_s\left( \mu_s^N \right)\left( v,v \right)
\right)
\mu_s^N\left( dv \right) ds
+
M_{\ell,t}^N,
\end{align*}
where $ M_{\ell}^N $ is a square-integrable martingale with
$M_{\ell,0}^N = 0 $. By the definition of $ \Gamma $ given in
\eqref{defgamma-abs-assump}, the preceding identity becomes
\begin{equation*}
\Phi_\ell\left( \mu_t^N \right)
=
P_t\Phi_\ell\left( \mu_0^N \right)
+
\frac{1}{N}
\int_0^t
\Gamma P_{t-s}\Phi_\ell \left( \mu_s^N \right) ds
+
M_{\ell,t}^N.
\end{equation*}
Taking expectations and using the fact that $ \E\left( M_{\ell,t}^N \right) = 0 $, we obtain
\begin{align}
\label{eq:proof-uniform-time-exact-dynamic-abs-assump}
\E\left( \Phi_\ell\left( \mu_t^N \right) \right) - \Phi_\ell\left( \mu_t \right)
&=
\E\left( P_t\Phi_\ell\left( \mu_0^N \right) \right) - P_t\Phi_\ell\left( \nu \right)
+
\frac{1}{N}
\int_0^t
  \Gamma P_{t-s}\Phi_\ell \left( \mu_s \right) ds
  \nonumber \\
  &\quad +
\frac{1}{N}
\int_0^t
\left[
\E\left( \Gamma P_{t-s}\Phi_\ell \left( \mu_s^N \right) \right)
-
\Gamma P_{t-s}\Phi_\ell \left( \mu_s \right)
  \right] ds.
\end{align}
We now focus on the initial term $\E\left( P_t\Phi_\ell\left( \mu_0^N
  \right) \right) - P_t\Phi_\ell\left( \nu \right)$ of the above
expansion. Since $ \mu_0^N $ is the empirical measure of i.i.d.\ samples from $ \nu $,
we apply \cite[Theorem~2.14(ii)]{CST2022} to the functional $ P_t\Phi_\ell $ at the
measure $ \nu $.
This yields
\begin{equation}
\label{eq:proof-uniform-time-initial-abs-assump}
\E\left( P_t\Phi_\ell\left( \mu_0^N \right) \right) - P_t\Phi_\ell\left( \nu \right)
=
\frac{\beta_\ell\left( t \right)}{N}
+
R_{\ell,N}^0\left( t \right),
\end{equation}
where we may take
\begin{equation*}
\beta_\ell\left( t \right)
=
\frac{1}{2}
\widetilde{\E}\left[
\int_{\R^d}
\frac{\delta^2\left( P_t\Phi_\ell \right)}{\delta m^2}\left( \nu \right)\left( \widetilde \xi,y \right)
\left( \delta_{\widetilde \xi} - \nu \right)\left( dy \right)
\right]
\end{equation*}
for an independent random variable $ \widetilde \xi $ with law $
\nu$. Here, $ \delta^2/\delta m^2 $ denotes the second linear
functional derivative in the sense of \cite[Section~2.1.1 and Definition~2.1]{CST2022}. Its relation
with the Lions derivative used in the present note is given in
\cite[Theorem~2.6 and Lemma 2.7]{CST2022}.
Moreover, the remainder $R_{\ell,N}^0\left( t \right)$ satisfies
\begin{equation}
\label{eq:proof-uniform-time-initial-rem-abs-assump}
\sup_{t \ge 0} \abs{R_{\ell,N}^0\left( t \right)} \le \frac{C}{N^2},
\end{equation}
where the constant above (coming from the static expansion \cite[Theorem~2.14(ii)]{CST2022}) depends only on
a finite collection of bounds on the measure derivatives of the test functional
$ P_t\Phi_\ell $ required by that theorem.
By \cite[Theorem~2.6]{CST2022}, these are controlled by the corresponding Lions
derivatives, and
Proposition~\ref{prop:uniform-time-semigroup-abs-assump} provides bounds on those
derivatives uniformly in $t$.
\\~\\
We now estimate $ \beta_\ell\left( t \right) $.
By \cite[Theorem~2.6 and Lemma~2.7]{CST2022}, the second linear functional derivative
is controlled by the second Lions derivative.
Since $ \nu $ is fixed, this gives
\begin{equation*}
\abs{\beta_\ell\left( t \right)}
\le
C
\sup_{x,y \in \R^d}
\abs{
\partial_{\mu\mu}^2 \left( P_t\Phi_\ell \right)\left( \nu \right)\left( x,y \right)
}.
\end{equation*}
For $\ell \in \left\{ 1,2 \right\}$, this is bounded by
Proposition~\ref{prop:uniform-time-semigroup-abs-assump}. Hence,
\begin{equation}
\label{eq:proof-uniform-time-beta-decay-abs-assump}
\abs{\beta_\ell\left( t \right)} \le C e^{-\omega t},
\qquad t \ge 0.
\end{equation}
We now turn to the control of the integral error term in the expansion
\eqref{eq:proof-uniform-time-exact-dynamic-abs-assump}. By Proposition~\ref{prop:uniform-time-semigroup-abs-assump}, we have
\begin{equation*}
\norm{\Gamma P_{t-s}\Phi_\ell}_{M^5\big( \Ptwo \big)}
\le
C e^{-\omega\left( t-s \right)}
\end{equation*}
for every $0 \le s \le t$, and in particular
\begin{equation}
\label{eq:proof-uniform-time-G-pointwise-abs-assump}
\abs{\Gamma P_{t-s}\Phi_\ell\left( \mu \right)}
\le
C e^{-\omega\left( t-s \right)}
\end{equation}
for every $0 \le s \le t$ and $\mu \in \Ptwo$. Applying now \cite[Theorem~2.14(i)]{CST2022} to the functional $\Gamma P_{t-s}\Phi_\ell$ at the
measure $ \mu_s $, we obtain, for every $0 \le s \le t$,
\begin{equation*}
\abs{
\E\left( \Gamma P_{t-s}\Phi_\ell\left( \mu_s^N \right) \right)
-
\Gamma P_{t-s}\Phi_\ell\left( \mu_s \right)
}
\le
\frac{C e^{-\omega\left( t-s \right)}}{N}.
\end{equation*}
Here, the constant is uniform in $s$ and $t$. Indeed, the constant in the static estimate \cite[Theorem~2.14(i)]{CST2022} depends only
on a finite collection of bounds on the measure derivatives of the test functional
$\Gamma P_{t-s}\Phi_\ell$ required by that theorem.
By \cite[Theorem~2.6]{CST2022}, these are controlled by the corresponding Lions derivatives,
and Proposition~\ref{prop:uniform-time-semigroup-abs-assump} yields bounds on those derivatives of
$\Gamma P_{t-s}\Phi_\ell$ of order $ C e^{-\omega\left( t-s \right)}
$, uniformly in $s$ and $t$, and uniformly in the base measure.
Consequently,
\begin{equation}
\label{eq:proof-uniform-time-integral-error-abs-assump}
\abs{
\frac{1}{N}
\int_0^t
\left[
\E\left( \Gamma P_{t-s}\Phi_\ell \left( \mu_s^N \right) \right)
-
\Gamma P_{t-s}\Phi_\ell \left( \mu_s \right)
\right] ds
}
\le
\frac{C}{N^2}
\int_0^t e^{-\omega\left( t-s \right)} ds
\le
\frac{C}{\omega N^2}.
\end{equation}
Now, all that remains to do is to put all the previous estimates
together. Substituting \eqref{eq:proof-uniform-time-initial-abs-assump} into
\eqref{eq:proof-uniform-time-exact-dynamic-abs-assump}, we get
\begin{align*}
\E\left( \Phi_\ell\left( \mu_t^N \right) \right)
&=
\Phi_\ell\left( \mu_t \right)
+
\frac{\beta_\ell\left( t \right)}{N}
+
\frac{1}{N}
\int_0^t
\Gamma P_{t-s}\Phi_\ell \left( \mu_s \right) ds
\\
&\quad
+
R_{\ell,N}^0\left( t \right)
+
\frac{1}{N}
\int_0^t
\left[
\E\left( \Gamma P_{t-s}\Phi_\ell \left( \mu_s^N \right) \right)
-
\Gamma P_{t-s}\Phi_\ell \left( \mu_s \right)
\right] ds.
\end{align*}
We therefore define
\begin{equation*}
R_{\ell,N}\left( t \right)
=
R_{\ell,N}^0\left( t \right)
+
\frac{1}{N}
\int_0^t
\left[
\E\left( \Gamma P_{t-s}\Phi_\ell \left( \mu_s^N \right) \right)
-
\Gamma P_{t-s}\Phi_\ell\left( \mu_s \right)
\right] ds
\end{equation*}
and \eqref{eq:uniform-time-weak-expansion-abs-assump} follows immediately, while
\eqref{eq:proof-uniform-time-initial-rem-abs-assump} and
\eqref{eq:proof-uniform-time-integral-error-abs-assump} imply
\begin{equation*}
\sup_{t \ge 0} \abs{R_{\ell,N}\left( t \right)} \le \frac{C}{N^2},
\end{equation*}
which proves \eqref{eq:uniform-time-weak-rem-abs-assump}. Finally, \eqref{eq:proof-uniform-time-beta-decay-abs-assump} gives the first bound in
\eqref{eq:uniform-time-weak-coeff-abs-assump}, and
\eqref{eq:proof-uniform-time-G-pointwise-abs-assump} gives the second.
\end{proof}

\begin{corollary}[Uniform boundedness of the first-order coefficient]
\label{cor:uniform-time-alpha-abs-assump}
Assume  Assumption~\ref{ass:coeff-level-new}, and let $\varphi \in C_b^7\left( \R^d \right)$. Define
\begin{equation*}
\alpha_\ell\left( t \right)
=
\beta_\ell\left( t \right)
+
\int_0^t \Gamma P_{t-s} \Phi_\ell\left( \mu_s \right) ds.
\end{equation*}
Then, it holds that
\begin{equation}
\label{eq:uniform-time-alpha-bound-abs-assump}
\sup_{t \ge 0} \abs{\alpha_\ell\left( t \right)} < \infty.
\end{equation}
\end{corollary}

\begin{proof}
By \eqref{eq:uniform-time-weak-coeff-abs-assump}, we can write
\begin{equation*}
\abs{\alpha_\ell\left( t \right)}
\le
C e^{-\omega t}
+
C \int_0^t e^{-\omega \left( t-s \right)} ds
\le C + \frac{C}{\omega},
\end{equation*}
which proves \eqref{eq:uniform-time-alpha-bound-abs-assump}.
\end{proof}

\subsection{Abstract assumptions on the flow}
\label{subsec-coeff-level}

The arguments of the previous section show that the key input for the uniform-in-time
variance analysis is not the coefficient-level assumption itself, but rather the decay of the
derivatives of the decoupled flow
\eqref{eq:def-decoupl-mcvlasov-flow} associated with the law flow
$\left( \mu_t^\mu \right)_{t \ge 0}$. In this sense,
Assumption~\ref{ass:coeff-level-new} should be viewed only as a concrete sufficient
condition ensuring the more intrinsic
Assumption~\ref{ass:flow-level-new-abstract}. The latter is the assumption that is
actually used in the proof of the uniform-in-time weak expansion.

More precisely, Assumption~\ref{ass:coeff-level-new} combines a dissipativity condition in
the state variable with a smallness assumption on the measure derivatives of $b$ and
$\sigma$, and thereby provides a verifiable set of hypotheses under which
Assumption~\ref{ass:flow-level-new-abstract} holds. However, once
Assumption~\ref{ass:flow-level-new-abstract} is established---possibly by a different
argument and under a different set of structural assumptions---the proof of
Proposition~\ref{prop:uniform-time-weak-expansion-abs-assump} goes through in the same
way.

\begin{assumption}[Flow-level  dissipative regularity]
\label{ass:flow-level-new-abstract}
Let $\left( X_t^{x,\mu} \right)_{t \ge 0}$ be the decoupled flow
defined in \eqref{eq:def-decoupl-mcvlasov-flow}. Assume that
there exists a constant $ \omega > 0 $ such that the following hold.
\begin{enumerate}
\item[(a)]
For every $ p \in \left[ 2,14 \right] $, there exists a constant $ C_p > 0 $ such that,
for every $ t \ge 0 $,
\begin{equation*}
\sup_{x,\mu,\abs{h}\le 1}
\E\left(
\abs{\partial_x X_t^{x,\mu}\left[ h \right]}^p
\right)
+
\sup_{x,\mu,v}
\E\left(
\abs{\partial_\mu X_t^{x,\mu}\left( v \right)}^p
\right)
\le C_p e^{-p\omega t}.
\end{equation*}

\item[(b)]
For all multi-indices $ \gamma $, $ \beta $ and every integer $ n \ge 0 $ such that
$ 2 \le \abs{\gamma} + \abs{\beta} + n \le 7 $, and every $p \in
\left[ 2,\frac{14}{\abs{\gamma} + \abs{\beta} + n} \right]$, there exists a constant $ C_{p,\gamma,\beta,n} > 0 $ such that, for every $ t \ge 0 $,
\begin{equation*}
\sup_{x,\mu,v_1,\dots,v_n}
\E\left(
\abs{
\partial_x^\gamma \partial_v^\beta \partial_\mu^n
X_t^{x,\mu}\left( v_1,\dots,v_n \right)
}^p
\right)
\le C_{p,\gamma,\beta,n} e^{-p\omega t}.
\end{equation*}
\end{enumerate}
\end{assumption}
As we show in the previous section, working under this flow-level Assumption~\ref{ass:flow-level-new-abstract}, we first
transfer these properties to the function
\begin{equation*}
u_t\left( x,\mu \right) = \E\left( \varphi\left( X_t^{x,\mu} \right) \right),
\end{equation*}
and then to the propagated functionals $P_t \Phi_\ell$. This yields the bounds required to apply
the fixed-time weak expansion uniformly in time, and therefore gives a weak expansion for
$\E\left( \Phi_\ell\left( \mu_t^N \right) \right)$ with a remainder of order $N^{-2}$ uniformly in
$t \ge 0$.

\begin{proposition}[Fixed-time regularity of functionals of the decoupled flow]
Assume  Assumption~\ref{ass:coeff-level-new}. Then,
Assumption \ref{ass:flow-level-new-abstract} holds.
\end{proposition}

\begin{proof}
This is a direct application of Proposition \ref{prop:fixed-time-flow-u}, Lemma \ref{lem:first-derivatives-decay} and
Lemma \ref{lem:uniform-time-second-variation}.
\end{proof}

\section{Variance of the limiting fluctuations}
\label{limiting-variance-section}

\subsection{Representation and characterization of the limiting variance}

The uniform variance comparison used in the proof of the main theorem was stated
earlier as Proposition~\ref{prop:uniform-time-variance-comparison-abs-assump}. We
prove it first. The argument is based on the uniform weak expansions for the two
functionals $\Phi_1$ and $\Phi_2$, together with the definition of the limiting
variance. Recall that, for $\varphi \in C_b^7\left( \R^d \right)$,
Definition~\ref{def:uniform-time-limiting-variance-abs-assump} sets
\begin{equation*}
\sigma_t^2\left( \varphi \right)
=
\alpha_2\left( t \right)
-
2 \left\langle \mu_t,\varphi \right\rangle \alpha_1\left( t \right),
\qquad t \ge 0,
\end{equation*}
where $\alpha_1$ and $\alpha_2$ are the coefficients obtained in
Corollary~\ref{cor:uniform-time-alpha-abs-assump}. This definition is chosen so that
$\sigma_t^2\left(\varphi\right)$ is the coefficient governing the limiting behavior of
the finite-particle variances.

\begin{proof}[Proof of
  Proposition~\ref{prop:uniform-time-variance-comparison-abs-assump}]
By Proposition~\ref{prop:uniform-time-weak-expansion-abs-assump}, we
have the weak expansions
\begin{equation}
\label{eq:uniform-time-first-moment-abs-assump}
\E\left( \left\langle \mu_t^N,\varphi \right\rangle \right)
=
 \left\langle \mu_t,\varphi \right\rangle + \frac{\alpha_1\left( t \right)}{N} + R_{1,N}\left( t \right)
\end{equation}
and
\begin{equation}
\label{eq:uniform-time-second-moment-abs-assump}
\E\left( \left\langle \mu_t^N,\varphi \right\rangle^2 \right)
=
 \left\langle \mu_t,\varphi \right\rangle^2 + \frac{\alpha_2\left( t \right)}{N} + R_{2,N}\left( t \right),
\end{equation}
with
\begin{equation*}
\sup_{t \ge 0} \abs{R_{\ell,N}\left( t \right)} \le \frac{C}{N^2}
\qquad \text{and} \qquad
\sup_{t \ge 0} \abs{\alpha_\ell\left( t \right)} \le C,
\end{equation*}
where the second bound comes from applying Corollary~\ref{cor:uniform-time-alpha-abs-assump}.
Squaring \eqref{eq:uniform-time-first-moment-abs-assump}, subtracting the result from
\eqref{eq:uniform-time-second-moment-abs-assump}, and multiplying by $ N $, we obtain
\begin{align*}
\sigma_{N,t}^2\left( \varphi \right) - \sigma_t^2\left( \varphi \right)
&=
N R_{2,N}\left( t \right)
-
2 \left\langle \mu_t,\varphi \right\rangle N R_{1,N}\left( t \right)
-
\frac{\alpha_1\left( t \right)^2}{N}
-
2 \alpha_1\left( t \right) R_{1,N}\left( t \right)
-
N R_{1,N}\left( t \right)^2.
\end{align*}
Since $\abs{\left\langle \mu_t,\varphi \right\rangle} \le
\norm{\varphi}_\infty$, each term on the right-hand side of the above equation
is bounded by $ C/N $, uniformly in $ t $. This concludes the proof.
\end{proof}

\begin{remark}
Definition~\ref{def:uniform-time-limiting-variance-abs-assump} is equivalent to saying that $\sigma_t^2\left(\varphi\right)$ is the $1/N$-coefficient in the
weak expansion of $N\Var\left(\left\langle\mu_t^N,\varphi\right\rangle\right)$.
\end{remark}

\subsection{Backward PDE representation of $\sigma_t^2\left(\varphi\right)$}\label{S:BackwardPDErepresentationVar}

For $\mu \in \Ptwo$ and a smooth function $f \colon \R^d \to \R$, define the McKean--Vlasov generator
\begin{equation*}
\left( \mathcal L_\mu f \right)\left( x \right)
=
b\left( x,\mu \right)\cdot \nabla f\left( x \right)
+
\frac12 \operatorname{Tr}
\left(
a\left( x,\mu \right)\nabla^2 f\left( x \right)
\right),
\end{equation*}
as well as the linearized operator
\begin{equation*}
\left( \mathcal A_\mu f \right)\left( y \right)
=
\int_{\R^d}
\left(
\partial_\mu b\left( x,\mu \right)\left( y \right)\cdot \nabla f\left( x \right)
+
\frac12 \operatorname{Tr}
\left(
\partial_\mu a\left( x,\mu \right)\left( y \right)\nabla^2 f\left( x \right)
\right)
\right)
\mu\left( dx \right),
\end{equation*}
where $ \partial_\mu b\left( x,\mu \right)\left( y \right) \in \R^d $ and
$ \partial_\mu a\left( x,\mu \right)\left( y \right) \in \R^{d \times d} $
denote Lions derivatives.

\begin{proposition}\label{P:BackwardRepPDEVariance}
Assume Assumption~\ref{ass:smooth}. Let $\left(\mu_s\right)_{0\le s\le t}$ be the law flow of the McKean--Vlasov SDE~\eqref{eq:MV}.
Let $\varphi\in C_b^5\left(\R^d\right)$ and let $\left(\psi_s\right)_{0\le s\le t}$ be a classical solution to the backward equation
\begin{equation}
\label{eq:backwardpsi}
\partial_s\psi_s+\mathcal L_{\mu_s}\psi_s+\mathcal A_{\mu_s}\psi_s=0,\qquad s\in\left[0,t\right),\qquad \psi_t=\varphi.
\end{equation}
Then, the limiting variance in Definition \ref{def:uniform-time-limiting-variance-abs-assump} can be written as
\begin{equation}
\label{eq:siglimitPDE}
\sigma_t^2\left(\varphi\right)
=
\Var\left(\psi_0\left(X_0\right)\right)
+\int_0^t \E\left(\norm{\sigma\left(X_s,\mu_s\right)^\top\nabla\psi_s\left(X_s\right)}^2\right) ds,
\end{equation}
where $X$ solves \eqref{eq:MV}. In particular, $\sigma_t^2\left(\varphi\right)=\lim_{N\to\infty}\sigma_{N,t}^2\left(\varphi\right)$.
\end{proposition}

\begin{proof}
Set $\eta_s^N=\sqrt{N}\left(\mu_s^N-\mu_s\right)$, so that $G_t^N\left(\varphi\right)=\left\langle\eta_t^N,\varphi\right\rangle$.
Under global Lipschitz and smoothness assumptions of the type in Assumption~\ref{ass:smooth},
the functional fluctuation CLT for McKean--Vlasov particle systems
(see Fern\'andez--M\'el\'eard \cite{FM1997}) states that
$\eta^N$ converges in law (in suitable negative Sobolev topologies) to a centered Gaussian
process $\eta$ which solves, in weak form, the linear SPDE
\begin{equation*}
d\left\langle\eta_s,f\right\rangle
=
\left\langle\eta_s,\mathcal L_{\mu_s}f\right\rangle ds
+
\left\langle\eta_s,\mathcal A_{\mu_s}f\right\rangle ds
+
dM_s\left(f\right),
\end{equation*}
for every smooth $f$, where $M\left(f\right)$ is a continuous centered Gaussian martingale with quadratic variation
\begin{equation}
\label{eq:bracket}
\left\langle M\left(f\right)\right\rangle_t=\int_0^t\left\langle\mu_s,\norm{\sigma\left(\cdot,\mu_s\right)^\top\nabla f\left(\cdot\right)}^2\right\rangle ds.
\end{equation}
See also Delarue--Lacker--Ramanan \cite{DLR2019} for such fluctuation limits,
and the review \cite{CDreview} for propagation of chaos and fluctuation results.
\\~\\
Assuming that $\psi$ is a classical solution of
\eqref{eq:backwardpsi}, we now apply Lemma \ref{lem:ito_pairing} below with $\mathcal G_s =
\mathcal L_{\mu_s}+\mathcal A_{\mu_s}$ and the semimartingale
$s\mapsto \left\langle\eta_s,\psi_s\right\rangle$. This yields
\begin{equation*}
d\langle \eta_s,\psi_s\rangle
=
\langle \eta_s,\partial_s\psi_s\rangle ds
+\langle \eta_s,\partial_s\psi_s + \mathcal L_{\mu_s}\psi_s+\mathcal A_{\mu_s}\psi_s\rangle ds
+dM_s(\psi_s)
\end{equation*}
and the drift terms cancel by \eqref{eq:backwardpsi}. Integrating from
$0$ to $t$ hence yields the martingale representation
\begin{equation*}
\left\langle\eta_t,\varphi\right\rangle=\left\langle\eta_0,\psi_0\right\rangle+M_t\left(\psi\right).
\end{equation*}
By independence of the initial conditions and the Brownian motions,
$\left\langle\eta_0,\psi_0\right\rangle$ is a centered Gaussian variable with variance $\Var\left(\psi_0\left(X_0\right)\right)$
and is independent of the martingale term. Hence,
\begin{equation*}
\Var\left(\left\langle\eta_t,\varphi\right\rangle\right)=\Var\left(\left\langle\eta_0,\psi_0\right\rangle\right)+\E\left(\left\langle M\left(\psi\right)\right\rangle_t\right).
\end{equation*}
Using \eqref{eq:bracket} to make the quadratic variation explicit gives \eqref{eq:siglimitPDE}.
\\~\\
Finally, we have
$\sigma_{N,t}^2\left(\varphi\right)=\Var\left(\left\langle\eta_t^N,\varphi\right\rangle\right)$
and as
$\left\langle\eta_t^N,\varphi\right\rangle$ converges to
$\left\langle\eta_t,\varphi\right\rangle$ as $N \to \infty$,
we have $\lim_{N\to\infty}\sigma_{N,t}^2\left(\varphi\right)=\Var\left(\left\langle\eta_t,\varphi\right\rangle\right)$.
On the other hand, Proposition~\ref{prop:uniform-time-variance-comparison-abs-assump} identifies the same limit as $\sigma_t^2\left(\varphi\right)$,
so the two characterizations agree.
\end{proof}
The next lemma allows one to pass from a weak semimartingale formulation stated for fixed
test functions to the corresponding identity for smoothly time-dependent test functions.
This is used to insert the semigroup-evolved test functions into the
fluctuation equation.
\begin{lemma}
\label{lem:ito_pairing}
Fix $T>0$ and let $\eta_t$ be a random finite signed measure on $\R^d$ for each $t\in[0,T]$.
Assume that for every fixed test function $f\in C_b^2(\R^d)$, the real
process
\begin{equation*}
t\mapsto \langle \eta_t,f\rangle
\end{equation*}
is a continuous semimartingale admitting a decomposition
\begin{equation}
\label{eq:weak_semimart}
\langle \eta_t,f\rangle
=
\langle \eta_0,f\rangle + \int_0^t \langle \eta_s, \mathcal G_s f\rangle ds + M_t(f),
\end{equation}
where $(\mathcal G_s)_{s\in[0,T]}$ is a (possibly random, but progressively measurable) family of linear operators on $C_b^2(\R^d)$
and $M_t(f)$ is a continuous martingale.
\\~\\
Let $\psi \colon [0,T]\times\R^d\to\R$ be deterministic with $\psi\in
C^1\left([0,T];C_b^2(\R^d)\right)$, i.e., $t\mapsto
\psi_t=\psi(t,\cdot)$  is $C^1$ into $C_b^2$. Then, $\langle \eta_t,\psi_t\rangle$ is a continuous semimartingale and satisfies
\begin{equation}
\label{eq:pairing_ito}
d\langle \eta_t,\psi_t\rangle
=
\langle \eta_t,\partial_t\psi_t\rangle dt
+\langle \eta_t,\mathcal G_t\psi_t\rangle dt
+dM_t(\psi_t).
\end{equation}
\end{lemma}

\begin{proof}
Fix a partition $\pi=\{0=t_0<t_1<\cdots<t_n=t\}$ of $[0,t]$. Write the
telescoping sum
\begin{equation*}
\langle \eta_t,\psi_t\rangle-\langle \eta_0,\psi_0\rangle
=
\sum_{k=0}^{n-1}\left(\langle \eta_{t_{k+1}},\psi_{t_{k+1}}\rangle-\langle \eta_{t_k},\psi_{t_k}\rangle\right).
\end{equation*}
Now, we add and subtract $\langle \eta_{t_{k+1}},\psi_{t_k}\rangle$ to
get
\begin{align*}
\langle \eta_{t_{k+1}},\psi_{t_{k+1}}\rangle-\langle \eta_{t_k},\psi_{t_k}\rangle
=
\left(\langle \eta_{t_{k+1}},\psi_{t_k}\rangle-\langle \eta_{t_k},\psi_{t_k}\rangle\right)
+
\langle
  \eta_{t_{k+1}},\psi_{t_{k+1}}-\psi_{t_k}\rangle= (\mathrm{I})_k + (\mathrm{II})_k.
\end{align*}
For $(\mathrm{I})_k$, we apply \eqref{eq:weak_semimart} with the fixed test function $f=\psi_{t_k}$ on the interval
$[t_k,t_{k+1}]$, which yields
\begin{equation*}
(\mathrm{I})_k
=
\int_{t_k}^{t_{k+1}}\langle \eta_s,\mathcal G_s\psi_{t_k}\rangle ds
+M_{t_{k+1}}(\psi_{t_k})-M_{t_k}(\psi_{t_k}).
\end{equation*}
For $(\mathrm{II})_k$, the fundamental theorem of calculus implies that
\begin{equation*}
\psi_{t_{k+1}}-\psi_{t_k}=\int_{t_k}^{t_{k+1}}\partial_r\psi_r dr.
\end{equation*}
Hence,
\begin{equation*}
(\mathrm{II})_k
=
\int_{t_k}^{t_{k+1}}\langle \eta_{t_{k+1}},\partial_r\psi_r\rangle dr.
\end{equation*}
Summing over $k$ yields
\begin{align*}
\langle \eta_t,\psi_t\rangle-\langle \eta_0,\psi_0\rangle
&=
\sum_{k=0}^{n-1}\int_{t_k}^{t_{k+1}}\langle \eta_s,\mathcal
  G_s\psi_{t_k}\rangle ds\\
  &\quad
+\sum_{k=0}^{n-1}\left(M_{t_{k+1}}(\psi_{t_k})-M_{t_k}(\psi_{t_k})\right)
+\sum_{k=0}^{n-1}\int_{t_k}^{t_{k+1}}\langle \eta_{t_{k+1}},\partial_r\psi_r\rangle dr.
\end{align*}
Now, we let the mesh $|\pi|\to 0$. Since $t\mapsto \psi_t$ is continuous into $C_b^2$ and
$t\mapsto \eta_t$ is continuous in the sense that $t\mapsto \langle \eta_t,f\rangle$ is continuous for each fixed $f$,
we have the Riemann-sum convergences
\begin{equation*}
\sum_{k=0}^{n-1}\int_{t_k}^{t_{k+1}}\langle \eta_s,\mathcal G_s\psi_{t_k}\rangle ds
\longrightarrow
\int_0^t \langle \eta_s,\mathcal G_s\psi_s\rangle ds
\end{equation*}
and
\begin{equation*}
\sum_{k=0}^{n-1}\int_{t_k}^{t_{k+1}}\langle \eta_{t_{k+1}},\partial_r\psi_r\rangle dr
\longrightarrow
\int_0^t \langle \eta_s,\partial_s\psi_s\rangle ds.
\end{equation*}
On the other hand, the martingale sum converges to the stochastic
integral $\int_0^t dM_s(\psi_s)$ by standard arguments. Passing to the
limit yields \eqref{eq:pairing_ito} in integral form, which concludes
the proof.
\end{proof}

\section{Bounds for the Malliavin derivatives}
\label{sec:malliavin}

The goal of this section is to obtain appropriate uniform-in-time bounds for the first- and second-order Malliavin derivatives of the particle system $X_t^{i,N}$ for $i=1,\cdots,N$. In particular, we prove that under Assumption \ref{ass:coeff-level-new} the first- and second-order Malliavin derivatives of the particle system decay exponentially in time. Section \ref{S:BoundFirstMalliavinDerivatives} includes the calculations for the first-order Malliavin derivatives. Section \ref{S:BoundSecondMalliavinDerivatives} focuses on the second-order Malliavin derivatives.

We view $\left(B^1,\dots,B^N\right)$ as an isonormal Gaussian process over
\begin{equation*}
\mathfrak{H}_t=L^2\left(\left[0,t\right];\R^{mN}\right)\cong \bigoplus_{j=1}^N L^2\left(\left[0,t\right];\R^m\right),
\end{equation*}
with underlying space
$A=\left[0,t\right]\times\left\{1,\dots,N\right\}\times\left\{1,\dots,m\right\}$
and measure
\begin{equation*}
\lambda\left(ds,dj,d\alpha\right)=ds \sum_{j=1}^N\sum_{\alpha=1}^m\delta_{\left(j,\alpha\right)}.
\end{equation*}
For $s\in\left[0,t\right]$, $j\in\left\{1,\dots,N\right\}$ and $\alpha\in\left\{1,\dots,m\right\}$, we write $D_s^{j,\alpha}$ for the Malliavin derivative
with respect to the coordinate $B^{j,\alpha}$.

\subsection{Bounds for the first Malliavin derivatives}\label{S:BoundFirstMalliavinDerivatives}

For $i,j\in\left\{1,\dots,N\right\}$ and $0\le s\le u\le t$, we denote
the first Malliavin derivative matrix by
\begin{equation*}
D_s^{j}X_u^{i,N}=\left(D_s^{j,\alpha}X_u^{i,N}\right)_{\alpha=1}^m \in\R^{d\times m}.
\end{equation*}

\begin{lemma}[SDE for the first Malliavin derivative]
\label{lem:SDEU}
Assume that Assumptions \ref{ass:smooth}, and
\ref{ass:second-malliavin-uniform}(i) hold. Fix
$ 0 \le s \le u \le t $ and $ i,j \in \left\{ 1,\dots,N \right\} $. Then
$ D_s^j X_u^{i,N} $ satisfies, on $ [s,t] $,
\begin{align}
\label{eq:SDEU}
D_s^j X_u^{i,N}
&=
\sigma\left( X_s^{i,N},\mu_s^N \right)\mathbf{1}_{\left\{ i=j \right\}}
\nonumber\\
&\quad
+
\int_s^u
\Biggl(
\partial_x b\left( X_r^{i,N},\mu_r^N \right) D_s^j X_r^{i,N}
+
\frac1N \sum_{\ell=1}^N
\partial_\mu b\left( X_r^{i,N},\mu_r^N \right)\left( X_r^{\ell,N} \right)
D_s^j X_r^{\ell,N}
\Biggr) dr
\nonumber\\
&\quad
+
\int_s^u
\Biggl(
\partial_x \sigma\left( X_r^{i,N},\mu_r^N \right) D_s^j X_r^{i,N}
+
\frac1N \sum_{\ell=1}^N
\partial_\mu \sigma\left( X_r^{i,N},\mu_r^N \right)\left( X_r^{\ell,N} \right)
D_s^j X_r^{\ell,N}
\Biggr) dB_r^i.
\end{align}
Equivalently, for every $ \alpha \in \left\{ 1,\dots,m \right\} $, the column
$ D_s^{j,\alpha} X_u^{i,N} $ satisfies
\begin{align}
\label{eq:SDEU-column}
D_s^{j,\alpha} X_u^{i,N}
&=
\sigma\left( X_s^{i,N},\mu_s^N \right)e_\alpha \mathbf{1}_{\left\{ i=j \right\}}
\nonumber\\
&\quad
+
\int_s^u
\Biggl(
\partial_x b\left( X_r^{i,N},\mu_r^N \right) D_s^{j,\alpha} X_r^{i,N}
+
\frac1N \sum_{\ell=1}^N
\partial_\mu b\left( X_r^{i,N},\mu_r^N \right)\left( X_r^{\ell,N} \right)
D_s^{j,\alpha} X_r^{\ell,N}
\Biggr) dr
\nonumber\\
&\quad
+
\int_s^u
\Biggl(
\partial_x \sigma\left( X_r^{i,N},\mu_r^N \right) D_s^{j,\alpha} X_r^{i,N}
+
\frac1N \sum_{\ell=1}^N
\partial_\mu \sigma\left( X_r^{i,N},\mu_r^N \right)\left( X_r^{\ell,N} \right)
D_s^{j,\alpha} X_r^{\ell,N}
\Biggr) dB_r^i.
\end{align}
\end{lemma}

\begin{proof}
Using the particle system's integral form
\begin{equation*}
X_u^{i,N}
=
X_0^{i,N}
+
\int_0^u b\left( X_r^{i,N},\mu_r^N \right) dr
+
\int_0^u \sigma\left( X_r^{i,N},\mu_r^N \right) dB_r^i,
\end{equation*}
we fix $ \alpha \in \left\{ 1,\dots,m \right\} $ and apply the Malliavin derivative
$ D_s^{j,\alpha} $ to both sides. Since $ X_0^{i,N} $ is independent
of the Brownian motions, we have $D_s^{j,\alpha} X_0^{i,N} = 0$. For
the drift term, differentiation under the integral sign gives
\begin{equation*}
D_s^{j,\alpha}
\left(
\int_0^u b\left( X_r^{i,N},\mu_r^N \right) dr
\right)
=
\int_s^u
D_s^{j,\alpha} b\left( X_r^{i,N},\mu_r^N \right) dr,
\end{equation*}
because the Malliavin derivative vanishes for $ r < s $. By the chain rule on
$ \R^d \times \Ptwo $,
\begin{align*}
D_s^{j,\alpha} b\left( X_r^{i,N},\mu_r^N \right)
=
\partial_x b\left( X_r^{i,N},\mu_r^N \right) D_s^{j,\alpha} X_r^{i,N}
+
\frac1N \sum_{\ell=1}^N
\partial_\mu b\left( X_r^{i,N},\mu_r^N \right)\left( X_r^{\ell,N} \right)
D_s^{j,\alpha} X_r^{\ell,N}.
\end{align*}
We treat the stochastic integral similarly. The commutation relation between the
Malliavin derivative and the It\^o integral yields
\begin{align*}
D_s^{j,\alpha}
\left(
\int_0^u \sigma\left( X_r^{i,N},\mu_r^N \right) dB_r^i
\right)
=
\sigma\left( X_s^{i,N},\mu_s^N \right)e_\alpha \mathbf{1}_{\left\{ i=j \right\}}
+
\int_s^u
D_s^{j,\alpha} \sigma\left( X_r^{i,N},\mu_r^N \right) dB_r^i.
\end{align*}
Again by the chain rule,
\begin{align*}
D_s^{j,\alpha} \sigma\left( X_r^{i,N},\mu_r^N \right)
=
\partial_x \sigma\left( X_r^{i,N},\mu_r^N \right) D_s^{j,\alpha} X_r^{i,N}
+
\frac1N \sum_{\ell=1}^N
\partial_\mu \sigma\left( X_r^{i,N},\mu_r^N \right)\left( X_r^{\ell,N} \right)
D_s^{j,\alpha} X_r^{\ell,N}.
\end{align*}
Collecting the preceding identities proves \eqref{eq:SDEU-column}. Since
$ D_s^j X_u^{i,N} $ is the $ d \times m $ matrix whose $ \alpha $-th column is
$ D_s^{j,\alpha} X_u^{i,N} $, the matrix form \eqref{eq:SDEU} follows immediately.
\end{proof}

\begin{lemma}[Linear dissipative estimate]
\label{lem:linear-dissipative-estimate}
Fix $ p \ge 2 $, $ \kappa > 0 $, $ M \ge 0 $, and let $ 0 \le s \le u $. Let
$ \left( W_r \right)_{r \in [s,u]} $ be an $ m $-dimensional Brownian motion. Let
$ Y $ be an adapted continuous $ \R^d $-valued process satisfying
\begin{equation*}
Y_r
=
Y_s
+
\int_s^r
\left(
A_\ell Y_\ell + f_\ell
\right)
d\ell
+
\int_s^r
\left(
C_\ell\left[ Y_\ell \right] + g_\ell
\right)
dW_\ell,
\qquad
r \in [s,u],
\end{equation*}
where, for almost every $\ell \in [s,u] $, $ A_\ell$ is an $ \R^{d
  \times d}$-valued progressively measurable process, $ C_\ell $ is progressively measurable with values in the space of linear maps from
$ \R^d $ into $ \R^{d \times m} $, $ f_\ell $ is an $ \R^d $-valued
progressively measurable process, and $ g_\ell $ is an $ \R^{d \times
  m} $-valued progressively measurable process. Assume that, for all $ h \in \R^d $,
\begin{equation}
\label{eq:linear-dissipative-coercivity}
2 \left\langle h, A_\ell h \right\rangle
+
\left( p-1 \right)
\norm{C_\ell\left[ h \right]}^2
\le
-\kappa \abs{h}^2\quad\mbox{and}\quad
\norm{C_\ell\left[ h \right]}
\le
M \abs{h}.
\end{equation}
Assume also that
\begin{equation*}
\E\left( \abs{Y_s}^p \right)
+
\int_s^u
\E\left(
\abs{f_\ell}^p + \norm{g_\ell}^p
\right)
d\ell
<
\infty.
\end{equation*}
Then, it holds that
\begin{align}
\label{eq:linear-dissipative-estimate}
\E\left(
\abs{Y_u}^p
\right)^{1/p}
\le
e^{-\kappa \left( u-s \right)/4}
\E\left(
\abs{Y_s}^p
\right)^{1/p}
+
\Lambda_{p,\kappa,M}
\left(
\int_s^u
e^{-p \kappa \left( u-r \right)/4}
\E\left(
\abs{f_r}^p + \norm{g_r}^p
\right)
dr
\right)^{1/p},
\end{align}
where $\Lambda_{p,\kappa,M}$ and $\Xi_{p,\kappa}$ are the quantities
defined in Assumption \ref{ass:first-malliavin-uniform}.
\end{lemma}

\begin{proof}
Set $\Sigma_r
=
C_r\left[ Y_r \right] + g_r$, $r \in [s,u]$. We apply It\^o's formula to the function $ y \mapsto \abs{y}^p $. Since $ p \ge 2 $,
\begin{align*}
\nabla \left( \abs{y}^p \right)
=
p \abs{y}^{p-2} y\quad\mbox{and}\quad
D^2 \left( \abs{y}^p \right)
=
p \abs{y}^{p-2} I_d
+
p \left( p-2 \right) \abs{y}^{p-4} y \otimes y.
\end{align*}
Therefore,
\begin{align*}
d \abs{Y_r}^p
&=
p \abs{Y_r}^{p-2}
\left\langle Y_r, A_r Y_r + f_r \right\rangle
dr
+
p \abs{Y_r}^{p-2}
\left\langle Y_r, \Sigma_r dW_r \right\rangle
+
\frac{p}{2}
\abs{Y_r}^{p-2}
\norm{\Sigma_r}^2
  dr
  \nonumber\\
&\quad
+
\frac{p \left( p-2 \right)}{2}
\abs{Y_r}^{p-4}
\abs{\Sigma_r^\top Y_r}^2
dr.
\end{align*}
Since $\abs{\Sigma_r^\top Y_r}^2
\le
\abs{Y_r}^2 \norm{\Sigma_r}^2$, we get
\begin{align*}
d \abs{Y_r}^p
\le
p \abs{Y_r}^{p-2}
\left\langle Y_r, A_r Y_r + f_r \right\rangle
dr
+
p \abs{Y_r}^{p-2}
\left\langle Y_r, \Sigma_r dW_r \right\rangle
+
\frac{p \left( p-1 \right)}{2}
\abs{Y_r}^{p-2}
\norm{\Sigma_r}^2
dr.
\end{align*}
We now multiply by the fixed exponential weight
$ e^{p \kappa \left( r-s \right)/4} $. Using the product rule,
\begin{align*}
d \left(
e^{p \kappa \left( r-s \right)/4} \abs{Y_r}^p
\right)
&=
\frac{p \kappa}{4}
e^{p \kappa \left( r-s \right)/4}
\abs{Y_r}^p
dr
+
e^{p \kappa \left( r-s \right)/4}
d \abs{Y_r}^p
\nonumber\\
&\le
e^{p \kappa \left( r-s \right)/4}
\Biggl[
\frac{p \kappa}{4} \abs{Y_r}^p
+
p \abs{Y_r}^{p-2}
\left\langle Y_r, A_r Y_r + f_r \right\rangle
\Biggr]
dr
\nonumber\\
&\quad
+
\frac{p \left( p-1 \right)}{2}
e^{p \kappa \left( r-s \right)/4}
\abs{Y_r}^{p-2}
\norm{\Sigma_r}^2
dr
+
p e^{p \kappa \left( r-s \right)/4}
\abs{Y_r}^{p-2}
\left\langle Y_r, \Sigma_r dW_r \right\rangle .
\end{align*}
Next, using $\norm{\Sigma_r}^2
=
\norm{C_r\left[ Y_r \right] + g_r}^2
\le
\norm{C_r\left[ Y_r \right]}^2
+
2 \norm{C_r\left[ Y_r \right]} \norm{g_r}
+
\norm{g_r}^2$, we obtain
\begin{align*}
d \left(
e^{p \kappa \left( r-s \right)/4} \abs{Y_r}^p
\right)
&\le
e^{p \kappa \left( r-s \right)/4}
\Biggl[
\frac{p \kappa}{4} \abs{Y_r}^p
+
p \abs{Y_r}^{p-2}
\left\langle Y_r, A_r Y_r \right\rangle
+
\frac{p \left( p-1 \right)}{2}
\abs{Y_r}^{p-2}
\norm{C_r\left[ Y_r \right]}^2
\Biggr]
dr
\nonumber\\
&\quad
+
e^{p \kappa \left( r-s \right)/4}
\Biggl[
p \abs{Y_r}^{p-1} \abs{f_r}
+
p \left( p-1 \right)
\abs{Y_r}^{p-2}
\norm{C_r\left[ Y_r \right]} \norm{g_r}
\nonumber\\
&\quad
+
\frac{p \left( p-1 \right)}{2}
\abs{Y_r}^{p-2}
\norm{g_r}^2
\Biggr]
dr
+
p e^{p \kappa \left( r-s \right)/4}
\abs{Y_r}^{p-2}
\left\langle Y_r, \Sigma_r dW_r \right\rangle.
\end{align*}
By \eqref{eq:linear-dissipative-coercivity},
\begin{align*}
p \abs{Y_r}^{p-2}
\left\langle Y_r, A_r Y_r \right\rangle
+
\frac{p \left( p-1 \right)}{2}
\abs{Y_r}^{p-2}
\norm{C_r\left[ Y_r \right]}^2
&=
\frac{p}{2}
\abs{Y_r}^{p-2}
\left(
2 \left\langle Y_r, A_r Y_r \right\rangle
+
\left( p-1 \right)
\norm{C_r\left[ Y_r \right]}^2
\right)
\\
&
\le
-\frac{p \kappa}{2}
\abs{Y_r}^p.
\end{align*}
Hence, we can write
\begin{align*}
d \left(
e^{p \kappa \left( r-s \right)/4} \abs{Y_r}^p
\right)
&\le
-\frac{p \kappa}{4}
e^{p \kappa \left( r-s \right)/4}
\abs{Y_r}^p
dr
\nonumber\\
&\quad
+
e^{p \kappa \left( r-s \right)/4}
\Biggl[
p \abs{Y_r}^{p-1} \abs{f_r}
+
p \left( p-1 \right)
\abs{Y_r}^{p-2}
\norm{C_r\left[ Y_r \right]} \norm{g_r}
  \nonumber\\
  &\quad
+
\frac{p \left( p-1 \right)}{2}
\abs{Y_r}^{p-2}
\norm{g_r}^2
\Biggr]
dr
+
p e^{p \kappa \left( r-s \right)/4}
\abs{Y_r}^{p-2}
\left\langle Y_r, \Sigma_r dW_r \right\rangle .
\end{align*}
Using \eqref{eq:linear-dissipative-coercivity}, we have $\norm{C_r\left[ Y_r \right]}
\le
M \abs{Y_r}$, and therefore
\begin{align}
d \left(
e^{p \kappa \left( r-s \right)/4} \abs{Y_r}^p
\right)
&\le
-\frac{p \kappa}{4}
e^{p \kappa \left( r-s \right)/4}
\abs{Y_r}^p
dr
+
e^{p \kappa \left( r-s \right)/4}
\Biggl[
p \abs{Y_r}^{p-1} \abs{f_r}
+
p \left( p-1 \right) M
\abs{Y_r}^{p-1} \norm{g_r}
\nonumber\\
&\quad
+
\frac{p \left( p-1 \right)}{2}
\abs{Y_r}^{p-2}
\norm{g_r}^2
\Biggr]
dr
+
p e^{p \kappa \left( r-s \right)/4}
\abs{Y_r}^{p-2}
\left\langle Y_r, \Sigma_r dW_r \right\rangle .
\label{eq:linear-dissipative-proof-6}
\end{align}
We now estimate the three source terms separately. For the first term, Young's inequality with exponents
$ p/(p-1) $ and $ p $ gives, for every $ r $,
\begin{equation}
\label{eq:linear-dissipative-proof-young-1}
p \abs{Y_r}^{p-1} \abs{f_r}
\le
\frac{\kappa}{12} \abs{Y_r}^p
+
\left(
\frac{12 \left( p-1 \right)}{\kappa}
\right)^{p-1}
\abs{f_r}^p.
\end{equation}

Applying the same inequality with
$ \abs{f_r} $ replaced by $ \left( p-1 \right) M \norm{g_r} $, we get
\begin{equation}
\label{eq:linear-dissipative-proof-young-2}
p \left( p-1 \right) M \abs{Y_r}^{p-1} \norm{g_r}
\le
\frac{\kappa}{12} \abs{Y_r}^p
+
\left(
\frac{12 \left( p-1 \right)}{\kappa}
\right)^{p-1}
\left(
\left( p-1 \right) M
\right)^p
\norm{g_r}^p.
\end{equation}
If $ p > 2 $, Young's inequality with exponents
$ p/(p-2) $ and $ p/2 $ gives
\begin{equation}
\label{eq:linear-dissipative-proof-young-3}
\frac{p \left( p-1 \right)}{2}
\abs{Y_r}^{p-2}
\norm{g_r}^2
\le
\frac{\kappa}{12} \abs{Y_r}^p
+
\left( p-1 \right)
\left(
\frac{6 \left( p-1 \right)\left( p-2 \right)}{\kappa}
\right)^{\left( p-2 \right)/2}
\norm{g_r}^p.
\end{equation}
If $ p = 2 $, this term is simply
\begin{equation}
\label{eq:linear-dissipative-proof-young-3-bis}
\frac{p \left( p-1 \right)}{2}
\abs{Y_r}^{p-2}
\norm{g_r}^2
=
\norm{g_r}^2.
\end{equation}
Substituting
\eqref{eq:linear-dissipative-proof-young-1},
\eqref{eq:linear-dissipative-proof-young-2},
and either
\eqref{eq:linear-dissipative-proof-young-3}
or
\eqref{eq:linear-dissipative-proof-young-3-bis}
into \eqref{eq:linear-dissipative-proof-6}, the three
$ \kappa \abs{Y_r}^p / 12 $ contributions exactly absorb the negative term
$ - p \kappa \abs{Y_r}^p / 4 $. Hence
\begin{align}
  \label{eq:linear-dissipative-proof-7}
d \left(
e^{p \kappa \left( r-s \right)/4} \abs{Y_r}^p
\right)
\le
\Lambda_{p,\kappa,M}^p
e^{p \kappa \left( r-s \right)/4}
\left(
\abs{f_r}^p + \norm{g_r}^p
\right)
dr
+
p e^{p \kappa \left( r-s \right)/4}
\abs{Y_r}^{p-2}
\left\langle Y_r, \Sigma_r dW_r \right\rangle .
\end{align}
We now localize the stochastic integral. For $ n \ge 1 $, define
\begin{equation*}
\tau_n
=
\inf\left\{
r \in [s,u] \colon
\int_s^r
e^{p \kappa \left( \ell-s \right)/2}
\abs{Y_\ell}^{2p-4}
\abs{\Sigma_\ell^\top Y_\ell}^2
d\ell
\ge n
\right\}
\wedge u.
\end{equation*}
Then, the stopped stochastic integral in \eqref{eq:linear-dissipative-proof-7} is a true
martingale on $ [s,u] $. Integrating from $ s $ to $ \tau_n $ and taking expectations, we get
\begin{align*}
\E\left(
e^{p \kappa \left( \tau_n-s \right)/4}
\abs{Y_{\tau_n}}^p
\right)
\le
\E\left( \abs{Y_s}^p \right)
+
\Lambda_{p,\kappa,M}^p
\E\left(
\int_s^{\tau_n}
e^{p \kappa \left( r-s \right)/4}
\left(
\abs{f_r}^p + \norm{g_r}^p
\right)
dr
\right).
\end{align*}
Letting $ n \to \infty $ and applying Fatou's lemma on the left-hand
side as well as monotone convergence
on the right-hand side yield
\begin{align*}
\E\left(
e^{p \kappa \left( u-s \right)/4}
\abs{Y_u}^p
\right)
\le
\E\left( \abs{Y_s}^p \right)
+
\Lambda_{p,\kappa,M}^p
\int_s^u
e^{p \kappa \left( r-s \right)/4}
\E\left(
\abs{f_r}^p + \norm{g_r}^p
\right)
dr.
\end{align*}
Multiplying by $ e^{-p \kappa \left( u-s \right)/4} $ gives
\begin{align*}
\E\left(
\abs{Y_u}^p
\right)
\le
e^{-p \kappa \left( u-s \right)/4}
\E\left(
\abs{Y_s}^p
\right)
+
\Lambda_{p,\kappa,M}^p
\int_s^u
e^{-p \kappa \left( u-r \right)/4}
\E\left(
\abs{f_r}^p + \norm{g_r}^p
\right)
dr.
\end{align*}
Taking the $ p $-th root and using
$ \left( a+b \right)^{1/p} \le a^{1/p} + b^{1/p} $ yields
\eqref{eq:linear-dissipative-estimate}.
\end{proof}

\begin{proposition}[Uniform-in-time first Malliavin derivative bounds]
\label{prop:firstbounds-uniform}
Let $p\in\left\{4,8\right\}$. Assume
Assumptions~\ref{ass:smooth}, \ref{ass:second-malliavin-uniform}(i),
\ref{ass:coeff-level-new}(i)(ii), and
\ref{ass:first-malliavin-uniform}. Then, there exists a constant
$C_p>0$ such that, for all $N\ge 1$, all $0\le s\le u$, and all
$i,j\in\left\{1,\dots,N\right\}$,
\begin{equation}
\label{eq:Ubounds-uniform}
\E\left(
\norm{D_s^j X_u^{i,N}}^p
\right)^{1/p}
\le
C_p e^{-\kappa_p \left( u-s \right)/8}
\left(
\mathbf{1}_{\left\{ i=j \right\}}
+
\frac{1}{N}\mathbf{1}_{\left\{ i\ne j \right\}}
\right).
\end{equation}
In particular,
\begin{equation}
\label{eq:Ubounds-average-uniform}
\frac{1}{N}\sum_{i=1}^N
\E\left(
\norm{D_s^j X_u^{i,N}}^p
\right)^{1/p}
\le
\frac{C_p}{N} e^{-\kappa_p \left( u-s \right)/8}.
\end{equation}
\end{proposition}

\begin{proof}
Fix $ p\in\left\{4,8\right\} $, $ j \in \left\{ 1,\dots,N \right\} $,
$ 0 \le s \le u $, and
$ \alpha \in \left\{ 1,\dots,m \right\} $. For $ r \in \left[ s,u \right] $, define
\begin{equation*}
u_i\left( r \right)
=
\E\left(
\abs{D_s^{j,\alpha} X_r^{i,N}}^p
\right)^{1/p}\quad\mbox{and}\quad
\bar u\left( r \right)
=
\frac{1}{N}\sum_{\ell=1}^N u_\ell\left( r \right).
\end{equation*}
By Lemma~\ref{lem:SDEU}, the process $ D_s^{j,\alpha} X_r^{i,N} $ satisfies, on
$ \left[ s,u \right] $,
\begin{align*}
d\left( D_s^{j,\alpha} X_r^{i,N} \right)
&=
\Biggl(
\partial_x b\left( X_r^{i,N},\mu_r^N \right)
\left[ D_s^{j,\alpha} X_r^{i,N} \right]
+
G_r^{i,\alpha}
\Biggr) dr
+
\Biggl(
\partial_x \sigma\left( X_r^{i,N},\mu_r^N \right)
\left[ D_s^{j,\alpha} X_r^{i,N} \right]
+
H_r^{i,\alpha}
\Biggr) dB_r^i,
\end{align*}
where
\begin{align*}
G_r^{i,\alpha}
=
\frac{1}{N}\sum_{\ell=1}^N
\partial_\mu b\left( X_r^{i,N},\mu_r^N \right)\left( X_r^{\ell,N} \right)
D_s^{j,\alpha} X_r^{\ell,N}
\end{align*}
and $H_r^{i,\alpha}$ is the same with $b$ replaced by $\sigma$,
with initial condition $D_s^{j,\alpha} X_s^{i,N}
=
\mathbf{1}_{\left\{ i=j \right\}}
\sigma\left( X_s^{i,N},\mu_s^N \right)e_\alpha$. We first estimate the forcing terms. By Minkowski's inequality and the measure-dependence bound
\eqref{eq:small-measure-dependence}, we can write
\begin{align*}
\E\left(
\abs{G_r^{i,\alpha}}^p
\right)^{1/p}
&\le
\frac{1}{N}\sum_{\ell=1}^N
\E\left(
\abs{
\partial_\mu b\left( X_r^{i,N},\mu_r^N \right)\left( X_r^{\ell,N} \right)
D_s^{j,\alpha} X_r^{\ell,N}
}^p
\right)^{1/p}
\nonumber\\
&\le
\frac{\gamma}{N}\sum_{\ell=1}^N
\E\left(
\abs{D_s^{j,\alpha} X_r^{\ell,N}}^p
\right)^{1/p}
=
\gamma \bar u\left( r \right).
\end{align*}
Similarly, we have
\begin{align*}
\E\left(
\norm{H_r^{i,\alpha}}^p
\right)^{1/p}
&\le
\frac{1}{N}\sum_{\ell=1}^N
\E\left(
\norm{
\partial_\mu \sigma\left( X_r^{i,N},\mu_r^N \right)\left( X_r^{\ell,N} \right)
D_s^{j,\alpha} X_r^{\ell,N}
}^p
\right)^{1/p}
\nonumber\\
&\le
\frac{\gamma}{N}\sum_{\ell=1}^N
\E\left(
\abs{D_s^{j,\alpha} X_r^{\ell,N}}^p
\right)^{1/p}
=
\gamma \bar u\left( r \right).
\end{align*}
Hence,
\begin{equation}
\label{eq:first-malliavin-forcing-bounds}
\E\left(
\abs{G_r^{i,\alpha}}^p + \norm{H_r^{i,\alpha}}^p
\right)^{1/p}
\le
2^{1/p} \gamma \bar u\left( r \right).
\end{equation}
Now, by Assumption~\ref{ass:coeff-level-new}(i) and the definition of $ M_\sigma $,
the coefficients $A_r = \partial_x b\left( X_r^{i,N},\mu_r^N \right)$
and $C_r = \partial_x \sigma\left( X_r^{i,N},\mu_r^N \right)$ satisfy the assumptions of Lemma~\ref{lem:linear-dissipative-estimate} with
$ \kappa = \kappa_p $ and $ M = M_\sigma $. Applying that lemma to
$Y_r = D_s^{j,\alpha} X_r^{i,N}$, $f_r = G_r^{i,\alpha}$ and $g_r =
H_r^{i,\alpha}$, we obtain
\begin{align*}
u_i\left( u \right)
&\le
e^{-\kappa_p \left( u-s \right)/4}
\E\left(
\abs{D_s^{j,\alpha} X_s^{i,N}}^p
\right)^{1/p}
\nonumber\\
&\quad
+
\Lambda_{p,\kappa_p,M_\sigma}
\left(
\int_s^u
e^{-p\kappa_p \left( u-r \right)/4}
\E\left(
\abs{G_r^{i,\alpha}}^p + \norm{H_r^{i,\alpha}}^p
\right)
dr
\right)^{1/p}.
\end{align*}
Since $ \sigma $ is bounded, we have $\E\left(
\abs{D_s^{j,\alpha} X_s^{i,N}}^p
\right)^{1/p}
\le
\norm{\sigma}_{\infty}\mathbf{1}_{\left\{ i=j \right\}}$, and combining this with \eqref{eq:first-malliavin-forcing-bounds}, we get
\begin{equation}
\label{eq:first-malliavin-u-i-2}
u_i\left( u \right)
\le
\norm{\sigma}_{\infty}
e^{-\kappa_p \left( u-s \right)/4}
\mathbf{1}_{\left\{ i=j \right\}}
+
2^{1/p}\Lambda_{p,\kappa_p,M_\sigma} \gamma
\left(
\int_s^u
e^{-p\kappa_p \left( u-r \right)/4}
\bar u\left( r \right)^p
dr
\right)^{1/p}.
\end{equation}
Define
\begin{equation*}
\eta_p
=
2^{1/p}\Lambda_{p,\kappa_p,M_\sigma} \gamma
\left(
\frac{8}{p\kappa_p}
\right)^{1/p}.
\end{equation*}
By Assumption~\ref{ass:first-malliavin-uniform}, applied with the present value of
$p\in\left\{4,8\right\}$, we have $\eta_p<1$. For $ v \in [s,u] $, define
\begin{equation*}
M_v
=
\sup_{r \in [s,v]}
e^{\kappa_p \left( r-s \right)/8}
\bar u\left( r \right).
\end{equation*}
Then, for every $ r \in [s,u] $, $\bar u\left( r \right)
\le
e^{-\kappa_p \left( r-s \right)/8}
M_u$. Substituting this into \eqref{eq:first-malliavin-u-i-2}, we obtain
\begin{align*}
u_i\left( u \right)
&\le
\norm{\sigma}_{\infty}
e^{-\kappa_p \left( u-s \right)/4}
\mathbf{1}_{\left\{ i=j \right\}}
+
2^{1/p}\Lambda_{p,\kappa_p,M_\sigma} \gamma
\left(
\int_s^u
e^{-p\kappa_p \left( u-r \right)/4}
\bar u\left( r \right)^p
dr
\right)^{1/p}
\\
&\le
\norm{\sigma}_{\infty}
e^{-\kappa_p \left( u-s \right)/4}
\mathbf{1}_{\left\{ i=j \right\}}
+
2^{1/p}\Lambda_{p,\kappa_p,M_\sigma} \gamma
\left(
\int_s^u
e^{-p\kappa_p \left( u-r \right)/4}
e^{-p\kappa_p \left( r-s \right)/8}
dr
\right)^{1/p}
M_u
\\
&=
\norm{\sigma}_{\infty}
e^{-\kappa_p \left( u-s \right)/4}
\mathbf{1}_{\left\{ i=j \right\}}
+
2^{1/p}\Lambda_{p,\kappa_p,M_\sigma} \gamma
e^{-\kappa_p \left( u-s \right)/8}
\left(
\int_s^u
e^{-p\kappa_p \left( u-r \right)/8}
dr
\right)^{1/p}
M_u
\\
&\le
\norm{\sigma}_{\infty}
e^{-\kappa_p \left( u-s \right)/4}
\mathbf{1}_{\left\{ i=j \right\}}
+
2^{1/p}\Lambda_{p,\kappa_p,M_\sigma} \gamma
e^{-\kappa_p \left( u-s \right)/8}
\left(
\frac{8}{p\kappa_p}
\right)^{1/p}
M_u
\\
&\le
e^{-\kappa_p \left( u-s \right)/8}
\left(
\norm{\sigma}_{\infty}\mathbf{1}_{\left\{ i=j \right\}}
+
\eta_p M_u
\right).
\end{align*}
Averaging over $ i $ yields
\begin{equation*}
\bar u\left( u \right)
\le
e^{-\kappa_p \left( u-s \right)/8}
\left(
\frac{\norm{\sigma}_{\infty}}{N}
+
\eta_p M_u
\right).
\end{equation*}
Taking the supremum over $ u \in [s,v] $, we get
\begin{equation*}
M_v
\le
\frac{\norm{\sigma}_{\infty}}{N}
+
\eta_p M_v.
\end{equation*}
Using the fact that $\eta_p <
1$, we conclude that for any $v \in [s,u]$,
\begin{equation}
\label{eq:first-malliavin-M-bound}
M_v
\le
\frac{\norm{\sigma}_{\infty}}{\left( 1-\eta_p \right)N}.
\end{equation}
Substituting \eqref{eq:first-malliavin-M-bound} back into the estimate for
$ u_i\left( u \right) $, we find
\begin{equation*}
u_i\left( u \right)
\le
e^{-\kappa_p \left( u-s \right)/8}
\left(
\norm{\sigma}_{\infty}\mathbf{1}_{\left\{ i=j \right\}}
+
\frac{\eta_p \norm{\sigma}_{\infty}}{\left( 1-\eta_p \right)N}
\right).
\end{equation*}
Since $ N \ge 1 $, the second term on the right-hand side is bounded by a constant
times $ \mathbf{1}_{\left\{ i=j \right\}} $ when $ i=j $, and is of order $ 1/N $ when
$ i\ne j $. After enlarging the constant, we obtain
\begin{equation}
\label{eq:first-malliavin-column-bound}
u_i\left( u \right)
\le
C_p e^{-\kappa_p \left( u-s \right)/8}
\left(
\mathbf{1}_{\left\{ i=j \right\}}
+
\frac{1}{N}\mathbf{1}_{\left\{ i\ne j \right\}}
\right).
\end{equation}
This proves the desired estimate for each fixed column
$ D_s^{j,\alpha} X_u^{i,N} $. Finally, by Minkowski's inequality,
\begin{equation*}
\E\left(
\norm{D_s^j X_u^{i,N}}^p
\right)^{1/p}
\le
\sum_{\alpha=1}^m
\E\left(
\abs{D_s^{j,\alpha} X_u^{i,N}}^p
\right)^{1/p},
\end{equation*}
and \eqref{eq:Ubounds-uniform} follows from
\eqref{eq:first-malliavin-column-bound} after enlarging the constant once more.
Averaging over $ i $ then gives \eqref{eq:Ubounds-average-uniform}.
\end{proof}

\subsection{Bounds for the second Malliavin derivatives}\label{S:BoundSecondMalliavinDerivatives}

Fix $0\le r\le s\le t$, $u\in\left[s,t\right]$, $i,j,k\in\left\{1,\dots,N\right\}$ and $\alpha,\beta\in\left\{1,\dots,m\right\}$.
We now study the second Malliavin derivative
$D_r^{k,\beta}D_s^{j,\alpha}X_u^{i,N}$. For an integer $p\geq2$, let
$\kappa_p>0$ be the constant from Assumption~\ref{ass:coeff-level-new}(i).
For the purposes of this section, we shall define the constant
\begin{equation}
\label{eq:def-omega-second-malliavin}
\hat{\omega}
=
\frac18 \min\left\{ \kappa_4,\kappa_8 \right\}.
\end{equation}

\begin{lemma}[SDE for the second Malliavin derivative]
\label{lem:secondSDE}
Assume that Assumptions \ref{ass:smooth},
and \ref{ass:second-malliavin-uniform} hold. Fix
$ 0 \le r \le s \le t $, $ u \in [s,t] $,
$ i,j,k \in \left\{ 1,\dots,N \right\} $, and
$ \alpha,\beta \in \left\{ 1,\dots,m \right\} $.
Then, the process $\left(
D_r^{k,\beta} D_s^{j,\alpha} X_u^{i,N}
\right)_{u \in [s,t]}$ solves the SDE
\begin{align}
\label{eq:secondSDE}
D_r^{k,\beta} D_s^{j,\alpha} X_u^{i,N}
&=
\mathcal I_{r,s}^{i,j,k,\alpha,\beta}
\nonumber\\
&\quad
+
\int_s^u
\left(
\partial_x b\left( X_\rho^{i,N},\mu_\rho^N \right)
D_r^{k,\beta} D_s^{j,\alpha} X_\rho^{i,N}
+
G_{r,s}^{i,j,k,\alpha,\beta}(\rho)
+
R_{r,s}^{i,j,k,\alpha,\beta}(\rho)
\right)
d\rho
\nonumber\\
&\quad
+
\int_s^u
\left(
\partial_x \sigma\left( X_\rho^{i,N},\mu_\rho^N \right)
D_r^{k,\beta} D_s^{j,\alpha} X_\rho^{i,N}
+
H_{r,s}^{i,j,k,\alpha,\beta}(\rho)
+
S_{r,s}^{i,j,k,\alpha,\beta}(\rho)
\right)
dB_\rho^i ,
\end{align}
where $\mathcal I_{r,s}^{i,j,k,\alpha,\beta}$ denotes the initial
condition given by
\begin{align*}
\mathcal I_{r,s}^{i,j,k,\alpha,\beta}
&=
\mathbf{1}_{\left\{ i=j \right\}}
\partial_x \sigma\left( X_s^{i,N},\mu_s^N \right)
D_r^{k,\beta} X_s^{i,N}
e_\alpha
+
\mathbf{1}_{\left\{ r=s \right\}}
\mathbf{1}_{\left\{ i=k \right\}}
\partial_x \sigma\left( X_s^{i,N},\mu_s^N \right)
D_s^{j,\alpha} X_s^{i,N}e_\beta
\nonumber\\
&\quad
+
\mathbf{1}_{\left\{ i=j \right\}} \frac1N \sum_{\ell=1}^N
\partial_\mu \sigma\left( X_s^{i,N},\mu_s^N \right)\left( X_s^{\ell,N} \right)
D_r^{k,\beta} X_s^{\ell,N} e_\alpha
\nonumber\\
&\quad
+
\mathbf{1}_{\left\{ r=s \right\}}
\mathbf{1}_{\left\{ i=k \right\}}
\frac1N \sum_{\ell=1}^N
\partial_\mu \sigma\left( X_s^{i,N},\mu_s^N \right)\left( X_s^{\ell,N} \right)
D_s^{j,\alpha} X_s^{\ell,N} e_\beta ,
\end{align*}
with $ e_\alpha $ denoting the $ \alpha $-th vector of the canonical basis of $ \R^m $,
\begin{align}
\label{eq:defGsecond}
G_{r,s}^{i,j,k,\alpha,\beta}(\rho)
&=
\frac1N \sum_{\ell=1}^N
\partial_\mu b\left( X_\rho^{i,N},\mu_\rho^N \right)\left( X_\rho^{\ell,N} \right)
D_r^{k,\beta} D_s^{j,\alpha} X_\rho^{\ell,N},
\end{align}
and $H_{r,s}^{i,j,k,\alpha,\beta}(\rho) $ is defined by the same formula, with
$ b $ replaced by $ \sigma $, while
\begin{align}
\label{eq:defRsecond}
R_{r,s}^{i,j,k,\alpha,\beta}(\rho)
&=
\partial_{xx}^2 b\left( X_\rho^{i,N},\mu_\rho^N \right)
\left[
D_r^{k,\beta} X_\rho^{i,N}, D_s^{j,\alpha} X_\rho^{i,N}
\right]
\nonumber\\
&\quad
+
\frac1N \sum_{\ell=1}^N
\partial_\mu\left[ \partial_x b \right]
\left( X_\rho^{i,N},\mu_\rho^N \right)\left( X_\rho^{\ell,N} \right)
\left[
D_r^{k,\beta} X_\rho^{\ell,N}, D_s^{j,\alpha} X_\rho^{i,N}
\right]
\nonumber\\
&\quad
+
\frac1N \sum_{\ell=1}^N
\partial_x\left[
\partial_\mu b\left( \cdot,\mu_\rho^N \right)\left( X_\rho^{\ell,N} \right)
\right]\left( X_\rho^{i,N} \right)
\left[
D_r^{k,\beta} X_\rho^{i,N}, D_s^{j,\alpha} X_\rho^{\ell,N}
\right]
\nonumber\\
&\quad
+
\frac1N \sum_{\ell=1}^N
\partial_v\left[
\partial_\mu b\left( X_\rho^{i,N},\mu_\rho^N \right)
\right]\left( X_\rho^{\ell,N} \right)
\left[
D_r^{k,\beta} X_\rho^{\ell,N}, D_s^{j,\alpha} X_\rho^{\ell,N}
\right]
\nonumber\\
&\quad
+
\frac1{N^2}\sum_{\ell=1}^N \sum_{q=1}^N
\partial_{\mu\mu}^2 b\left( X_\rho^{i,N},\mu_\rho^N \right)
\left( X_\rho^{\ell,N}, X_\rho^{q,N} \right)
\left[
D_r^{k,\beta} X_\rho^{q,N}, D_s^{j,\alpha} X_\rho^{\ell,N}
\right],
\end{align}
and $ S_{r,s}^{i,j,k,\alpha,\beta}(\rho) $ is defined by the same formula, with
$ b $ replaced by $ \sigma $. Finally, there exists a constant $ C > 0 $, independent of
$ r,s,t,\rho,N,i,j,k,\alpha,\beta $, such that, for every $ \rho \in [s,t] $,
\begin{align}
\label{eq:secondRSbound}
\norm{R_{r,s}^{i,j,k,\alpha,\beta}(\rho)}
+
\norm{S_{r,s}^{i,j,k,\alpha,\beta}(\rho)}
&\le
C \Bigg(
\norm{D_r^{k,\beta} X_\rho^{i,N}}
\norm{D_s^{j,\alpha} X_\rho^{i,N}}
+
\norm{D_s^{j,\alpha} X_\rho^{i,N}}
\frac1N \sum_{\ell=1}^N
\norm{D_r^{k,\beta} X_\rho^{\ell,N}}
\nonumber\\
&\quad
+
\norm{D_r^{k,\beta} X_\rho^{i,N}}
\frac1N \sum_{\ell=1}^N
\norm{D_s^{j,\alpha} X_\rho^{\ell,N}}
+
\frac1N \sum_{\ell=1}^N
\norm{D_r^{k,\beta} X_\rho^{\ell,N}}
\norm{D_s^{j,\alpha} X_\rho^{\ell,N}}
\nonumber\\
&\quad
+
\left(
\frac1N \sum_{\ell=1}^N
\norm{D_r^{k,\beta} X_\rho^{\ell,N}}
\right)
\left(
\frac1N \sum_{\ell=1}^N
\norm{D_s^{j,\alpha} X_\rho^{\ell,N}}
\right)
\Bigg).
\end{align}
\end{lemma}

\begin{proof}
For $ \rho \in [s,t] $, let
\begin{align*}
A_\rho^{i,j,\alpha}
&=
\partial_x b\left( X_\rho^{i,N},\mu_\rho^N \right)
D_s^{j,\alpha} X_\rho^{i,N}
+
\frac1N \sum_{\ell=1}^N
\partial_\mu b\left( X_\rho^{i,N},\mu_\rho^N \right)\left( X_\rho^{\ell,N} \right)
D_s^{j,\alpha} X_\rho^{\ell,N}
\end{align*}
and
\begin{align*}
C_\rho^{i,j,\alpha}
&=
\partial_x \sigma\left( X_\rho^{i,N},\mu_\rho^N \right)
D_s^{j,\alpha} X_\rho^{i,N}
+
\frac1N \sum_{\ell=1}^N
\partial_\mu \sigma\left( X_\rho^{i,N},\mu_\rho^N \right)\left( X_\rho^{\ell,N} \right)
D_s^{j,\alpha} X_\rho^{\ell,N}.
\end{align*}
Then, Lemma~\ref{lem:SDEU} gives
\begin{equation}
\label{eq:firstcomponentwise}
D_s^{j,\alpha} X_u^{i,N}
=
\mathbf{1}_{\left\{ i=j \right\}}
\sigma\left( X_s^{i,N},\mu_s^N \right)e_\alpha
+
\int_s^u A_\rho^{i,j,\alpha} d\rho
+
\int_s^u C_\rho^{i,j,\alpha} dB_\rho^i.
\end{equation}
Applying $ D_r^{k,\beta} $ to \eqref{eq:firstcomponentwise}, using differentiation under
the Lebesgue integral sign and the commutation formula for the stochastic integral
(because $ r \le s $), we obtain
\begin{align}
\label{eq:secondpre}
D_r^{k,\beta} D_s^{j,\alpha} X_u^{i,N}
&=
\mathbf{1}_{\left\{ i=j \right\}}
D_r^{k,\beta}
\left(
\sigma\left( X_s^{i,N},\mu_s^N \right)e_\alpha
\right)
\nonumber\\
&\quad
+
\mathbf{1}_{\left\{ r=s \right\}}
\mathbf{1}_{\left\{ i=k \right\}}
C_s^{i,j,\alpha} e_\beta
+
\int_s^u D_r^{k,\beta} A_\rho^{i,j,\alpha} d\rho
+
\int_s^u D_r^{k,\beta} C_\rho^{i,j,\alpha} dB_\rho^i.
\end{align}
Applying the Malliavin chain rule to the map
$ (x,\mu) \mapsto \sigma(x,\mu)e_\alpha $ yields the initial condition. It remains to compute
$ D_r^{k,\beta} A_\rho^{i,j,\alpha} $ and $ D_r^{k,\beta}
C_\rho^{i,j,\alpha} $. For the drift term, the product rule gives
\begin{align}
\label{eq:driftprod}
D_r^{k,\beta} A_\rho^{i,j,\alpha}
&=
D_r^{k,\beta}
\left(
\partial_x b\left( X_\rho^{i,N},\mu_\rho^N \right)
\right)
D_s^{j,\alpha} X_\rho^{i,N}
+
\partial_x b\left( X_\rho^{i,N},\mu_\rho^N \right)
D_r^{k,\beta} D_s^{j,\alpha} X_\rho^{i,N}
\nonumber\\
&\quad
+
\frac1N \sum_{\ell=1}^N
D_r^{k,\beta}
\left(
\partial_\mu b\left( X_\rho^{i,N},\mu_\rho^N \right)\left( X_\rho^{\ell,N} \right)
\right)
D_s^{j,\alpha} X_\rho^{\ell,N}
\nonumber\\
&\quad
+
\frac1N \sum_{\ell=1}^N
\partial_\mu b\left( X_\rho^{i,N},\mu_\rho^N \right)\left( X_\rho^{\ell,N} \right)
D_r^{k,\beta} D_s^{j,\alpha} X_\rho^{\ell,N}.
\end{align}
The chain rule on $ \R^d \times \Ptwo $ yields
\begin{align*}
D_r^{k,\beta}
\left(
\partial_x b\left( X_\rho^{i,N},\mu_\rho^N \right)
\right)
D_s^{j,\alpha} X_\rho^{i,N}
&=
\partial_{xx}^2 b\left( X_\rho^{i,N},\mu_\rho^N \right)
\left[
D_r^{k,\beta} X_\rho^{i,N}, D_s^{j,\alpha} X_\rho^{i,N}
\right]
\nonumber\\
&\quad
+
\frac1N \sum_{\ell=1}^N
\partial_\mu\left[ \partial_x b \right]
\left( X_\rho^{i,N},\mu_\rho^N \right)\left( X_\rho^{\ell,N} \right)
\left[
D_r^{k,\beta} X_\rho^{\ell,N}, D_s^{j,\alpha} X_\rho^{i,N}
\right],
\end{align*}
and
\begin{align*}
D_r^{k,\beta}
\left(
\partial_\mu b\left( X_\rho^{i,N},\mu_\rho^N \right)\left( X_\rho^{\ell,N} \right)
\right)
D_s^{j,\alpha} X_\rho^{\ell,N}
&=
\partial_x\left[
\partial_\mu b\left( \cdot,\mu_\rho^N \right)\left( X_\rho^{\ell,N} \right)
\right]\left( X_\rho^{i,N} \right)
\left[
D_r^{k,\beta} X_\rho^{i,N}, D_s^{j,\alpha} X_\rho^{\ell,N}
\right]
\nonumber\\
&\quad
+
\partial_v\left[
\partial_\mu b\left( X_\rho^{i,N},\mu_\rho^N \right)
\right]\left( X_\rho^{\ell,N} \right)
\left[
D_r^{k,\beta} X_\rho^{\ell,N}, D_s^{j,\alpha} X_\rho^{\ell,N}
\right]
\nonumber\\
&\quad
+
\frac1N \sum_{q=1}^N
\partial_{\mu\mu}^2 b\left( X_\rho^{i,N},\mu_\rho^N \right)
\left( X_\rho^{\ell,N}, X_\rho^{q,N} \right)
\left[
D_r^{k,\beta} X_\rho^{q,N}, D_s^{j,\alpha} X_\rho^{\ell,N}
\right].
\end{align*}
Substituting the previous two equalities into \eqref{eq:driftprod}, we find
\begin{equation}
\label{eq:expanddrift}
D_r^{k,\beta} A_\rho^{i,j,\alpha}
=
\partial_x b\left( X_\rho^{i,N},\mu_\rho^N \right)
D_r^{k,\beta} D_s^{j,\alpha} X_\rho^{i,N}
+
G_{r,s}^{i,j,k,\alpha,\beta}(\rho)
+
R_{r,s}^{i,j,k,\alpha,\beta}(\rho).
\end{equation}
The same computation with $ b $ replaced by $ \sigma $ gives
\begin{equation}
\label{eq:expanddiff}
D_r^{k,\beta} C_\rho^{i,j,\alpha}
=
\partial_x \sigma\left( X_\rho^{i,N},\mu_\rho^N \right)
D_r^{k,\beta} D_s^{j,\alpha} X_\rho^{i,N}
+
H_{r,s}^{i,j,k,\alpha,\beta}(\rho)
+
S_{r,s}^{i,j,k,\alpha,\beta}(\rho).
\end{equation}
Inserting \eqref{eq:expanddrift} and \eqref{eq:expanddiff} into \eqref{eq:secondpre}
gives \eqref{eq:secondSDE}. Finally, Assumption~\ref{ass:second-malliavin-uniform} implies that every second
derivative of $ b $ and $ \sigma $ appearing in \eqref{eq:defRsecond} is uniformly bounded.
Hence each term in $ R_{r,s}^{i,j,k,\alpha,\beta}(\rho) $ is bounded, up to a universal
constant, by one of the five quantities appearing on the right-hand side of
\eqref{eq:secondRSbound}. The same holds for
$ S_{r,s}^{i,j,k,\alpha,\beta}(\rho) $. This proves \eqref{eq:secondRSbound}.
\end{proof}

\begin{lemma}
\label{lem:secondSDE-uniform-local}
Assume Assumption~\ref{ass:second-malliavin-uniform}. Let $\hat{\omega}$ be the constant defined via \eqref{eq:def-omega-second-malliavin}. Fix
$ 0 \le r \le s \le u $, $ i,j,k \in \left\{ 1,\dots,N \right\} $, and
$ \alpha,\beta \in \left\{ 1,\dots,m \right\} $. Define, for every $\rho \in [s,u]$,
\begin{equation*}
Y_\rho^{q,j,k,\alpha,\beta}
=
D_r^{k,\beta} D_s^{j,\alpha} X_\rho^{q,N}\quad\mbox{and}\quad
\bar v(\rho)
=
\frac1N \sum_{q=1}^N
\E\left(
\norm{Y_\rho^{q,j,k,\alpha,\beta}}^4
\right)^{1/4}.
\end{equation*}
Then, there exists a constant $ C > 0 $, depending only on the constants in
Assumption~\ref{ass:second-malliavin-uniform}, such that the following bounds hold.

\begin{enumerate}
\item[(i)] The initial condition in Lemma~\ref{lem:secondSDE} satisfies
\begin{align}
\label{eq:second-initial-bound-uniform-local}
&\E\left(\norm{
\mathcal I_{r,s}^{i,j,k,\alpha,\beta}
}^4
\right)^{1/4}
\le
C e^{-\hat{\omega} (s-r)}
\left(
\mathbf{1}_{\left\{ i=j=k \right\}}
+
\frac1N
\left(
\mathbf{1}_{\left\{ i=j \right\}}
+
\mathbf{1}_{\left\{ i=k \right\}}
\right)
+
\frac1{N^2}
\right).
\end{align}

\item[(ii)] For every $ \rho \in [s,u] $,
\begin{equation}
\label{eq:second-GH-bound-uniform-local}
\E\left(
\abs{G_{r,s}^{i,j,k,\alpha,\beta}(\rho)}^4
+
\norm{H_{r,s}^{i,j,k,\alpha,\beta}(\rho)}^4
\right)^{1/4}
\le
2^{1/4}\gamma \bar v(\rho).
\end{equation}

\item[(iii)] For every $ \rho \in [s,u] $,
\begin{align}
\label{eq:second-RS-bound-uniform-local}
&\E\left(
\norm{R_{r,s}^{i,j,k,\alpha,\beta}(\rho)}^4
+
\norm{S_{r,s}^{i,j,k,\alpha,\beta}(\rho)}^4
  \right)^{1/4}\nonumber \\
  &\qquad\qquad\qquad\qquad
\le
C e^{-\hat{\omega} (\rho-r)}
\left(
\mathbf{1}_{\left\{ i=j=k \right\}}
+
\frac1N
\left(
\mathbf{1}_{\left\{ i=j \right\}}
+
\mathbf{1}_{\left\{ i=k \right\}}
+
\mathbf{1}_{\left\{ j=k \right\}}
\right)
+
\frac1{N^2}
\right).
\end{align}
\end{enumerate}
\end{lemma}

\begin{proof}
We first prove \eqref{eq:second-GH-bound-uniform-local}. By \eqref{eq:defGsecond},
Minkowski's inequality, and the definition of $ \gamma $,
\begin{align*}
\E\left(
\abs{G_{r,s}^{i,j,k,\alpha,\beta}(\rho)}^4
\right)^{1/4}
&\le
\frac1N \sum_{\ell=1}^N
\E\left(
\abs{
\partial_\mu b\left( X_\rho^{i,N},\mu_\rho^N \right)\left( X_\rho^{\ell,N} \right)
Y_\rho^{\ell,j,k,\alpha,\beta}
}^4
\right)^{1/4}
\\
&\le
\frac{\gamma}{N} \sum_{\ell=1}^N
\E\left(
\norm{Y_\rho^{\ell,j,k,\alpha,\beta}}^4
\right)^{1/4}
=
\gamma \bar v(\rho).
\end{align*}
The same argument yields
\begin{equation*}
\E\left(
\norm{H_{r,s}^{i,j,k,\alpha,\beta}(\rho)}^4
\right)^{1/4}
\le
\gamma \bar v(\rho).
\end{equation*}
Therefore,
\begin{equation*}
\E\left(
\abs{G_{r,s}^{i,j,k,\alpha,\beta}(\rho)}^4
+
\norm{H_{r,s}^{i,j,k,\alpha,\beta}(\rho)}^4
\right)^{1/4}
\le
2^{1/4}\gamma \bar v(\rho).
\end{equation*}
We next prove \eqref{eq:second-initial-bound-uniform-local}. By Minkowski's inequality and the boundedness of
$ \partial_x \sigma $ and $ \partial_\mu \sigma $, we get
\begin{align*}
\E\left(
\norm{
\mathcal I_{r,s}^{i,j,k,\alpha,\beta}
}^4
\right)^{1/4}&\le
C \Biggl(
\mathbf{1}_{\left\{ i=j \right\}}
\E\left(
\norm{D_r^{k,\beta} X_s^{i,N}}^4
\right)^{1/4}
+
\mathbf{1}_{\left\{ i=j \right\}}
\frac1N \sum_{\ell=1}^N
\E\left(
\norm{D_r^{k,\beta} X_s^{\ell,N}}^4
\right)^{1/4}
\\
&\quad
+
\mathbf{1}_{\left\{ r=s \right\}}
\mathbf{1}_{\left\{ i=k \right\}}
\E\left(
\norm{D_s^{j,\alpha} X_s^{i,N}}^4
\right)^{1/4}
+
\mathbf{1}_{\left\{ r=s \right\}}
\mathbf{1}_{\left\{ i=k \right\}}
\frac1N \sum_{\ell=1}^N
\E\left(
\norm{D_s^{j,\alpha} X_s^{\ell,N}}^4
\right)^{1/4}
\Biggr).
\end{align*}
Applying Proposition~\ref{prop:firstbounds-uniform} with $ p = 4 $
proves \eqref{eq:second-initial-bound-uniform-local}. Finally, we prove \eqref{eq:second-RS-bound-uniform-local}. By \eqref{eq:secondRSbound},
\begin{align*}
\norm{R_{r,s}^{i,j,k,\alpha,\beta}(\rho)}
+
\norm{S_{r,s}^{i,j,k,\alpha,\beta}(\rho)}&\le
C \Biggl(
\norm{D_r^{k,\beta} X_\rho^{i,N}}
\norm{D_s^{j,\alpha} X_\rho^{i,N}}
+
\norm{D_s^{j,\alpha} X_\rho^{i,N}}
\frac1N \sum_{\ell=1}^N
\norm{D_r^{k,\beta} X_\rho^{\ell,N}}
\\
&\quad
+
\norm{D_r^{k,\beta} X_\rho^{i,N}}
\frac1N \sum_{\ell=1}^N
  \norm{D_s^{j,\alpha} X_\rho^{\ell,N}}
+
\frac1N \sum_{\ell=1}^N
\norm{D_r^{k,\beta} X_\rho^{\ell,N}}
\norm{D_s^{j,\alpha} X_\rho^{\ell,N}}
\\
&\quad
+
\left(
\frac1N \sum_{\ell=1}^N
\norm{D_r^{k,\beta} X_\rho^{\ell,N}}
\right)
\left(
\frac1N \sum_{\ell=1}^N
\norm{D_s^{j,\alpha} X_\rho^{\ell,N}}
\right)
\Biggr).
\end{align*}
By Minkowski's inequality and H\"older's inequality,
\begin{align*}
&\E\left(
\norm{R_{r,s}^{i,j,k,\alpha,\beta}(\rho)}^4
+
\norm{S_{r,s}^{i,j,k,\alpha,\beta}(\rho)}^4
\right)^{1/4}
\le
C \Biggl(
\E\left(
\norm{D_r^{k,\beta} X_\rho^{i,N}}^8
\right)^{1/8}
\E\left(
\norm{D_s^{j,\alpha} X_\rho^{i,N}}^8
\right)^{1/8}
\\
&
+
\E\left(
\norm{D_s^{j,\alpha} X_\rho^{i,N}}^8
\right)^{1/8}
\frac1N \sum_{\ell=1}^N
\E\left(
\norm{D_r^{k,\beta} X_\rho^{\ell,N}}^8
\right)^{1/8}
+
\E\left(
\norm{D_r^{k,\beta} X_\rho^{i,N}}^8
\right)^{1/8}
\frac1N \sum_{\ell=1}^N
\E\left(
\norm{D_s^{j,\alpha} X_\rho^{\ell,N}}^8
\right)^{1/8}
\\
&
+
\frac1N \sum_{\ell=1}^N
\E\left(
\norm{D_r^{k,\beta} X_\rho^{\ell,N}}^8
\right)^{1/8}
\E\left(
\norm{D_s^{j,\alpha} X_\rho^{\ell,N}}^8
\right)^{1/8}
\\
&
+
\left(
\frac1N \sum_{\ell=1}^N
\E\left(
\norm{D_r^{k,\beta} X_\rho^{\ell,N}}^8
\right)^{1/8}
\right)
\left(
\frac1N \sum_{\ell=1}^N
\E\left(
\norm{D_s^{j,\alpha} X_\rho^{\ell,N}}^8
\right)^{1/8}
\right)
\Biggr).
\end{align*}
Applying Proposition~\ref{prop:firstbounds-uniform} with $ p = 8 $
together with the fact that since $ r \le s \le \rho $, $e^{-\hat{\omega} (\rho-r)} e^{-\hat{\omega} (\rho-s)}
\le
e^{-\hat{\omega} (\rho-r)}$ yields \eqref{eq:second-RS-bound-uniform-local}.
\end{proof}

\begin{proposition}[Uniform-in-time second Malliavin derivative bounds]
\label{prop:secondbounds-uniform}
Assume that Assumptions \ref{ass:smooth},
\ref{ass:second-malliavin-uniform}, \ref{ass:coeff-level-new}(i)(ii),
and \ref{ass:first-malliavin-uniform} hold. Then, there exists a constant
$ C > 0 $ such that, for all $ N \ge 1 $, all $ 0 \le r,s \le u $, all
$ i,j,k \in \left\{ 1,\dots,N \right\} $, and all
$ \alpha,\beta \in \left\{ 1,\dots,m \right\} $,
\begin{equation}
\label{eq:secondbounds-uniform}
\E\left(
\norm{D_r^{k,\beta} D_s^{j,\alpha} X_u^{i,N}}^4
\right)^{1/4}
\le
C e^{-\hat{\omega}\left( u-\min\left\{ r,s \right\} \right)}
\left(
\mathbf{1}_{\left\{ i=j=k \right\}}
+
\frac{1}{N}
\left(
\mathbf{1}_{\left\{ i=j \right\}}
+
\mathbf{1}_{\left\{ i=k \right\}}
+
\mathbf{1}_{\left\{ j=k \right\}}
\right)
+
\frac{1}{N^2}
\right).
\end{equation}
\end{proposition}

\begin{proof}
By symmetry of the Malliavin derivatives, it is enough to consider the case
$ r \le s \le u $. Fix such $ r,s,u $, and fix
$ i,j,k \in \left\{ 1,\dots,N \right\} $ and
$ \alpha,\beta \in \left\{ 1,\dots,m \right\} $. Define, for every $\rho \in [s,u]$,
\begin{equation*}
Y_\rho^i
=
D_r^{k,\beta} D_s^{j,\alpha} X_\rho^{i,N},
\quad
v_i(\rho)
=
\E\left(
\norm{Y_\rho^i}^4
\right)^{1/4}\quad\mbox{and}\quad
\bar v(\rho)
=
\frac1N \sum_{\ell=1}^N v_\ell(\rho).
\end{equation*}
Also set
\begin{equation*}
K_i
=
\mathbf{1}_{\left\{ i=j=k \right\}}
+
\frac{1}{N}
\left(
\mathbf{1}_{\left\{ i=j \right\}}
+
\mathbf{1}_{\left\{ i=k \right\}}
+
\mathbf{1}_{\left\{ j=k \right\}}
\right)
+
\frac{1}{N^2}.
\end{equation*}
By Lemma~\ref{lem:secondSDE}, the process $ Y^i $ satisfies, on $ [s,u] $,
\begin{align*}
Y_u^i
&=
\mathcal I_{r,s}^{i,j,k,\alpha,\beta}
+
\int_s^u
\left(
\partial_x b\left( X_\rho^{i,N},\mu_\rho^N \right) Y_\rho^i
+
G_{r,s}^{i,j,k,\alpha,\beta}(\rho)
+
R_{r,s}^{i,j,k,\alpha,\beta}(\rho)
\right)
d\rho
\\
&\quad
+
\int_s^u
\left(
\partial_x \sigma\left( X_\rho^{i,N},\mu_\rho^N \right) Y_\rho^i
+
H_{r,s}^{i,j,k,\alpha,\beta}(\rho)
+
S_{r,s}^{i,j,k,\alpha,\beta}(\rho)
\right)
dB_\rho^i .
\end{align*}
Applying Lemma~\ref{lem:linear-dissipative-estimate} with
$ p = 4 $, $ \kappa = \kappa_4 $, and $ M = M_\sigma $, we obtain
\begin{align*}
v_i(u)
&\le
e^{-\kappa_4 (u-s)/4}
\E\left(
\norm{\mathcal I_{r,s}^{i,j,k,\alpha,\beta}
}^4
\right)^{1/4}
\nonumber\\
&\quad
+
\Lambda_4
\left(
\int_s^u
e^{-\kappa_4 (u-\rho)}
\E\left(
\abs{
G_{r,s}^{i,j,k,\alpha,\beta}(\rho)
+
R_{r,s}^{i,j,k,\alpha,\beta}(\rho)
}^4
+
\norm{
H_{r,s}^{i,j,k,\alpha,\beta}(\rho)
+
S_{r,s}^{i,j,k,\alpha,\beta}(\rho)
}^4
\right)
d\rho
\right)^{1/4}.
\end{align*}
By Minkowski's inequality in
$ L^4\left( [s,u], e^{-\kappa_4 (u-\rho)} d\rho \right) $, it follows that
\begin{align*}
v_i(u)
&\le
e^{-\kappa_4 (u-s)/4}
\E\left(
\norm{\mathcal I_{r,s}^{i,j,k,\alpha,\beta}
}^4
\right)^{1/4}
\nonumber\\
&\quad
+
\Lambda_4
\left(
\int_s^u
e^{-\kappa_4 (u-\rho)}
\E\left(
\abs{G_{r,s}^{i,j,k,\alpha,\beta}(\rho)}^4
+
\norm{H_{r,s}^{i,j,k,\alpha,\beta}(\rho)}^4
\right)
d\rho
\right)^{1/4}
\nonumber\\
&\quad
+
\Lambda_4
\left(
\int_s^u
e^{-\kappa_4 (u-\rho)}
\E\left(
\norm{R_{r,s}^{i,j,k,\alpha,\beta}(\rho)}^4
+
\norm{S_{r,s}^{i,j,k,\alpha,\beta}(\rho)}^4
\right)
d\rho
\right)^{1/4}.
\end{align*}
By \eqref{eq:second-initial-bound-uniform-local},
\begin{equation*}
e^{-\kappa_4 (u-s)/4}
\E\left(
\norm{\mathcal I_{r,s}^{i,j,k,\alpha,\beta}
}^4
\right)^{1/4}
\le
C e^{-\hat{\omega} (u-r)} K_i,
\end{equation*}
because $ \hat{\omega} \le \kappa_4/8 \le \kappa_4/4 $. By \eqref{eq:second-GH-bound-uniform-local},
\begin{align*}
\Lambda_4
\left(
\int_s^u
e^{-\kappa_4 (u-\rho)}
\E\left(
\abs{G_{r,s}^{i,j,k,\alpha,\beta}(\rho)}^4
+
\norm{H_{r,s}^{i,j,k,\alpha,\beta}(\rho)}^4
\right)
d\rho
\right)^{1/4}
&\le
2^{1/4}\Lambda_4 \gamma
\left(
\int_s^u
e^{-\kappa_4 (u-\rho)}
\bar v(\rho)^4
d\rho
\right)^{1/4}.
\end{align*}
By \eqref{eq:second-RS-bound-uniform-local},
\begin{align*}
\Lambda_4
\left(
\int_s^u
e^{-\kappa_4 (u-\rho)}
\E\left(
\norm{R_{r,s}^{i,j,k,\alpha,\beta}(\rho)}^4
+
\norm{S_{r,s}^{i,j,k,\alpha,\beta}(\rho)}^4
\right)
d\rho
\right)^{1/4}
&\le
C K_i
\left(
\int_s^u
e^{-\kappa_4 (u-\rho)} e^{-4\omega (\rho-r)}
d\rho
\right)^{1/4}
\nonumber\\
&\le
C e^{-\hat{\omega} (u-r)} K_i,
\end{align*}
because $ \kappa_4 - 4\omega \ge \kappa_4/2 > 0 $. Combining the last
four estimates, we obtain
\begin{equation}
\label{eq:proof-prop-second-uniform-3}
v_i(u)
\le
C e^{-\hat{\omega} (u-r)} K_i
+
2^{1/4}\Lambda_4 \gamma
\left(
\int_s^u
e^{-\kappa_4 (u-\rho)}
\bar v(\rho)^4
d\rho
\right)^{1/4}.
\end{equation}
For $ v \in [s,u] $, define $M_v
=
\sup_{\rho \in [s,v]}
e^{\hat{\omega} (\rho-r)} \bar v(\rho)$. Then, for every $\rho \in [s,u]$, $\bar v(\rho)
\le
e^{-\hat{\omega} (\rho-r)} M_u$. Substituting into \eqref{eq:proof-prop-second-uniform-3}, we get
\begin{align*}
v_i(u)
&\le
C e^{-\hat{\omega} (u-r)} K_i
+
2^{1/4}\Lambda_4 \gamma
\left(
\int_s^u
e^{-\kappa_4 (u-\rho)} e^{-4\hat{\omega} (\rho-r)}
d\rho
\right)^{1/4}
M_u
\nonumber\\
&=
C e^{-\hat{\omega} (u-r)} K_i
+
2^{1/4}\Lambda_4 \gamma
e^{-\hat{\omega} (u-r)}
\left(
\int_s^u
e^{-(\kappa_4 - 4\hat{\omega})(u-\rho)}
d\rho
\right)^{1/4}
M_u
\nonumber\\
&\le
C e^{-\hat{\omega} (u-r)} K_i
+
2^{1/4}\Lambda_4 \gamma
\left(
\frac{2}{\kappa_4}
\right)^{1/4}
e^{-\hat{\omega} (u-r)} M_u.
\end{align*}
Set
\begin{equation*}
\eta_4
=
2^{1/4}\Lambda_4 \gamma
\left(
\frac{2}{\kappa_4}
\right)^{1/4}.
\end{equation*}
By Assumption~\ref{ass:first-malliavin-uniform} (with $p=4$ and $p=8$), we have $\eta_4 <
1$. Hence,
\begin{equation}
\label{eq:proof-prop-second-uniform-5}
v_i(u)
\le
C e^{-\hat{\omega} (u-r)} K_i
+
\eta_4 e^{-\hat{\omega} (u-r)} M_u.
\end{equation}
We now average over $ i $. Since
\begin{align*}
\frac1N \sum_{i=1}^N K_i
&=
\frac1N \sum_{i=1}^N \mathbf{1}_{\left\{ i=j=k \right\}}
+
\frac1{N^2}\sum_{i=1}^N
\left(
\mathbf{1}_{\left\{ i=j \right\}}
+
\mathbf{1}_{\left\{ i=k \right\}}
+
\mathbf{1}_{\left\{ j=k \right\}}
\right)
+
\frac1N \sum_{i=1}^N \frac1{N^2}
\\
&\le
C\left(
\frac{\mathbf{1}_{\left\{ j=k \right\}}}{N}
+
\frac1{N^2}
\right),
\end{align*}
averaging \eqref{eq:proof-prop-second-uniform-5} yields
\begin{equation*}
\bar v(u)
\le
C e^{-\hat{\omega} (u-r)}
\left(
\frac{\mathbf{1}_{\left\{ j=k \right\}}}{N}
+
\frac1{N^2}
\right)
+
\eta_4 e^{-\hat{\omega} (u-r)} M_u.
\end{equation*}
Taking the supremum over $ u \in [s,v] $, we obtain
\begin{equation*}
M_v
\le
C\left(
\frac{\mathbf{1}_{\left\{ j=k \right\}}}{N}
+
\frac1{N^2}
\right)
+
\eta_4 M_v.
\end{equation*}
Since $ \eta_4 < 1 $, it follows that for every $v \in [s,u]$,
\begin{equation}
\label{eq:proof-prop-second-uniform-6}
M_v
\le
C\left(
\frac{\mathbf{1}_{\left\{ j=k \right\}}}{N}
+
\frac1{N^2}
\right).
\end{equation}
Substituting \eqref{eq:proof-prop-second-uniform-6} back into
\eqref{eq:proof-prop-second-uniform-5}, we find
\begin{align*}
v_i(u)
&\le
C e^{-\hat{\omega} (u-r)}
\left[
K_i
+
\frac{\mathbf{1}_{\left\{ j=k \right\}}}{N}
+
\frac1{N^2}
\right]
\\
&\le
C e^{-\hat{\omega} (u-r)}
\left(
\mathbf{1}_{\left\{ i=j=k \right\}}
+
\frac{1}{N}
\left(
\mathbf{1}_{\left\{ i=j \right\}}
+
\mathbf{1}_{\left\{ i=k \right\}}
+
\mathbf{1}_{\left\{ j=k \right\}}
\right)
+
\frac1{N^2}
\right).
\end{align*}
This proves \eqref{eq:secondbounds-uniform} in the case $ r \le s \le u $.
The general case follows by symmetry in $ r $ and $ s $.
\end{proof}

\subsection{Proof of Proposition \ref{prop:VidottoFunctional}}\label{S:UniformControlVidotto}

In this section, we provide the proof of Proposition \ref{prop:VidottoFunctional}. Lemma \ref{lem:Fbounds} that we present next plays an instrumental role in that proof, providing uniform-in-time bounds for the Malliavin derivatives of $F_t^N\left(\varphi\right)$. To obtain these bounds, we make use of the uniform-in-time bounds for the first- and second-order Malliavin derivatives of the flow as obtained in Proposition~\ref{prop:firstbounds-uniform} and
Proposition~\ref{prop:secondbounds-uniform}.
\begin{lemma}[Uniform bounds for the Malliavin derivatives of $F_t^N\left(\varphi\right)$]
\label{lem:Fbounds}
Assume that Assumptions \ref{ass:smooth},
\ref{ass:second-malliavin-uniform}, \ref{ass:coeff-level-new}(i)(ii),
and \ref{ass:first-malliavin-uniform} hold, let $ \hat{\omega} $ be given by
\eqref{eq:def-omega-second-malliavin}, and let
$\varphi\in C_b^2\left(\R^d\right)$.
Then, there exists a constant $C>0$, depending only on
$\omega$, $m$, and $\varphi$, such that for all $N\ge 1$, all
$0 \le s,r \le t$, all $j,k\in\left\{1,\dots,N\right\}$ and all
$\alpha,\beta\in\left\{1,\dots,m\right\}$,
\begin{equation}
\label{eq:DFboundfinal}
\E\left(
\abs{D_s^{j,\alpha}F_t^N\left(\varphi\right)}^4
\right)^{\frac{1}{4}}
\le
\frac{C}{\sqrt{N}} e^{-\hat{\omega}\left( t-s \right)},
\end{equation}
and
\begin{equation}
\label{eq:D2Fboundfinal}
\E\left(
\abs{D_r^{k,\beta}D_s^{j,\alpha}F_t^N\left(\varphi\right)}^4
\right)^{\frac{1}{4}}
\le
\frac{C}{\sqrt{N}}
e^{-\hat{\omega}\left( t-\min\left\{ r,s \right\} \right)}
\left(
\mathbf{1}_{\left\{j=k\right\}}+\frac1N
\right).
\end{equation}
\end{lemma}

\begin{proof}
We begin with the first derivative. By the Malliavin chain rule,
\begin{equation*}
D_s^{j,\alpha}F_t^N\left(\varphi\right)
=
\frac1{\sqrt{N}}\sum_{i=1}^N
\nabla\varphi\left(X_t^{i,N}\right)\cdot D_s^{j,\alpha}X_t^{i,N}.
\end{equation*}
Since $\nabla\varphi$ is bounded, Minkowski's inequality and
Proposition~\ref{prop:firstbounds-uniform} with $p=4$ give
\begin{align*}
\E\left(\abs{D_s^{j,\alpha}F_t^N\left(\varphi\right)}^4\right)^{\frac{1}{4}}
&\le
\frac{\norm{\nabla\varphi}_\infty}{\sqrt{N}}
\sum_{i=1}^N
\E\left(\norm{D_s^{j,\alpha}X_t^{i,N}}^4\right)^{\frac{1}{4}}
\\
&\le
\frac{C}{\sqrt{N}} e^{-\hat{\omega}\left( t-s \right)}
\sum_{i=1}^N
\left(
\mathbf{1}_{\left\{i=j\right\}}+\frac1N\mathbf{1}_{\left\{i\ne j\right\}}
\right)
\le
\frac{C}{\sqrt{N}} e^{-\hat{\omega}\left( t-s \right)},
\end{align*}
which proves \eqref{eq:DFboundfinal}. For the second derivative, by commutativity of Malliavin derivatives,
\begin{align*}
D_r^{k,\beta}D_s^{j,\alpha}F_t^N\left(\varphi\right)
&=
\frac1{\sqrt{N}}\sum_{i=1}^N
\nabla\varphi\left(X_t^{i,N}\right)\cdot
D_r^{k,\beta}D_s^{j,\alpha}X_t^{i,N}
\\
&\quad
+
\frac1{\sqrt{N}}\sum_{i=1}^N
\left(D_r^{k,\beta}X_t^{i,N}\right)^\top
\nabla^2\varphi\left(X_t^{i,N}\right)
D_s^{j,\alpha}X_t^{i,N}.
\end{align*}
Hence, using boundedness of $\nabla\varphi$ and $\nabla^2\varphi$,
Minkowski's inequality, H\"older's inequality,
Proposition~\ref{prop:firstbounds-uniform} with $p=8$, and
Proposition~\ref{prop:secondbounds-uniform}, we obtain
\begin{align*}
&\E\left(
\abs{D_r^{k,\beta}D_s^{j,\alpha}F_t^N\left(\varphi\right)}^4
\right)^{\frac{1}{4}}
\\
&\le
\frac{C}{\sqrt{N}}
\sum_{i=1}^N
\E\left(
\norm{D_r^{k,\beta}D_s^{j,\alpha}X_t^{i,N}}^4
\right)^{\frac{1}{4}}
+
\frac{C}{\sqrt{N}}
\sum_{i=1}^N
\E\left(
\norm{D_r^{k,\beta}X_t^{i,N}}^8
\right)^{\frac{1}{8}}
\E\left(
\norm{D_s^{j,\alpha}X_t^{i,N}}^8
\right)^{\frac{1}{8}}
\\
&\le
\frac{C}{\sqrt{N}}
e^{-\hat{\omega}\left( t-\min\left\{ r,s \right\} \right)}
\sum_{i=1}^N
\left(
\mathbf{1}_{\left\{i=j=k\right\}}
+\frac1N
\left(
\mathbf{1}_{\left\{i=j\right\}}
+\mathbf{1}_{\left\{i=k\right\}}
+\mathbf{1}_{\left\{j=k\right\}}
\right)
+\frac1{N^2}
\right)
\\
&\quad
+
\frac{C}{\sqrt{N}}
e^{-\hat{\omega}\left( t-r \right)}e^{-\hat{\omega}\left( t-s \right)}
\sum_{i=1}^N
\left(
\mathbf{1}_{\left\{i=k\right\}}+\frac1N
\right)
\left(
\mathbf{1}_{\left\{i=j\right\}}+\frac1N
\right).
\end{align*}
Since
\begin{equation*}
\sum_{i=1}^N
\left(
\mathbf{1}_{\left\{i=j=k\right\}}
+\frac1N
\left(
\mathbf{1}_{\left\{i=j\right\}}
+\mathbf{1}_{\left\{i=k\right\}}
+\mathbf{1}_{\left\{j=k\right\}}
\right)
+\frac1{N^2}
\right)
\le
C\left( \mathbf{1}_{\left\{j=k\right\}}+\frac1N \right)
\end{equation*}
and
\begin{equation*}
\sum_{i=1}^N
\left(
\mathbf{1}_{\left\{i=k\right\}}+\frac1N
\right)
\left(
\mathbf{1}_{\left\{i=j\right\}}+\frac1N
\right)
\le
C\left( \mathbf{1}_{\left\{j=k\right\}}+\frac1N \right),
\end{equation*}
while $e^{-\hat{\omega}\left( t-r \right)}e^{-\hat{\omega}\left( t-s \right)}
\le
e^{-\hat{\omega}\left( t-\min\left\{ r,s \right\} \right)}$, we obtain \eqref{eq:D2Fboundfinal}.
\end{proof}

\begin{proof}[Proof of Proposition \ref{prop:VidottoFunctional}]
Fix $ t \ge 0 $, $x=\left(s,j,\alpha\right)\in A_t$ and
$y=\left(r,k,\beta\right)\in A_t$.
By definition of the first contraction on
$\mathfrak H_t=L^2\left(A_t,\lambda_t\right)$,
\begin{equation*}
\left(
D^2F_t^N\left(\varphi\right)\otimes_1
D^2F_t^N\left(\varphi\right)
\right)\left(x,y\right)
=
\int_0^t\sum_{\ell=1}^N\sum_{\gamma=1}^m
D_u^{\ell,\gamma}D_s^{j,\alpha}F_t^N\left(\varphi\right)
D_u^{\ell,\gamma}D_r^{k,\beta}F_t^N\left(\varphi\right)
du.
\end{equation*}
Applying Minkowski's inequality and H\"older's inequality gives
\begin{align*}
&\sqrt{
\E\left(
\left(
D^2F_t^N\left(\varphi\right)\otimes_1
D^2F_t^N\left(\varphi\right)
\right)\left(x,y\right)^2
\right)
}
\\
&\qquad\qquad\qquad\qquad\qquad \le
\int_0^t\sum_{\ell=1}^N\sum_{\gamma=1}^m
\E\left(
\abs{
D_u^{\ell,\gamma}D_s^{j,\alpha}F_t^N\left(\varphi\right)
}^4
\right)^{\frac{1}{4}}
\E\left(
\abs{
D_u^{\ell,\gamma}D_r^{k,\beta}F_t^N\left(\varphi\right)
}^4
\right)^{\frac{1}{4}}
du.
\end{align*}
Using \eqref{eq:D2Fboundfinal}, we obtain
\begin{align*}
&\sqrt{
\E\left(
\left(
D^2F_t^N\left(\varphi\right)\otimes_1
D^2F_t^N\left(\varphi\right)
\right)\left(x,y\right)^2
\right)
}
\\
&\qquad\qquad\qquad \le
\frac{C}{N}
\int_0^t
e^{-\hat{\omega}\left( t-\min\left\{ u,s \right\} \right)}
e^{-\hat{\omega}\left( t-\min\left\{ u,r \right\} \right)}
du
\sum_{\ell=1}^N
\left(
\mathbf{1}_{\left\{\ell=j\right\}}+\frac1N
\right)
\left(
\mathbf{1}_{\left\{\ell=k\right\}}+\frac1N
\right)
\\
&\qquad\qquad\qquad \le
\frac{C}{N}
\left(
\mathbf{1}_{\left\{j=k\right\}}+\frac1N
\right)
\int_0^t
e^{-\hat{\omega}\left( t-\min\left\{ u,s \right\} \right)}
e^{-\hat{\omega}\left( t-\min\left\{ u,r \right\} \right)}
du.
\end{align*}
Assume without loss of generality that $ r \le s $. Splitting the
integral over $ \left[ 0,r \right] $, $ \left[ r,s \right] $, and $
\left[ s,t \right] $, we get
\begin{align*}
\int_0^t
e^{-\hat{\omega}\left( t-\min\left\{ u,s \right\} \right)}
e^{-\hat{\omega}\left( t-\min\left\{ u,r \right\} \right)}
du
&=
\int_0^r e^{-2\hat{\omega}\left( t-u \right)} du
+
\int_r^s e^{-\hat{\omega}\left( t-u \right)}e^{-\hat{\omega}\left( t-r \right)}
  du\\
  &\quad
+
\int_s^t e^{-\hat{\omega}\left( t-s \right)}e^{-\hat{\omega}\left( t-r \right)} du
\\
&\le
C e^{-2\hat{\omega}\left( t-r \right)}
+
C e^{-\hat{\omega}\left( 2t-s-r \right)}
+
\left( t-s \right)e^{-\hat{\omega}\left( 2t-s-r \right)}
\\
&\le
C\left( 1+t-s \right)e^{-\hat{\omega}\left( 2t-s-r \right)}.
\end{align*}
Since the last bound is symmetric in $ r $ and $ s $, we may write
\begin{equation*}
\int_0^t
e^{-\hat{\omega}\left( t-\min\left\{ u,s \right\} \right)}
e^{-\hat{\omega}\left( t-\min\left\{ u,r \right\} \right)}
du
\le
C\left( 1+\min\left\{ t-s,t-r \right\} \right)
e^{-\hat{\omega}\left( 2t-s-r \right)}.
\end{equation*}
Therefore,
\begin{equation*}
\sqrt{
\E\left(
\left(
D^2F_t^N\left(\varphi\right)\otimes_1
D^2F_t^N\left(\varphi\right)
\right)\left(x,y\right)^2
\right)
}
\le
\frac{C}{N}
\left(
\mathbf{1}_{\left\{j=k\right\}}+\frac1N
\right)
\left( 1+\min\left\{ t-s,t-r \right\} \right)
e^{-\hat{\omega}\left( 2t-s-r \right)}.
\end{equation*}
On the other hand, \eqref{eq:DFboundfinal} gives
\begin{align*}
\sqrt{
\E\left(
DF_t^N\left(\varphi\right)\left(x\right)^2
DF_t^N\left(\varphi\right)\left(y\right)^2
\right)
}
&\le
\E\left(
\abs{DF_t^N\left(\varphi\right)\left(x\right)}^4
\right)^{\frac{1}{4}}
\E\left(
\abs{DF_t^N\left(\varphi\right)\left(y\right)}^4
\right)^{\frac{1}{4}}
\le
\frac{C}{N}e^{-\hat{\omega}\left( 2t-s-r \right)}.
\end{align*}
Multiplying the last two bounds yields
\begin{align*}
&\sqrt{
\E\left(
\left(
D^2F_t^N\left(\varphi\right)\otimes_1
D^2F_t^N\left(\varphi\right)
\right)\left(x,y\right)^2
\right)
}
\sqrt{
\E\left(
DF_t^N\left(\varphi\right)\left(x\right)^2
DF_t^N\left(\varphi\right)\left(y\right)^2
\right)
}
\\
&\qquad\qquad\qquad\qquad\qquad\qquad\qquad\quad \le
\frac{C}{N^2}
\left(
\mathbf{1}_{\left\{j=k\right\}}+\frac1N
\right)
\left( 1+\min\left\{ t-s,t-r \right\} \right)
e^{-2\hat{\omega}\left( 2t-s-r \right)}.
\end{align*}
Integrating over $ x $ and $ y $, and using the change of variables
$ a=t-s $ and $ b=t-r $, we obtain
\begin{align*}
\Delta_{N,t}
&\le
\frac{C}{N^2}
\sum_{j=1}^N\sum_{k=1}^N
\left(
\mathbf{1}_{\left\{j=k\right\}}+\frac1N
\right)
\int_0^t\int_0^t
\left( 1+\min\left\{ a,b \right\} \right)
e^{-2\hat{\omega}\left( a+b \right)}
da db
\\
&\le
\frac{C}{N^2}
\sum_{j=1}^N\sum_{k=1}^N
\left(
\mathbf{1}_{\left\{j=k\right\}}+\frac1N
\right)
\int_0^\infty\int_0^\infty
\left( 1+\min\left\{ a,b \right\} \right)
e^{-2\hat{\omega}\left( a+b \right)}
da db.
\end{align*}
The double integral above is finite and depends only on $ \hat{\omega} $. Since
\begin{equation*}
\sum_{j=1}^N\sum_{k=1}^N \mathbf{1}_{\left\{j=k\right\}} = N
\qquad \text{and} \qquad
\sum_{j=1}^N\sum_{k=1}^N \frac1N = N,
\end{equation*}
we conclude that
\begin{equation*}
\Delta_{N,t}\le \frac{C}{N},
\end{equation*}
with $ C $ independent of $ t $. Taking the supremum over $ t \ge 0 $ proves
\eqref{eq:GammaBound}.
\end{proof}

\bibliographystyle{plain}

\end{document}